\documentclass[12pt]{amsart}
\usepackage[dvips]{epsfig}
\usepackage{fancyhdr}
\usepackage{colortbl}
\usepackage{afterpage}
\usepackage{times}
\usepackage{chngcntr}
\usepackage{graphicx}
\usepackage{latexsym}
\usepackage{amsfonts}
\usepackage{amscd}
\usepackage{amssymb}
\usepackage{amsmath}
\usepackage{amsthm}
\usepackage{bm}
\usepackage{psfrag}
\usepackage{url}
\usepackage{subfig}
\usepackage{enumitem}
\usepackage{mathdots}
\usepackage{hyperref}
\usepackage[margin=1.25in]{geometry}

\counterwithout{figure}{section}

\newtheorem{thm}{Theorem}[section]
\newtheorem{cor}[thm]{Corollary}

\newtheorem{lem}[thm]{Lemma}

\theoremstyle{definition}
\newtheorem{defn}[thm]{Definition}
\theoremstyle{remark}
\newtheorem{rem}[thm]{Remark}

\newtheorem{ex}[thm]{Example}

\numberwithin{equation}{section}
\newtheoremstyle{dotless}{}{}{}{}{\bfseries}{}{ }{}
\theoremstyle{dotless}
\newtheorem*{repthm}{Theorem}
\newtheorem*{repcor}{Corollary}

\def\N{\mathbb{N}}
\def\R{\mathbb{R}}

\def\Z{\mathbb{Z}}

\def\MCG{{\mathcal MCG}}

\newcommand{\inj}{\rightarrowtail}
\newcommand{\surj}{\twoheadrightarrow}
\newcommand{\Ra}{\Rightarrow}
\newcommand{\La}{\Leftarrow}
\newcommand{\ra}{\rightarrow}

\newcommand{\co}{\colon\thinspace}
\newcommand{\uas}{{$\omega$--a.s.\ }}
\newcommand{\e}[1]{\mathcal{E}(#1)}
\newcommand{\nse}[1]{\mathcal{NE}(#1)}
\newcommand{\se}[1]{\mathcal{SE}(#1)}
\newcommand{\eultra}[1]{\mathcal{E}^{\omega}(#1)}
\newcommand{\nseultra}[1]{\mathcal{NE}^{\omega}(#1)}
\newcommand{\seultra}[1]{\mathcal{SE}^{\omega}(#1)}

\newcommand{\x}[1]{\mathcal{X}(#1)}
\newcommand{\q}[1]{\mathcal{Q}(#1)}
\newcommand{\qbar}[1]{\mathcal{Q}_{\omega}(\overline{#1})}

\newcommand{\xbar}[1]{\mathcal{X}_{\omega}(\overline{#1})}

\newcommand{\norm}[1] {\left\lVert #1 \right\lVert}

 \def\co{\colon\thinspace}

\begin{document}
\title{The Asymptotic Cone of Teichm\"uller Space: Thickness and Divergence}
\author{Harold Sultan}
\address{Department  of Mathematics\\Brandeis University\\
Waltham\\MA 02453}
\email{HSultan@brandeis.edu}
\date{\today}
\maketitle

\begin{abstract}
We study the Asymptotic Cone of Teichm\"uller space equipped with the Weil-Petersson metric.  In particular, we provide a characterization of the canonical finest pieces in the tree-graded structure of the asymptotic cone of Teichm\"uller space along the same lines as a similar characterization for right angled Artin groups in \cite{behrstockcharney} and for mapping class groups in \cite{bkmm}.  As a corollary of the characterization, we complete the thickness classification of Teichm\"uller spaces for all surfaces of finite type, thereby answering questions of Behrstock-Dru\c tu \cite{behrstockdrutu}, Behrstock-Dru\c tu-Mosher \cite{bdm}, and Brock-Masur \cite{brockmasur}.  In particular, we prove that Teichm\"uller space of the genus two surface with one boundary component (or puncture) can be uniquely characterized in the following two senses: it is thick of order two, and it has superquadratic yet at most cubic divergence.  In addition, we characterize strongly contracting quasi-geodesics in Teichm\"uller space, generalizing results of Brock-Masur-Minsky \cite{bmm2}.  As a tool, we develop a complex of separating multicurves, which may be of independent interest. 
\end{abstract}
 \tableofcontents
\section{Introduction}
\label{chap:introduction}
\label{sec:overview}

In the setting of spaces of non-positive curvature, Euclidean and hyperbolic space represent the two classically well understood extremes. In this paper we will see that the geometries of Teichm\"uller Spaces for various surfaces provide natural examples of non-positively curves spaces which nontrivially interpolate between these two ends of the spectrum of non-positively curved spaces.  In particular, the geometry of Teichm\"uller spaces includes on the one hand examples of hyperbolic and strongly relatively hyperbolic metric spaces, and on the other hand thick of order one and thick of order two metric spaces.  The example of a Teichm\"uller space which is thick of order two is a novelty of this paper.  In a similar vein, we will see that the divergence function of Teichm\"uller spaces includes examples of spaces with quadratic divergence, superquadratic yet at most cubic divergence, exponential divergence, and infinite divergence.  Again, the example of a Teichm\"uller space which has superquadratic yet at most cubic divergence is a novelty of this paper.
  
The above referenced notion of thickness, developed in \cite{bdm} and further explored in \cite{behrstockdrutu}, is aptly named as it stands in stark contrast to relative hyperbolicity.  In fact, if a space is thick of any finite order than it is not strongly relatively hyperbolic, \cite{bdm}.  Informally the \emph{order of thickness of a space} should thought as a precise means of interpolating between product spaces, which are thick of order zero, and (relatively)-hyperbolic spaces, which are not thick of any finite order, or are thick of order infinity.  More specifically, thickness will be defined inductively in Section \ref{subsec:background}.  


For $S$ a surface of finite type, \emph{Teichm\"uller space}, $\mathcal{T}(S),$ is a classical space which parameterizes isotopy classes of hyperbolic structures on $S.$  In the literature there are various natural metrics with which Teichm\"uller space can be equipped.  Hereinafter, we always consider $\mathcal{T}(S)$ with the Weil-Petersson metric.  The Weil-Petersson metric on $\mathcal{T}(S)$ is a complex analytically defined Riemannian metric of variable non-positive curvature.  While the space is not complete, its completion, $\mathcal{\overline{T}}(S),$ obtained by augmenting Teichm\"uller spaces of lower complexity surfaces corresponding to limit points in the space with pinched curves, is a CAT(0) metric space \cite{wolpertcat,yamada}.  The large scale geometry of Teichm\"uller space has been an object of recent interest, especially within the circles of ideas surrounding Thurston's Ending Lamination Conjecture.  In this context, the pants complex, $\mathcal{P}(S),$ a combinatorial complex associated to a hyperbolic surface $S,$ becomes relevant.  Specifically, by a groundbreaking theorem of Brock \cite{brock}, $\mathcal{P}(S)$ is quasi-isometric to $\mathcal{T}(S).$  Accordingly, in order to study large scale geometric properties of Teichm\"uller space, it suffices to study the pants complex of a surface.  For instance, significant recent results of Behrstock \cite{behrstock}, Behrstock-Minsky \cite{behrstockminsky}, Brock-Farb \cite{brockfarb}, Brock-Masur \cite{brockmasur}, and Brock-Masur-Minsky \cite{bmm1,bmm2}, among others, can be viewed from this perspective.  Similarly, all of the results of this paper regarding the coarse structure of the pants complex should be interpreted as coarse results regarding Teichm\"uller space.  


In recent years, study of asymptotic cones has proven extremely fruitful in considering the coarse geometry of groups and spaces.  See for instance \cite{bds, drutu, drutusapir}.  One aspect in common to the aforementioned studies of asymptotic cones is interest in \emph{cut-points}, namely single points whose removal disconnects the asymptotic cone.  The general theme is that cut-points in asymptotic cones correspond to a weak form of hyperbolicity in the underlying space.  One of the highlights of the paper is a characterization of when two points in the asymptotic cone of Teichm\"uller space are separated by a cut-point, see Theorem \ref{thm:pieces}.    

On the one hand, it is shown in \cite{behrstock} that in the asymptotic cone of Teichm\"uller space, every point is a global cut-point.  On the other hand, for high enough complexity surfaces, Teichm\"uller space has natural quasi-isometrically embedded flats, or \emph{quasi-flats}, \cite{behrstockminsky, brockfarb, mm2}.  In turn, this implies the existence of naturally embedded flats in the asymptotic cone and hence the existence of nontrivial subsets of the asymptotic cone without cut-points.  Putting things together, for high enough complexity surfaces, the asymptotic cone of Teichm\"uller space is a \emph{tree-graded space}.  In such a setting, there are canonically defined \emph{finest pieces} of the tree-graded structure, which are defined to be maximal subsets of the asymptotic cone subject to the condition that no two points in a finest piece can be separated by the removal of a point.  A highlight of this paper is the following theorem that characterizes the finest pieces in tree-graded structure of the asymptotic cone of Teichm\"uller space.

\begin{repthm} $\textbf{\ref{thm:pieces}.}$
\emph{Let $S=S_{g,n},$ and let $\mathcal{P}_{\omega}(S)$ be any asymptotic cone of $\mathcal{P}(S).$  Then $\forall a_{\omega},b_{\omega}  \in\mathcal{P}{\omega}(S),$ the following are equivalent:
\begin{enumerate}
\item No point separates $a_{\omega}$ and $b_{\omega},$ or equivalently $a_{\omega}$ and $b_{\omega}$ are in the same canonical finest piece, and
\item In any neighborhood of $a_{\omega}, b_{\omega},$ respectively, there exists $a'_{\omega},b'_{\omega},$ with representative sequences $(a'_{n})$,$(b'_{n})$, such that $\lim_{\omega}d_{\mathbb{S}(S)}(a'_{n},b'_{n}) < \infty.$  
\end{enumerate}}
\end{repthm} 

The characterization of finest pieces in Theorem \ref{thm:pieces} is given in terms of the complex of separating multicurves $\mathbb{S}(S)$ which encodes information about the natural product structures in the pants complex.  The complex of separating multicurves will be defined and explored in Section \ref{chap:separatingcomplex}.  The proof of Theorem \ref{thm:pieces} relies heavily on a notion of \emph{structurally integral corners} to be developed in Section \ref{sec:corners}.  Roughly speaking, a structurally integral corner is a point in the asymptotic cone whose removal disconnects particular natural product regions.  Structurally integral corners only exist for low complexity surfaces.  Theorem \ref{thm:pieces} should be compared with Theorem 4.6 of \cite{behrstockcharney} and Theorem 7.9 of \cite{bkmm} where similar characterizations of the finest pieces are proven for  right angled Artin groups and mapping class groups, respectively.  

The following theorems can be recovered as special cases of Theorem \ref{thm:pieces}.

\begin{repcor}$\textbf{\ref{cor:hyperbolic}.}$ (\cite{behrstock, brockfarb} Theorem 5.1, Theorem 1.1).
\emph{$\mathcal{T}(S_{1,2})$ and $\mathcal{T}(S_{0,5})$ are $\delta$-hyperbolic. }
\end{repcor}

\begin{repcor} $\textbf{\ref{cor:relhyperbolic}.}$ (\cite{brockmasur} Theorem 1).
\emph{For $\xi(S)=3,$ $\mathcal{T}(S)$ is relatively hyperbolic with respect to natural quasi-convex product regions consisting of all pairs of pants with a fixed separating curve.}
\end{repcor}

More generally, in the course of studying non-positively curved metric spaces, such as $\mathcal{T}(S),$ one is frequently interested in families of geodesics which admit \emph{hyperbolic type} properties, or properties exhibited by geodesics in hyperbolic space which are not exhibited by geodesics in Euclidean space.  In the geometric group theory literature there are various well studied examples of such hyperbolic type properties including being Morse, being contracting, and having cut-points in the asymptotic cone.  Such studies have proven fruitful in analyzing right angled Artin groups \cite{behrstockcharney}, Teichm\"uller space \cite{behrstock,brockfarb,brockmasur,bmm2}, the mapping class group \cite{behrstock}, CAT(0) spaces \cite{behrstockdrutu,bestvinafujiwara,charney}, and Out($F_{n}$) \cite{algomkfir} among others (for instance \cite{drutumozessapir,drutusapir,kapovitchleeb,osin,mm1}).

A \emph{morse geodesic} $\gamma$ is defined by the property that all quasi-geodesics $\sigma$ with endpoints on $\gamma$ remain within a bounded distance from $\gamma.$  A  \emph{strongly contracting geodesic} has the property that metric balls disjoint from the geodesic have nearest point projections onto the geodesic with uniformly bounded diameter.  It is an elementary fact that in hyperbolic space all geodesics are Morse and strongly contracting.  On the other end of the spectrum, in product spaces such as Euclidean spaces of dimension two and above, there are no Morse or strongly contracting geodesics.  Relatedly, there are no cut-points in any asymptotic cones of product spaces, whereas all asymptotic cones of $\delta$-hyperbolic spaces are $\R$-trees, and hence any two distinct points are separated by a cut-point.  The following theorem characterizes strongly contracting  (or equivalently Morse) quasi-geodesics in $\mathcal{T}(S).$  This family of strongly contracting quasi-geodesics represents a generalization of quasi-geodesics with \emph{bounded combinatorics} studied in \cite{bmm2} and similarly in \cite{behrstock}.  

\begin{repthm} $\textbf{\ref{thm:contracting}.}$  \emph{Let $\gamma$ be a quasi-geodesic in $\mathcal{\overline{T}}(S),$ and let $\gamma'$ be a corresponding quasi-geodesic in $\mathcal{P}(S).$  Then $\gamma$ is strongly contracting if and only if there exists a constant $C$ such that for all $Y \in \se{S},$ the subsurface projection $\pi_{Y}(\gamma')$ has diameter bounded above by $C.$ }
\end{repthm}

%

Later, we focus in particular on the Teichm\"uller space of the surface $S_{2,1}$ which in the literature has previously proven to be difficult to analyze.  As noted, for ``small'' complexity surfaces which don't admit any nontrivial separating curves, Brock-Farb \cite{brockfarb} prove that $\mathcal{T}(S)$ is hyperbolic.  A new proof was later provided by Behrstock in \cite{behrstock}.  Similarly, for ``medium'' complexity surfaces, which admit nontrivial separating curves, yet have the property that any two separating curves intersect, Brock-Masur prove that $\mathcal{T}(S)$ is relatively hyperbolic, \cite{brockmasur}.  Finally, for all the remaining ``large'' complexity surfaces excluding $S_{2,1},$ whose complexes of separating multicurves only have a single infinite diameter connected component, the combined work of \cite{behrstock, brockmasur}, implies that the Teichm\"uller spaces of these surfaces are not relatively hyperbolic and in fact are thick of order one.  However, unlike all other surfaces of finite type, the surface $S_{2,1}$ has the peculiar property that it is ``large enough'' such that it admits disjoint separating curves, although ``too small'' such that the complex of separating multicurves has infinitely many infinite diameter connected components.  As we will see, this phenomenon makes the study of the Teichm\"uller space of $S_{2,1}$ quite rich.

Using Theorem \ref{thm:pieces} in conjunction with a careful analysis of the Brock-Masur construction for showing that $\mathcal{T}(S_{2,1})$ is thick of order at most two \cite{brockmasur}, we prove the following theorem answering question 12.8 of \cite{bdm}.  

\begin{repthm} $\textbf{\ref{thm:thick2}.}$
 \emph{$\mathcal{T}(S_{2,1})$ is thick of order two.}
\end{repthm}

Notably, Theorem \ref{thm:thick2} completes the thickness classification of the Teichm\"uller spaces of all surfaces of finite type.  Moreover, among all surfaces of finite type, $S_{2,1}$ is the only surface that is thick of order two.  

The \emph{divergence} of a metric space measures the inefficiency of detour paths.  More formally, divergence along a geodesic is defined as the growth rate of the length of detour paths connecting sequences of pairs of points on a geodesic, where the distance between the pairs of points is growing linearly while the detour path is forced to avoid linearly sized metric balls centered along the geodesic between the pairs of points.  It is an elementary fact of Euclidean geometry that Euclidean space has linear divergence.  On the other end of the spectrum, hyperbolic space has exponential divergence.  

Given this gap between the linear divergence in Euclidean space and the exponential divergence in hyperbolic space, the exploration of spaces with ``intermediate divergence'' provides a means of understanding a rich spectrum of non-positively curved geometries which interpolate between flat and negatively curved geometries.  The history of this exploration goes back to Gromov, who noticed that $\delta$-hyperbolic spaces, like $\mathbb{H}^{n},$ have at least exponential divergence, \cite{gromov}.  Gromov then asked if there were non-positively curved spaces whose divergence functions were superlinear yet subexponential, \cite{gromov2}.  Soon afterward, Gersten answered this question in the affirmative by constructing CAT(0) groups with quadratic divergence, \cite{gersten}.  In short order Gersten proved that in fact the family of fundamental groups of graph manifolds provided natural examples of spaces with quadratic divergence \cite{gersten2}.  Moreover, in recent years it has been shown that various other well studied groups such as mapping class groups, right angled Artin groups, and Teichm\"uller spaces with the Teichm\"uller metric also have quadratic divergence, \cite{behrstock,behrstockcharney,duchinrafi}.  

After identifying spaces with quadratic divergence, Gersten went on to reformulate Gromov's question and asked if there existed CAT(0) spaces with superquadratic yet subexponential divergence.  This latter question of Gersten was recently answered in the affirmative by independent papers of Behrstock-Dru\c tu and Macura who each constructed CAT(0) groups with polynomial of degree $n$ divergence functions for every natural number $n,$ \cite{behrstockdrutu, macura}.  In Section \ref{sec:divergence} we show that a naturally occurring Teichm\"uller space, $\mathcal{T}(S_{2,1}),$ which is CAT(0), also provides an example answering Gersten's question in the affirmative.  In fact, we prove the following theorem answering question 4.19 in \cite{behrstockdrutu}:
\begin{repthm} $\textbf{\ref{thm:atleastsuperquad}.}$
\emph{$\mathcal{T}(S_{2,1})$ has superquadratic yet at most cubic divergence.  Moreover, it is the unique Teichm\"uller space with this property.}
\end{repthm}

A common approach to proving that a geodesic has at least quadratic divergence is to show that a geodesic is contracting.  Contraction implies that in order for a connected subsegment of a detour path avoiding a ball of radius $R$ centered on the geodesic to have nearest point projection onto the geodesic of more than a uniformly bounded diameter, the length of the subsegment must be linear in $R.$  In turn, it follows that a detour path must travel at least a linear amount of linear distances, and hence at least a quadratic distance.  See \cite{behrstock} for such an approach in proving that $\mathcal{MCG}$ has quadratic divergence.  In the proof of Theorem \ref{thm:atleastsuperquad} we follow the previously sketched outline, although we pick a careful example of a quasi-geodesic such that the detour path must in fact travel a linear amount of superlinear distances, thereby ensuring superquadratic divergence.  Since cut-points in asymptotic cones correspond to instances of superlinear divergence, Theorem \ref{thm:pieces} has a role in the proof of Theorem \ref{thm:atleastsuperquad}.  Conjecturally $\mathcal{T}(S_{2,1})$ has cubic divergence, see \cite{behrstockdrutu}.  

\subsection*{Acknowledgements}
This paper comprises the main portion of the author's doctoral thesis \cite{sultanthesis}.  My sincerest gratitude goes to my advisors Jason Behrstock and Walter Neumann for all of their invaluable guidance and extremely helpful advice throughout my research and specifically with regard to this paper.  They selflessly encouraged, advised, directed, and challenged me in my research.  I would also like to thank Lee Mosher and Saul Schleimer for useful suggestions and comments.  Finally, I dedicate this paper to my wife Ann and my son Bobby.

\section{Preliminaries}
\label{chap:preliminaries}
\label{subsec:background}

 \begin{defn}[coarse intersection]\label{def:coarse} 
Given a metric space $X,$ and subsets $A,B \subset X,$ the subsets \emph{coarsely intersect}, denoted $A \hat{\cap} B \ne \emptyset ,$ if there exists a positive constant $r$ such that any two elements in the collection of subsets $\{N_{R}(A) \cap N_{R}(B) | R \geq r \}$ have finite Hausdorff distance.  Moreover, if $C \subset X$ has finite Hausdorff distance from any set $N_{R}(A) \cap N_{R}(B),$ then $C$ is the \emph{coarse intersection} of the subsets $A$ and $B.$ If $C$ has bounded diameter, we say the subsets $A$ and $B$ have \emph{bounded coarse intersection}.
\end{defn}   


\begin{defn} [quasi-isometry]  Given metric spaces $(X,d_{X}),(Y,d_{Y}),$ a map $f\co(X,d_{X}) \ra (Y,d_{Y})$ is called a  \emph{$(K,L)$ quasi-isometric embedding of $X$ into $Y$} if there exist constants
$K\geq 1, L \geq 0$ such that for all $x,x'\in X$ the following
inequality holds: $$ K^{-1}d_{X}(x,x')-L \leq  d_{Y}(f(x),f(x')) \leq
Kd_{X}(x,x')+L$$ If in addition,  the map f is  \textit{roughly
onto}, i.e. a fixed neighborhood of the image must be the entire
codomain, f is called a \emph{quasi-isometry}.  Two metric spaces are called \textit{quasi-isometric} if and only if there exists a quasi-isometry between them.  The special case of a quasi-isometric embedding with domain a line (segment, ray, or bi-infinite) is a \emph{quasi-geodesic}. 
\end{defn}

\begin{rem} To simplify notation, we sometimes write:  $$ d_{X}(x,x')\approx_{K,L} d_{Y}(y,y') \mbox{ to imply }K^{-1}d_{X}(x,x')-L \leq  d_{Y}(y,y') \leq
Kd_{X}(x,x')+L$$ for some $K,L.$  Similarly, we write $d_{X}(x,x') \lesssim_{K,L} d_{Y}(y,y')$ to imply $ d_{X}(x,x') \leq Kd_{Y}(y,y') + L .$  When the constants $K,L$ are not important, they will be omitted from the notation.  
\end{rem}

\subsubsection{Curves and Essential Subsurfaces}
  Let  $S=S_{g,n},$ by any surface of finite type.  That is, $S$ is a genus g surface with $n$ boundary components (or punctures).  The \emph{complexity} of $S,$ denoted $\xi(S),$ is defined to be $3g-3+n.$  While in terms of the mapping class group there is a distinction between boundary components of a surface and punctures on a surface, as elements of the mapping class group must fix the former, yet can permute the latter, for our the purposes such a distinction will not be relevant.  Accordingly, throughout while we will always refer to surfaces with boundary components, the same results hold mutatis mutandis for surfaces with punctures.

A simple closed curve $\gamma$ on a surface $S$ is \emph{peripheral} if it bounds a disk, once punctured disk, or annulus; a non-peripheral curve is \emph{essential}.  We only consider essential simple closed curves up to isotopy and by abuse of notation will refer to the isotopy classes simply as curves.  Since we consider curves up to isotopy, we can always assume that their intersections are transverse and cannot be removed.  Equivalently, $S \setminus \left( \gamma_{1} \cup \gamma_{2} \right)$ does not contain any bigons.  We say that two curves are \emph{disjoint}, denoted $\gamma_{1} \cap \gamma_{2}  = \emptyset,$ if they can be drawn disjointly on the surface.  Otherwise, we say that the curves \emph{intersect}, denoted $\gamma_{1} \cap \gamma_{2}  \ne \emptyset.$  A \emph{multicurve} is a set of disjoint non parallel curves.  

An \emph{essential subsurface} $Y$ of a surface $S$ is a subsurface $Y \subseteq S$ such that $Y$ is a union of (not necessarily all) complementary components of a multicurve.  We always consider essential subsurfaces and by abuse of notation will refer to the isotopy classes of essential subsurfaces simply as essential subsurfaces.  Furthermore, we always assume every connected component of every essential subsurface $Y\subset S$ has complexity at least one.  In particular, unless otherwise noted annuli or pairs of pants are not considered essential subsurfaces and do not appear as connected components of essential subsurfaces.  For a fixed surface $S,$ let $\e{S}$ denote the set of all connected essential subsurfaces of $S.$   

Given any essential subsurface $Y$ we define the \emph{essential complement of $Y$}, denoted $Y^{c},$ to be the  maximal (in terms of containment) essential subsurface in the complement $S \setminus Y$ if such an essential subsurface exists, and to be the empty set otherwise.  An essential subsurface $Y$ is called a \emph{separating essential subsurface} if the complement $S \setminus Y$ contains an essential subsurface, or equivalently $Y^{c}$ is nontrivial.  The reason for the name separating essential subsurface is due to that the fact that $Y$ is a separating essential subsurface if and only if the boundary $\partial Y$ is a \emph{separating multicurve}, an object we will consider at length in Section~\ref{chap:separatingcomplex}.  All other essential subsurfaces which are not separating essential subsurfaces, are defined to be \emph{nonseparating essential subsurfaces}.  For example, if $Y$ is an essential subsurface such that the complement $S \setminus Y$ consists of a disjoint union of annuli and pairs of pants, then $Y$ is a nonseparating essential subsurface.  Let the subsets $\se{S},\nse{S}\subset \e{S}$ denote the sets of all connected separating, nonseparating essential subsurfaces of $S,$ respectively. 

An essential subsurface $Y$ is \emph{proper} if it is not all of $S.$  If two essential subsurfaces $W,V$ have representatives which can be drawn disjointly on a surface they are said to be \emph{disjoint}.  On the other hand, we say $W$ is \emph{nested} in $V,$ denoted $W \subset V,$ if $W$ has a representative which can be realized as an essential subsurface inside a representative of the essential subsurface $V.$  If $W$ and $V$ are not disjoint, yet neither essential subsurface is nested in the other, we say that $W$ \emph{overlaps} $V,$ denoted $W \pitchfork V .$  In general, if two essential subsurfaces $W,V$ either are nested or overlap, we say that the surfaces \emph{intersect} each other.  In such a setting we define the \emph{essential intersection}, denoted $W \cap V,$ to be the maximal essential subsurface which is nested in both $W$ and $V,$ if such an essential subsurface exists, and the emptyset otherwise.  Note that $W \cap V$ may be trivial even if the essential subsurfaces $W,V$ are not disjoint, as the intersection $W \cap V$ may be supported in a subsurface which is not essential.  For instance, see Figure \ref{fig:essentialintersections}.  Similarly, the \emph{essential complement of V in W}, denoted $W \setminus V,$ is defined to be the maximal essential subsurface in $(S\cap W) \setminus Y$ if such an essential subsurface exists, and to be the empty set otherwise. 

\begin{figure}[h]
\centering
\includegraphics[height=5 cm]{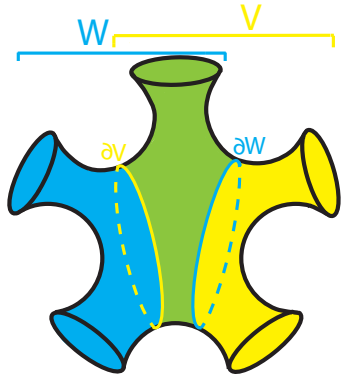}
\caption[Overlapping essential subsurfaces]{$W,V \in \e{S},$ $W \pitchfork V.$}  \label{fig:essentialintersections}
\end{figure}
A multicurve $C$ is \emph{disjoint} from an essential subsurface $Y,$ denoted  $C \cap Y  = \emptyset,$ if the multicurve and essential subsurface have representatives which can be drawn disjointly on the surface.  Otherwise, the multicurve $C$ and the essential subsurface $Y$ are said to \emph{intersect}.  In particular, given a proper essential subsurface $Y \subsetneq S,$ the boundary parallel curve(s) $\partial Y$ are disjoint from $Y.$ 

\subsubsection{Curve and Pants Complex}
 
For any surface $S$ with positive complexity, the \emph{curve complex} of $S,$ denoted $\mathcal{C}(S),$ is the simplicial complex obtained by associating a 0-cell to each curve, and more generally a k-cell to each multicurve with $k+1$ elements.  In the special case of low complexity surfaces which do not admit disjoint curves, we relax the notion of adjacency to allow edges between vertices corresponding to curves which intersect minimally on the surface.  $\mathcal{C}(S)$ is a locally infinite, infinite diameter, $\delta$-hyperbolic metric space, see \cite{mm1}. 

We will be particularly interested in maximal multicurves, or \emph{pants decompositions}.  Equivalently, a pants decomposition is a multicurve $\{ \gamma_{1},...,\gamma_{m} \}$ such that $S - \{\gamma_{1},...,\gamma_{m} \}$ consists of a disjoint union of \emph{pairs of pants}, or $S_{0,3}$'s.  For example, in Figure \ref{fig:essentialintersections} the multicurve $\{\partial W, \partial V\}$ is a pants decomposition of $S_{0,5}.$ 

Related to the curve complex, $\mathcal{C}(S),$ there is another natural complex associated to any surface of finite type with positive complexity: the \emph{pants complex}, $\mathcal{P}(S).$  To be sure, the pants complex is a 2-complex, although for our purposes, since we will only be interested in the quasi-isometry type of the pants complex, it will suffice to consider the 1-skeleton of the pants complex, the \emph{pants graph}.  By abuse of notation, we often refer the pants graph as the pants complex.  The pants graph has vertices corresponding to different pants decompositions of the surface up to isotopy, and edges between two vertices when the two corresponding pants decompositions differ by a so-called \emph{elementary pants move}.  Specifically, two pants decompositions of a surface differ by an elementary pants move if the two decompositions differ in exactly one curve and inside the unique connected complexity one essential subsurface in the complement of all the other agreeing curves of the pants decompositions  (topologically either an $S_{1,1}$ or an $S_{0,4}$) the differing curves intersect minimally (namely, once if the connected complexity one essential subsurface is $S_{1,1}$ and twice if the connected complexity one essential subsurface is $S_{0,4}$).  The pants graph is connected \cite{hatcherthurston}, and we will view it as a metric space by endowing it with the graph metric.

\subsubsection{The Pants complex and Teichm\"uller space}
\begin{defn} [Teichm\"uller space]  For $S$ a surface of finite type with $\chi(S)<0,$ the \emph{Teichm\"uller space of S} is the set of isotopy classes of hyperbolic structures on $S.$  Formally, $\mathcal{T}(S) = \{(f,X) | f\co S \ra X \} / \sim,$ where $S$ is a \emph{model surface} (a topological surface without a metric), $X$ is a surface with a hyperbolic metric, the map $f$ is a homeomorphism called a \emph{marking}, and the equivalence relation is given by $(g,Y) \sim (f,X) \iff gf^{-1} \mbox{ is isotopic to an isometry}.$  Often we omit the marking from the notation.
\end{defn}

It is a standard result that as a topological space, $\mathcal{T}(S)$ homeomorphic to $\mathcal{R}^{6g-6+2b+3p},$ where $g$ is the genus, $b$ is the number of boundary components, and $n$ is the number of punctures; for instance, see \cite{primer} for a proof.  On the other hand, a more interesting and active area of research is to study $\mathcal{T}(S)$ as a metric space.  For purposes of this paper, since we seek to explore the large scale geometric properties of Teichm\"uller space, we will not need to use the actual integral form definition of the WP metric but in its place will use the pants complex as a combinatorial model for studying $\mathcal{T}(S).$  Specifically, as justified by the conjunction of the following two theorems, in order to study quasi-isometry invariant properties of $\mathcal{T}(S),$ such as for instance thickness and divergence, it suffices to study the quasi-isometric model of Teichm\"uller space given by the pants complex.  

\begin{thm}[\cite{bers1, bers2} Bers constant] \label{thm:bers} $\exists$ a \emph{Bers constant} $B(S),$ such that $\forall X \in \mathcal{T}(S),$ there exists a \emph{Bers pants decomposition} $X_{B} \in \mathcal{P}^{0}(S)$ such that $\forall \alpha \in X_{B},$ the length $l_{X}(\alpha) \leq B.$  In other words, every point in Teichm\"uller space has a pants decomposition consisting of all short curves, where short is measured relative to a uniform constant depending only on the topology of the surface.    
\end{thm}   


Using the mapping suggested by Theorem \ref{thm:bers}, the following groundbreaking theorem of Brock proves that $\mathcal{T}(S)$ and $\mathcal{P}(S)$ are quasi-isometric.

\begin{thm} [\cite{brock} Theorem 3.2] \label{thm:brock} The mapping $\Psi \co (\mathcal{T}(S),WP) \ra (\mathcal{P}(S),\mbox{graph metric})$ given by $$X \mapsto B_{X} $$ where $B_{X} \in \mathcal{P}(S)$ is a Bers pants decomposition of $X$ as in Theorem \ref{thm:bers}, is coarsely well-defined, and moreover, is a quasi-isometry. 
\end{thm}   

\subsubsection{Ultrapowers and Asymptotic Cones}
 

A \emph{non-principal ultrafilter} is a subset $\omega \subset 2^{\N},$ satisfying the following properties:
\begin{enumerate}
\item $\omega$ is non empty; $\omega$ does not contain the empty set (filter),
\item $X,Y \in \omega\;  \implies \; X\cap Y \in \omega$  (filter),
\item $X\subset Y,$ $X \in \omega\;  \implies \; Y \in \omega$ (filter),
\item $X \not \in \omega$ $\implies$ $(\N \setminus X) \in \omega$ (ultrafilter), and
\item $|X| < \infty$ $\implies \; X \not \in \omega$ (non-principal).
\end{enumerate}

Given a sequence of points $(x_{i})$ and an ultrafilter $\omega,$ the \emph{ultralimit of $(x_{i})$}, denoted $\lim_{\omega} x_{i},$ is defined to be $x$ if for any neighborhood $U$ of $x,$ the set $\{i: x_{i} \in U\} \in \omega.$  That is, \emph{$\omega$ almost surely (or \uas)} $x_{i} \in U.$  Ultralimits are unique when they exist.


Given any set $S$ and an ultrafilter $\omega,$ we define the \emph{ultrapower} of $S,$ denoted $S^{\omega},$ as sequences $\overline{s}$ or $(s_{i})$ under the equivalence relation $\overline{s}\sim \overline{s'} \iff$ \uas $s_{i} =s'_{i}.$   Elements of the ultrapower will be denoted $s^{\omega}$ and their representative sequences will be denoted by $\overline{s}$ or $(s_{i}).$  By abuse of notation we  will sometimes denote elements of the ultrapower and similarly elements of the asymptotic cone by their representative sequences.

For a metric space $(X,d),$ we define the \emph{asymptotic cone of X}, relative to a fixed choice of ultrafilter $\omega,$ a sequence of base points in the space $(x_{i}),$ and an unbounded sequence of positive scaling constants $(s_{i}),$ as follows:
$$ Cone_{\omega}(X,(x_{i}) ,(s_{i})) \equiv  \lim_{\omega}(X,x_{i},d_{i}=\frac{d}{s_{i}})  $$
When the choice of scaling constants and base points are not relevant we denote the asymptotic cone of a metric space $X$ by $X_{\omega}.$  Elements of asymptotic cones will be denoted $x_{\omega}$ with representatives denoted by $\overline{x}$ or $(x_{i}).$  For $\mathcal{P}(S)$ we denote $Cone_{\omega}(\mathcal{P}(S),(P_{0}^{i}),(s_{i})) = \mathcal{P}_{\omega}(S).$  In particular, we assume a fixed base point of our asymptotic cone with representative given by $(P^{0}_{i}).$  Furthermore, unless otherwise specified always assume a fixed ultrafilter $\omega.$ 

More generally, given a subset $Y \subset X,$ and a choice of asymptotic cone $X_{\omega},$ throughout we will often consider the \emph{ultralimit} of Y, denoted $Y_{\omega},$ defined as follows: $$Y_{\omega} =: \{y_{\omega} \in X_{\omega} | y_{\omega} \mbox{ has a representative sequence $(y'_{i})$ with } y'_{i} \in Y \; \omega\mbox{-a.s}\}$$  In particular, when dealing with ultralimits we will always be considering the ultralimits as subsets contained inside an understood asymptotic cone.  Furthermore, given a sequence of subspaces $Y_{i} \subset X,$ we can similarly define the ultralimit, $Y_{\omega}.$  Based on the context it will be clear which type of ultralimit is being considered.   


The following elementary theorem organizes  some well known elementary facts about asymptotic cones, see for instance \cite{kapovich}.
\begin{thm}  \label{thm:conebasics} For metric spaces $X,Y$ and any asymptotic cones $X_{\omega},$ $Y_{\omega},$
\begin{enumerate} 
\item $(X \times Y)_{\omega}$=$X_{\omega} \times Y_{\omega}.$ 
\item For $X$ a geodesic metric space,  $X_{\omega}$ is a geodesic metric space and in particular is locally path connected.
\item $X \approx Y$ implies $X_{\omega}$ and $Y_{\omega}$ are bi-Lipschitz equivalent.
\end{enumerate}
\end{thm}

The next elementary lemma which follows from the definition of $\omega$ will be useful on a couple of occasions.  
\begin{lem} \label{lem:finite}
If A is a finite set, then any $\overline{\alpha} \in A^{\omega}$ is \emph{\uas constant}.  That is, $\exists! a_{0}\in A$ such that $\{ i | \alpha_{i} =a_{0} \} \in \omega.$  In particular, $|A^{\omega}|= |A|.$ 
\end{lem}

%

Ultrapowers are more general than asymptotic cones, as the construction of an ultrapower can be applied to arbitrary sets as opposed to the construction of an asymptotic cone which can only be applied to metric spaces.  In fact, we will often be interested in ultrapowers of objects such as $\eultra{S},$ or the ultrapower of connected essential subsurfaces of $S.$  Similarly, we will consider $\seultra{S},\nseultra{S}),$ or the ultrapowers of separating, nonseparating connected essential subsurfaces of $S,$ respectively.  As an application of Lemma \ref{lem:finite}, since any essential subsurface is either separating or nonseparating, for any ultrapower of essential subsurfaces $\overline{Y},$ \uas $Y_{i}$ is either always separating or always nonseparating.  In particular, any $\overline{Y} \in \eultra{S}$ is either in $\seultra{S}$ or $\nseultra{S},$ and the two options are mutually exclusive.   

\subsubsection{(Relative) Hyperbolicity and Thickness}
 
For points $x_{1},x_{2}$ in any geodesic metric space $X,$ we use the notation $[x_{1},x_{2}]$ to denote a geodesic between them. 
\begin{defn}[$\delta$-hyperbolic]  A geodesic metric space $X$ is said to be \emph{$\delta$-hyperbolic} if it satisfies the \emph{$\delta$-thin triangles inequality}.  Specifically, there exists some constant $\delta \geq 0$ such that for any three points in the space $x_{1},x_{2},x_{3}$ and $[x_{i},x_{j}]$ any geodesic connecting $x_{i}$ and $x_{j},$ then $[x_{1},x_{3}] \subset N_{\delta}([x_{1},x_{2}]) \bigcup  N_{\delta}([x_{2},x_{3}]).$  A metric space is called \emph{hyperbolic} if it is $\delta$-hyperbolic for some $\delta.$
\end{defn}


An important generalization of hyperbolicity is the notion of relative hyperbolicity.  Informally, a metric space $X$ is relatively hyperbolic with respect to a collection of subsets $\mathcal{A},$ if when all of the subsets in $\mathcal{A}$ are collapsed to finite diameter sets, the resulting ``electric space,'' $X/\mathcal{A},$ is hyperbolic.  To exclude trivialities we can assume no set $A \in \mathcal{A}$ has finite Hausdorff distance from $X.$  More specifically, spaces satisfying the above are said to be \emph{weakly relatively hyperbolic}.  If, in addition, a weakly relatively hyperbolic space $X$ has the \emph{bounded coset penetration property}, namely quasi-geodesics with the same endpoints travel roughly through the same subsets in $\mathcal{A}$ both entering and exiting the same subsets near each other, then $X$ is said to be \emph{strongly relatively hyperbolic}.  We will use the following equivalent definition of strong relative hyperbolicity of a metric space due to \cite{drutusapir} formulated in terms of asymptotic cones: 

\begin{defn}[Relatively Hyperbolic] A metric space $(X,d)$ is said to be \emph{hyperbolic relative} to a collection of \emph{peripheral subsets} $\mathcal{A}$ if $X$ is \emph{asymptotically tree-graded}, with respect to $\mathcal{A}.$  That is,
\begin{enumerate}
\item Every asymptotic cone $X_{\omega}$ is \emph{tree-graded} with respect to the \emph{pieces} $A_{\omega}$ for $A \in \mathcal{A}.$  More specifically, the intersection of each pair of distinct pieces, $A_{\omega},A'_{\omega},$  has at most one point and every simple geodesic triangle (a simple loop composed of three geodesics) in $X_{\omega}$ lies in one piece $A_{\omega}.$
\item  $X$ is not contained in a finite radius neighborhood of any of the subsets in $\mathcal{A}.$
\end{enumerate}
\end{defn}

In contrast to earlier concepts of hyperbolicity or relatively hyperbolicity, we have the following notion of thickness developed in \cite{bdm} and explored further in \cite{behrstockdrutu}.  
 
\begin{defn}[Thickness]   \label{defn:thick}$\\$\begin{enumerate}
\item A space $X$ is said to be \emph{thick of order zero} if none of its asymptotic cones $X_{\omega}$ have cut-points, or equivalently $X$ is \emph{wide}, and moreover it satisfies the following nontriviality condition: there is a constant $c$ such that every $x\in X$ is distance at most $c$ from a bi-infinite quasi-geodesic in $X.$  
\item A space $X$ is said to be \emph{thick of order at most $n+1$} if there exist subsets  $P_{\alpha} \subset X,$ satisfying the following conditions:
\subitem (i) The subsets $P_{\alpha}$ are \emph{quasi-convex} (namely, there exist constants $(K,L,C)$ such that any two points in $P_{\alpha}$ can be connected by a (K,L)-quasi-geodesic remaining inside $N_{C}(P_{\alpha})$) and are thick of order at most $n$ when endowed with the restriction metric from the space $X,$
\subitem (ii) The subsets are \emph{almost everything}.  Namely, $\exists$ a fixed constant $R_{1}$ such that $ \bigcup_{\alpha} N_{R_{1}}(P_{\alpha}) = X,$
\subitem (iii) The subsets can be \emph{chained together thickly}.  Specifically, for any subsets $P_{\alpha}, P_{\beta},$ there exists a sequence of subsets $P_{\alpha}=P_{\gamma_{1}}, ..., P_{\gamma_{n}} = P_{\beta}$ such that for some fixed constant $R_{2} \geq 0,$ $diam( N_{R_{2}}(P_{\gamma_{i}}) \bigcap N_{R_{2}}(P_{\gamma_{i+1}}) )=\infty.$  In particular, due to the quasi-convexity assumption in (i), it follows that the coarse intersection between consecutive subsets being chained together is coarsely connected.
 \item A space $X$ is \emph{thick of order n} if $n$ is the lowest integer such that $X$ is thick of order at most $n.$
\end{enumerate}
\end{defn}

In Section \ref{chap:thick} we will often be interested in subspaces $Y \subset X$ which are \emph{thick of order zero}.  Namely, we say that a subspace $Y$ is thick of order zero if in every asymptotic cone $X_{\omega}$ the subset corresponding to the ultralimit $Y_{\omega}$ has the property that any two distinct points in $Y_{\omega}$ are not separated by a cut-point.  Additionally, we require that $Y$ satisfies the nontriviality condition of every point being within distance $c$ from a bi-infinite quasi-geodesic in $Y.$  

\begin{rem} \label{rem:wide}
It should be mentioned that Definition \ref{defn:thick} of thickness is what is in fact called \emph{strongly thick} in \cite{behrstockdrutu}, as opposed to the slightly more general version of thickness considered in \cite{bdm}.  As in \cite{behrstockdrutu}, for our purposes the notion of strong thickness is more natural as it proves to be more conducive to proving results regarding divergence, such as we will do in Section \ref{chap:thick}.  There are two differences between the different definitions of thickness.

First, as opposed to requirement in Definition \ref{defn:thick} (or equivalently in the definition of strong thickness in \cite{behrstockdrutu}) that thick of order zero subsets be wide, in \cite{bdm} a thick of order zero subset is only required to be \emph{unconstricted}.  Namely, there exists some ultrafilter $\omega$ and some sequence of scalars $s_{i}$ such that any asymptotic cone $Cone_{\omega}(X, \cdot ,(s_{i}))$ does not have cut-points.  Nonetheless, as noted in \cite{bdm} for the special case of finitely generated groups, the definition of thick of order zero in Definition \ref{defn:thick} (or being wide) is equivalent to the definition considered in \cite{bdm} (or being unconstricted).  Moreover, in \cite{sultanthesis} it is shown that for CAT(0) spaces with extendable geodesics, being wide and unconstricted are similarly equivalent.  Second, the requirement for quasi-convexity in condition (i) of Definition \ref{defn:thick} is omitted in the definition of thickness in \cite{bdm}.  \end{rem}

The following theorem of \cite{bdm}, which in fact inspired the development of the notion of thickness, captures the contrasting relationship between hyperbolicity and thickness:
\begin{thm}[\cite{bdm} Corollary 7.9] \label{thm:bdm}  A metric space $X$ which is thick of any finite order is not strongly relatively hyperbolic with respect to any subsets, i.e. \emph{non relatively hyperbolic (NRH)}.
\end{thm}

Another perspective is to understand thickness as a precise means of interpolating between two ends of the spectrum of non-positively curved spaces: product spaces and hyperbolic spaces.  On the one hand, nontrivial product spaces are thick of order zero (this follows from Theorem \ref{thm:conebasics} statement (2) as nontrivial products do not contain cut-points).  On the other hand, Theorem \ref{thm:bdm} says that strongly relatively hyperbolic and hyperbolic spaces are not thick of any order, or equivalently can be thought of as thick of order infinity.  Then, in this sense the higher the order of thickness of a metric space the closer the space resembles hyperbolic space and shares features of negative curvature.  From this point of view, the close connections between thickness and divergence explored in \cite{behrstockdrutu} as well as in Section \ref{chap:thick} are very natural.  

%


\subsection{Tools from mapping class groups}
\label{subsec:tools}

 In this section we review some tools developed by Behrstock \cite{behrstock}, Behrstock-Kleiner-Minsky-Mosher \cite{bkmm}, Behrstock-Minsky \cite{behrstockminsky}, and Masur-Minsky \cite{mm2} in their geometric analyses of the curve complex, $\mathcal{C}(S),$ and the marking complex, $M(S).$  If fact, in the aforementioned papers, many of these tools developed for the marking complex have simplifications which immediately apply to the pants complex.  


\subsubsection{Subsurface projections}  \label{subsec:subsurfaceprojection}

Given a curve $\alpha \in \mathcal{C}(S)$ and a connected essential subsurface $Y \in \e{S}$ such that $\alpha$ intersects $Y,$  we can define the projection of $\alpha$ to $2^{\mathcal{C}(Y)}$, denoted $\pi_{\mathcal{C}(Y)}(\alpha)$, to be the collection of vertices in $\mathcal{C}(Y)$ obtained in the following surgical manner.  Specifically, the intersection $\alpha \cap Y$ consists of either the curve $\alpha,$ if $\alpha \subset Y,$ or a non-empty disjoint union of arc subsegments of $\alpha$ with the endpoints of the arcs on boundary components of $Y.$  In the former case we define the projection $\pi_{\mathcal{C}(Y)}(\alpha)=\alpha.$  In the latter case, $\pi_{\mathcal{C}(Y)}(\alpha)$ consists of all curves obtained by the following process.  If an arc in $\alpha \cap Y$ has both endpoints on the same boundary component of $\partial Y,$ then $\pi_{\mathcal{C}(Y)}(\alpha)$ includes the curves obtained by taking the union of the arc and the boundary component containing the endpoints of the arc.  Note that this yields at most two curves, at least one of which is essential.  On the other hand, if an arc in $\alpha \cap Y$ has endpoints on different boundary components of $\partial Y,$ then $\pi_{\mathcal{C}(Y)}(\alpha)$ includes the curve on the boundary of a regular neighborhood of the union of the arc and the different boundary components containing the end points of the arc.  See Figure \ref{fig:curveproj} for an example. Note that above we have only defined the projection $\pi_{\mathcal{C}(Y)}$ for curves intersecting $Y,$ for all curves $\gamma$ disjoint from $Y,$ the projection $\pi_{\mathcal{C}(Y)}(\gamma)=\emptyset.$  

In any context concerning the curve complex of an essential subsurface, $\mathcal{C}(Y)$ in order to avoid distractions we alway assume that $Y \in \e{Y},$ i.e. the essential subsurface $Y$ is connected.  If not, then by definition $\mathcal{C}(Y)$ is a nontrivial join and hence has diameter two.  

\begin{figure}[h]
\centering
\includegraphics[height=5 cm]{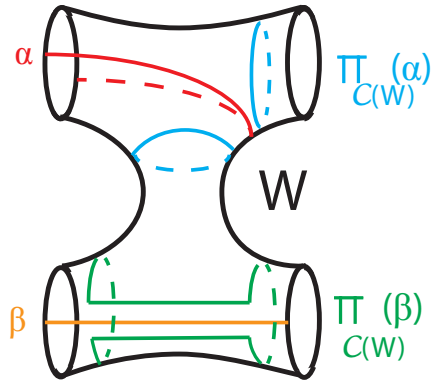}
\caption[Subsurface projection in the curve complex]{Performing s surgery on arcs in the connected proper essential subsurface $W \subsetneq S$ which makes them into curves in $\mathcal{C}(W).$  The arc $\alpha$ has both endpoints on the same boundary component of $W,$ whereas $\beta$ has endpoints on different boundary components of $W.$}\label{fig:curveproj}
\end{figure}


To simplify notation, we write $d_{\mathcal{C}(Y )} (\alpha_{1}, \alpha_{2})$ as shorthand for 
$d_{\mathcal{C}(Y )} (\pi_{\mathcal{C}(Y)}(\alpha_{1}), \pi_{\mathcal{C}(Y)}(\alpha_{2})).$  In particular, this distance is only well-defined if $\alpha_{1},\alpha_{2}$ intersect $Y.$  Similarly, for $A \subset \mathcal{C}(S),$ we write  $diam_{\mathcal{C}(Y)}(A)$ as shorthand for $diam_{\mathcal{C}(Y)}(\pi_{\mathcal{C}(Y)}(A)).$ 
  
The following lemma ensures that the subsurface projection $\pi_{\mathcal{C}(Y)}$ defined above gives a coarsely well-defined projection $\pi_{\mathcal{C}(Y)} \co \mathcal{C}(S) \ra \mathcal{C}(Y) \cup \emptyset.$
\begin{lem}[\cite{mm2}, Lemma 2.2] \label{lem:curveprojub}  
For $\alpha$ any curve and any $Y\in \e{Y}$ the set of curves $\pi_{\mathcal{C}(Y)}(\alpha)$ has diameter bounded above by three.  Hence, we have a coarsely well-defined subsurface projection map which by abuse of notation we refer to as $\pi_{\mathcal{C}(Y)} \co \mathcal{C}(S) \ra \mathcal{C}(Y) \cup \emptyset.$  In particular, if $\sigma$ is any connected path in $\mathcal{C}(S)$ of length $n,$ and $Y$ is any connected subsurface such that every curve in the path $\sigma$ intersects $Y,$ then $diam_{\mathcal{C}(Y)}(\sigma)\leq 3n.$  
\end{lem}

The next theorem describes a situation in which subsurface projection maps geodesics in the curve complex to uniformly bounded diameter subsets in the curve complex of a connected essential subsurface.  
\begin{thm}[\cite{mm2}, Theorem 3.1; Bounded Geodesic Image] \label{thm:projdefined}  Let $Y\in \e{S}$ be a connected proper essential subsurface of $S,$ and let $g$ be a geodesic (segment, ray, or bi-infinite) in $\mathcal{C}(S)$ such that every curve corresponding to a vertex of $g$ intersects $Y,$ then $diam_{\mathcal{C}(Y)} (g) $ is uniformly bounded by a constant $K(S)$ depending only on the topological type of $S.$
\end{thm}

In addition to projecting curves, we can similarly project multicurves.  In particular, we can project pants decompositions of surfaces to essential subsurfaces.  Specifically, for any essential subsurface $Y$ we have an induced coarsely well-defined projection map: $$\pi_{\mathcal{P}(Y)} \co \mathcal{P}(S) \ra \mathcal{P}(Y)$$  The induced map is defined as follows.  Beginning with any pair of pants $P \in \mathcal{P}(S)$ there is at least one curve $\alpha_{1} \in P$ intersecting $Y.$   We then proceed to construct a pants decomposition of $Y$ inductively.  As our first curve we simply pick any curve $ \beta_{1} \in \pi_{\mathcal{C}(Y)}(\alpha_{1}).$  Then, we consider the surface $Y \setminus \beta_{1}$ and notice that $\xi(Y \setminus \beta_{1}) = \xi(Y) -1.$  Replace $Y$ by $Y \setminus \beta_{1}$ and repeat this process until the complexity is reduced to zero.  At this point, the curves $\{ \beta_{i} \}$ are a pants decomposition of the essential subsurface $Y.$  Due to all the choice, the above process does not produce a unique pants decomposition.  Nonetheless, as in Lemma \ref{lem:curveprojub} the map is coarsely well-defined and in fact is coarsely Lipschitz with uniform constants \cite{mm2, behrstock}.  

The next lemma makes precise a sense in which distances under projections to curve complexes of overlapping surfaces are related to each other.  Intuitively, the point is that the distance in one subsurface projection can be large only at the expense of the distance in all overlapping essential subsurfaces being controlled.  
\begin{lem}[\cite{behrstock, mangahas} Theorem 4.3, Lemma 2.5; Behrstock Inequality]\label{lem:projectionestimates} For $S=S_{g,n},$ let $W,V\in \e{S}$ be such that $W \pitchfork V.$  Then, $\forall P \in \mathcal{P}(S)\co$
$$\min \left( d_{\mathcal{C}(W)}(\mu,\partial V), d_{\mathcal{C}(V)}(P,\partial W) \right) \leq 10$$
\end{lem}

Utilizing the projection $\pi_{\mathcal{P}(Y)} \co \mathcal{P}(S) \ra \mathcal{P}(Y),$ for $\overline{Y} \in \eultra{S}$ we can define $\mathcal{P}_{\omega}(\overline{Y})$ to be the ultralimit of $\mathcal{P}(Y_{i}).$  It is clear that $\mathcal{P}_{\omega}(\overline{Y})$ is isomorphic to $\mathcal{P}_{\omega}(Y)$ for $Y$ an essential subsurface \uas isotopic to $Y_{i}.$  Moreover, extending the coarsely well-defined Lipschitz projection $\pi_{\mathcal{P}(Y)} \co \mathcal{P}(S) \ra \mathcal{P}(Y)$ to the asymptotic cone, we have a Lipschitz projection $$\pi_{\mathcal{P}_{\omega}(\overline{Y})}: \mathcal{P}_{\omega}(S) \ra \mathcal{P}_{\omega}(\overline{Y}).$$

\subsubsection{Tight Geodesics and Hierarchies}\label{subsec:tight}
 
A fundamental obstacle in studying geodesics in the curve complex stems from the fact that the 1-skeleton is locally infinite.  In an effort to navigate this problem, in \cite{mm2} Masur-Minsky introduced a notion of \emph{tight multigeodesics}, or simply \emph{tight geodesics}, in $\mathcal{C}(S).$  Specifically, for $S$ a surface of finite type with $\xi(S)\geq 2,$ a tight geodesic in $\mathcal{C}(S)$ is a sequence of simplices $\sigma=(w_{0},...,w_{n})$ such that the selection of any curves $v_{i}\in w_{i}$ yields a geodesic in $\mathcal{C}(S)$ and moreover, for $1 \leq i \leq n-1,$ the simplex $w_{i}$ is the boundary of the essential subsurface filled by the curves $w_{i-1} \cup w_{i+1}.$  In the case of a surface $S$ with $\xi(S)=1$ every geodesic is considered tight.  For $\sigma$ a tight geodesic as above, we use the notation $[w_{i},w_{j}]=(w_{i},...,w_{j})$ to refer to a subsegment of the tight geodesic.  In \cite{mm2} it is shown that any two curves in $\mathcal{C}(S)$ can be joined by a tight geodesic (and in fact there are only finitely many).  

Using tight geodesics, in \cite{mm2} a 2-transitive family of quasi-geodesics, with constants depending on the topological type of $S,$ in $\mathcal{P}(S)$ called \emph{hierarchies}, are developed.  Since we are interested in paths in the pants complex as opposed to the marking complex, unless specified otherwise we use the term ``hierarchies'' to refer to what are in fact called ``resolutions of hierarchies without annuli'' in \cite{mm2}.  The construction of hierarchies which are defined inductively as a union of tight geodesics in the curve complexes of connected essential subsurfaces of $S$ is technical.  For our purposes, it will suffice to record some of their properties in the following theorem.  See \cite{brockmasur} Definition 9 for a similar statement. 

\begin{thm} [\cite{mm2} Section 4; Hierarchies]  \label{thm:hierarchy} For $S$ any surface of finite type, given $P,Q \in \mathcal{P}(S),$ there exists a hierarchy path $\rho=\rho(P,Q)\co [0,n] \ra \mathcal{P}(S)$ with $\rho(0)=P,$ $\rho(n)=Q.$  Moreover, $\rho$ is a quasi-isometric embedding with uniformly bounded constants depending only on the topological type of $S,$ which has the following properties: 
\begin{enumerate}
\item  [H1:] The hierarchy $\rho$ \emph{shadows} a tight $\mathcal{C}(S)$ geodesic $g_{S}$ from a multicurve $p \in P$ to a multicurve $q\in Q,$ called the \emph{main geodesic of the hierarchy}.  That is, there is a monotonic map $\nu\co \rho \ra g_{S}$ such that $\forall i, \; \nu_{i}=\nu(\rho(i))\in g_{S}$ is a curve in the pants decomposition $\rho(i).$ 
\item [H2:] There is a constant $M_{1}$ such that if $Y \in \e{S}$ satisfies $d_{\mathcal{C}(Y)}(P,Q) >M_{1},$ then there is a maximal connected interval $I_{Y}=[t_{1},t_{2}]$ and a tight geodesic $g_{Y}$ in $\mathcal{C}(Y)$  from a multicurve in $\rho(t_{1})$ to a multicurve in $\rho(t_{2})$ such that for all $t_{1}\leq t \leq t_{2},$ $\partial Y$ is a multicurve in $\rho(t),$ and  $\rho |_{I_{Y}}$ shadows the geodesic $g_{Y}.$  Such a connected essential subsurface $Y$ is called an \emph{$M_{1}$-component domain} or simply a \emph{component domain} of $\rho.$  By convention the entire surface $S$ is always considered a component domain.  
\item [H3:] If $Y_{1} \pitchfork Y_{2}$ are two component domains of $\rho,$ then there is a notion of time ordering $<_{t}$ of the domains with the property that $Y_{1} <_{t} Y_{2},$ implies  $d_{Y_{2}}(P,\partial Y_{1})<M_{1}$ and $d_{Y_{1}}(Q,\partial Y_{2})<M_{1}.$  Moreover, the time ordering is independent of the choice of the hierarchy $\rho$ from $P$ to $Q.$    
\item [H4:] For $Y$ a component domain with $I_{Y}=[t_{1},t_{2}],$ let $0 \leq s \leq t_{1},$ $ t_{2} \leq u \leq n.$  Then, $$d_{\mathcal{C}(Y)}(\rho(s),\rho(t_{1})), d_{\mathcal{C}(Y)}(\rho(u),\rho(t_{2}))\leq M_{1} .$$
\end{enumerate}
\end{thm}  

As a corollary of Theorem \ref{thm:hierarchy}, we have the following quasi-distance formula for computing distances in $\mathcal{P}(S)$ in terms of a sum of subsurface projection distances, where the sum is over all connected essential subsurfaces above a certain threshold. 
\begin{thm}[\cite{mm2} Theorem 6.12; Quasi-Distance Formula]  \label{thm:quasidistance}For $S=S_{g,n}$ there exists a \emph{minimal threshold} $M_{2}$ depending only on the surface $S$ and quasi-isometry constants depending only on the surface $S$ and the threshold $M\geq M_{2}$ such that: 
$$d_{\mathcal{P}(S)}(P,Q) \approx \sum_{Y\in \e{S}} \{ d_{\mathcal{C}(Y)}(P,Q) \}_{M} $$ 
 where the \emph{threshold function} $\{ f(x) \}_{M} :=f(x) \;\; \mbox{ if } f(x) \geq M,$ and $0$ otherwise.
\end{thm}

Note that by setting $M'=\max\{10,K,M_{1},M_{2}\}$ we have a single constant $M',$ depending only on the topology of the surface $S,$ which simultaneously satisfies Lemmas \ref{lem:curveprojub} and \ref{lem:projectionestimates}, and Theorems \ref{thm:projdefined}, \ref{thm:hierarchy}, and \ref{thm:quasidistance}.  Throughout we will use this constant $M'.$  

Sequences of hierarchies in the pants complex give rise to ultralimits of hierarchies in the asymptotic cone of the pants complex.  Specifically, given $x_{\omega},y_{\omega}\in \mathcal{P}_{\omega}(S)$ with representatives $(x_{i}),(y_{i}),$ respectively, let $\rho_{\omega}$ be the ultralimit of the sequence of hierarchy paths $\rho_{i}$ from $x_{i}$ to $y_{i}.$  Note that by construction, since $\rho_{i}$ are quasi-geodesics with uniform constants, as in Theorem \ref{thm:conebasics} it follows that $\rho_{\omega}$ is a (K,0)-quasi-geodesic path in the asymptotic cone from $x_{\omega}$ to $y_{\omega}.$

\subsubsection{Convex Regions, Extensions of Multicurves, and Regions of Sublinear Growth}
\label{subsec:sublineargrowth}
 
Given a multicurve $C \subset \mathcal{C}(S),$ by Theorem \ref{thm:quasidistance} we have a \emph{natural quasi-convex region}: 
\begin{equation} \label{eq:qcproduct}
\q{C}\equiv \{P \in \mathcal{P}(S) | C\subset P \}.
\end{equation}
Consider that an element $Q \in \q{C}$ is determined by a choice of a pants decomposition of $S \setminus C.$  Hence, $\q{C}$ can be naturally identified with $\mathcal{P}(S \setminus C),$ which has nontrivial product structure in the event that $S \setminus C$ is a disjoint union of two or more connected essential subsurfaces.  For example, given $W \in \se{S},$ $\q{\partial W}\approx \mathcal{P}(W) \times \mathcal{P}(W^{c}).$

After taking ultralimits, quasi-convex regions give rise to convex regions in the asymptotic cone.  Specifically, given an asymptotic cone $\mathcal{P}_{\omega}(S)$ and element of the ultrapower of multicurves $\overline{C}$ we have an ultralimit $$\qbar{C}=: \{x_{\omega} \in \mathcal{P}_{\omega}(S) | x_{\omega} \mbox{ has a representative $(x'_{i})$ with } x'_{i} \in \q{C_{i}} \; \omega\mbox{-a.s}  \}.$$   Note that unless $\lim_{\omega} \frac{1}{s_{i}} d_{\mathcal{P}(S)}(P^{0}_{i},\q{C_{i}}) < \infty,$ the ultralimit $\qbar{C}$ is trivial.  On the other hand, if $\lim_{\omega} \frac{1}{s_{i}} d_{\mathcal{P}(S)}(P^{0}_{i},\q{C_{i}}) < \infty,$ then $\qbar{C}$ can be naturally identified with $\mathcal{P}_{\omega}(S \setminus \overline{C}),$ which has a nontrivial product structure in the event that the multicurves $C_{i}$ \uas separate the surface $S$ into at least two disjoint connected essential subsurfaces.  Recall that we always assume essential subsurfaces have complexity at least one.  

Given a multicurve $C$ on a surface $S$ and a pants decomposition $X \in \mathcal{P}(S),$ we define the coarsely well-defined \emph{extension of C by X}, denoted $C \lrcorner X,$ by: $$C \lrcorner X  \equiv C \cup \pi_{\mathcal{P}(S \setminus C)}(X).$$  More generally, for $\overline{C}$ an element of the ultrapower of multicurves  satisfying $$\lim_{\omega} \frac{1}{s_{i}} d_{\mathcal{P}(S)}(P^{0}_{i},\q{C_{i}}) < \infty,$$ and $x_{\omega} \in \mathcal{P}_{\omega}(S)$ we can define the \emph{extension of $\overline{C}$ by $x_{\omega}$}, denoted $\overline{C} \lrcorner x_{\omega},$ by: $$\overline{C} \lrcorner x_{\omega}  \equiv \lim_{\omega} (C_{i} \lrcorner X_{i}) \in \mathcal{P}_{\omega}(S),$$ where $(X_{i})$ is any representative of $x_{\omega}.$  

In \cite{bkmm} the set of natural quasi-convex regions  $\q{C}$ and their generalization to the asymptotic cone is studied at length.  In particular, the following theorem is proven:
\begin{thm}[\cite{bkmm} Lemma 3.3, Section 3.4] \label{thm:convex} 
Given two quasi-convex regions  $\q{C},$  $\q{D}$ for $C,D$ isotopy classes of multicurves, the closest point set in $\q{C}$ to $\q{D}$ is coarsely $\q{C \lrcorner D}.$  In particular, $\q{C} \hat{\cap} \q{D}$ can be represented by either $\q{C \lrcorner D}$ or equivalently $\q{D \lrcorner C}.$

For convex regions $\qbar{C},$  $\qbar{D}$ in the asymptotic cone $\mathcal{P}_{\omega}(S),$ the closest point set in $\qbar{C}$ to $\qbar{D}$ is $\qbar{C \lrcorner D}.$  In fact, the intersection $ \qbar{C} \cap \qbar{D}$ is nonempty if and only if $\qbar{C \lrcorner D} = \qbar{D \lrcorner C}.$  Moreover, in this case the intersection is equal to $\qbar{C \lrcorner D}.$
\end{thm}

With the result of Theorem \ref{thm:quasidistance} in mind, \cite{behrstock} and later \cite{behrstockminsky} developed a stratification of $\mathcal{P}_{\omega}(S)$ by considering regions of so-called \emph{sublinear growth}.  Specifically, given $\overline{W} \in \eultra{S}$ and $x_{\omega} \in \mathcal{P}_{\omega}(\overline{W}),$ we define the subset of $\mathcal{P}_{\omega}(\overline{W})$ with \emph{sublinear growth from $x_{\omega}$}, denoted $F_{\overline{W},x_{\omega}},$ as follows: 
$$F_{\overline{W},x_{\omega}}= \{ y_{\omega} \in \mathcal{P}_{\omega}(\overline{W}) \; | \; \forall \overline{U} \subsetneq \overline{W}, \;  d_{\mathcal{P}_{\omega}(\overline{U})}(x_{\omega},y_{\omega}) =0  \}.$$ 


The following theorem organizes some properties of subsets of sublinear growth.
\begin{thm}[\cite{behrstockminsky} Theorem 3.1] \label{thm:sublinear} 
With the same notation as above,
\begin{enumerate}
\item [S1:] $z_{\omega} \ne z'_{\omega} \in  F_{\overline{W},x_{\omega}}$ $\implies$ $\lim_{\omega} d_{\mathcal{C}(W_{i})}(z_{i},z'_{i}) \ra \infty$ for $(z_{i}),(z'_{i})$ any representatives of $z_{\omega},z'_{\omega},$ respectively.  In particular, if $\gamma_{i}$ is a hierarchy between $z_{i}$ and $z'_{i}$ shadowing a tight main geodesic $\beta_{i}$ in $\mathcal{C}(W_{i})$ connecting any curves in the simplices $z_{i}$ and $z'_{i},$ then $\lim_{\omega} |\beta_{i}|$ is unbounded.
\item [S2:] $F_{\overline{W},x_{\omega}} \subset  \mathcal{P}_{\omega}(\overline{W}) $ is a convex $\R$-tree.  
\item [S3:] There is a continuous nearest point projection $$\rho_{\overline{W},x_{\omega}}\co \mathcal{P}_{\omega}(\overline{W}) \ra F_{\overline{W},x_{\omega}}$$ where $\rho_{\overline{W},x_{\omega}}$ is the identity on $F_{\overline{W},x_{\omega}}$ and locally constant on $ \mathcal{P}_{\omega}(\overline{W}) \setminus F_{\overline{W},x_{\omega}}.$
\end{enumerate}
 \end{thm}
 
We record a proof of property [S1] as ideas therein will be used later in the proof of Theorem \ref{thm:pieces}.  For a proof of the rest of the theorem see \cite{behrstockminsky}.

\begin{proof} 
Proof of [S1]: Assume not.  That is, assume $\exists$ a constant $K\geq 0$ such that \uas $\lim_{\omega}d_{\mathcal{C}(S)}(z_{i},z'_{i}) \leq K.$  Since $\{0,...,K\}$ is a finite set, by Lemma \ref{lem:finite} there is a $k \leq K$ such that \uas $\lim_{\omega}d_{\mathcal{C}(S)}(z_{i},z'_{i})= k.$  In particular, \uas there is a tight geodesic $\beta_{i}$ in $\mathcal{C}(S),$ with simplices $b_{i0},..., b_{ik}$ such that $b_{i0} \subset  z_{i} ,$ $b_{ik} \subset  z'_{i} .$  Thus \uas we can construct a quasi-geodesic hierarchy path $\gamma_{i}$ between $z_{i}$ and $z'_{i}$ with main geodesic $\beta_{i}$ of length $k.$

At the level of the asymptotic cone we have a quasi-geodesic $\gamma_{\omega}$ from $z_{\omega}$ to $z'_{\omega}$ which travels through a finite list of regions $\qbar{b_{j}}$ where $\overline{b_{j}}=(b_{i,j})_{i} \in \mathcal{C}(S)^{\omega}$ for $j \in \{0,...,k\} .$  Moreover, $\gamma_{\omega}$ enters each region $\qbar{b_{j}}$ at the point $\overline{b_{j}} \lrcorner z_{\omega}$ and exits each region at the point  $\overline{b_{j}} \lrcorner z'_{\omega}.$  Since $z_{\omega}, z'_{\omega} \in  F_{\overline{W},x_{\omega}},$ by definition for any $\overline{Y} \subsetneq \overline{W}$ $\pi_{\mathcal{P}_{\omega}(\overline{Y})}(z_{\omega})=\pi_{\mathcal{P}_{\omega}(\overline{Y})}(z'_{\omega}).$  In particular, this holds for $\overline{Y^{j}}$ with $Y^{j}_{i}=W_{i} \setminus b_{i,j}$ for any $j.$  It follows that the ultralimit of the hierarchy paths $\gamma_{\omega}$ enters and exits each region $\qbar{b_{j}}$ at the same point.  Since the regions $\qbar{b_{j}}$ are convex, we can assume the quasi-geodesic $\gamma_{\omega}$ intersects each region in a single point.  This leads to a contradiction since by assumption  $z_{\omega} \ne z'_{\omega},$ yet there is a quasi-geodesic path $\gamma_{\omega}$ of length zero connecting the two points.
\end{proof}

In \cite{behrstockminsky}, regions of sublinear growth are used to stratify product regions in the asymptotic cone.  Specifically,  for $\overline{W} \in \eultra{S}$ such that $\lim_{\omega} \frac{1}{s_{i}} d_{\mathcal{P}(S)}(P^{0},\q{\partial W_{i}}) < \infty,$ and $x_{\omega} \in \mathcal{P}_{\omega}(\overline{W}),$ we define the set $P_{\overline{W},x_{\omega}} \subset \qbar{\partial W}$ as follows:
 $$P_{\overline{W},x_{\omega}} =\{ y_{\omega} \in  \qbar{\partial W} \; | \; \pi_{\mathcal{P}_{\omega}(\overline{W})}(y_{\omega}) \in F_{\overline{W},x_{\omega}} \} \cong  \mathcal{P}_{\omega}(\overline{W^{c}}) \times F_{\overline{W},x_{\omega}}.$$
By precomposition with the projection $\pi_{\mathcal{P}_{\omega}(\overline{W})}\co \mathcal{P}_{\omega}(S) \ra \mathcal{P}_{\omega}(\overline{W}),$ the continuous nearest point projection of property [S3] gives rise to a continuous map: 
\begin{equation} \label{eq:projection} \Phi_{\overline{W},x_{\omega}} = \rho_{\overline{W},x_{\omega}} \circ  \pi_{\mathcal{P}_{\omega}(\overline{W})}\co  \mathcal{P}_{\omega}(S) \ra F_{\overline{W},x_{\omega}}.
\end{equation}   The following theorem regarding the above projection is an extension of Theorem \ref{thm:sublinear}.

\begin{thm}[\cite{behrstockminsky} Theorem 3.5] \label{thm:locconstant} $\Phi_{\overline{W},x_{\omega}}$ restricted to $P_{\overline{W},x_{\omega}}$ is a projection onto the $F_{\overline{W},x_{\omega}} $ factor in its natural product structure, and $\Phi_{\overline{W},x_{\omega}}$ is locally constant on $ \mathcal{P}_{\omega}(S) \setminus P_{\overline{W},x_{\omega}}.$
\end{thm}

The following lemma shows that the sets $F_{\overline{W},x_{\omega}}$ can be used to study distance in $\mathcal{P}_{\omega}(S).$
\begin{lem}[\cite{behrstockminsky} Theorem 3.6] \label{lem:sepfiber}
$\forall x_{\omega} \ne y_{\omega} \in \mathcal{P}_{\omega}(S), \exists \overline{W} \in \eultra{S}$ such that $$\lim_{\omega} \frac{1}{s_{i}} d_{\mathcal{P}(S)}(P^{0}_{i},\q{\partial W_{i}}) < \infty,$$ with the property that $\pi_{\mathcal{P}_{\omega}(\overline{W})}(x_{\omega}) \ne \pi_{\mathcal{P}_{\omega}(\overline{W})}(y_{\omega}) \in F_{\overline{W},x_{\omega}}.$  
\end{lem}


The following corollary provides a sufficient condition for identifying when two sequences represent the same point in the asymptotic cone.  The proof follows immediately from Lemma \ref{lem:sepfiber} and property [S1] of Theorem \ref{thm:sublinear}.    
\begin{cor} \label{cor:sepbounded}
Let $(x_{i}),(y_{i})$ be sequences representing the points $x_{\omega},y_{\omega} \in \mathcal{P}(S),$ and assume for all $\overline{W} \in \eultra{S}$ that $\lim_{\omega}d_{\mathcal{C}(W_{i})}(x_{i},y_{i})$ is bounded.  Then $x_{\omega}=y_{\omega}.$ 
\end{cor}

\subsubsection{Jets}
 
In \cite{bkmm}, subsets of $\mathcal{P}_{\omega}(S)$ called jets are developed.  Jets are particular subsets of the asymptotic cone corresponding to sequences of geodesics in  the curve complexes of connected essential subsurfaces which give rise to separation properties in $\mathcal{P}_{\omega}(S).$

Fix $P,Q \in \mathcal{P}(S),$ $Y\in \e{S}$ a connected essential subsurface, and $\sigma$ a tight geodesic in $\mathcal{C}(Y)$ from an element of $\pi_{\mathcal{C}(Y)}(P)$ to an element of $\pi_{\mathcal{C}(Y)}(Q).$  If $g=[\alpha,\beta]$ is a subsegment of $\sigma,$ $(g, P,Q)$ is called a \emph{tight triple} supported in $Y$ with \emph{ambient geodesic} $\sigma.$  For  $(g, P,Q)$ a tight triple as above, we define the \emph{initial pants} of the   
triple, denoted $\iota(g, P,Q) \equiv \alpha \cup \pi_{\mathcal{P}(S \setminus \alpha)}(P).$  Similarly, we define the \emph{terminal pants} of the triple, denoted $\tau(g, P,Q) \equiv \beta \cup \pi_{\mathcal{P}(S \setminus \beta)}(Q).$  Then, we define the \emph{length} of a tight triple supported in $Y$ by $$\norm{g}=\norm{(g, P,Q)}_{Y} \equiv d_{\mathcal{P}(Y)}( \iota(g, P,Q), \tau(g, P,Q)).$$  

For $\overline{P},\overline{Q} \in \mathcal{P}^{\omega}(S)$ which have nontrivial ultralimits in $\mathcal{P}_{\omega}(S),$ a \emph{Jet} J, is a quadruple of ultrapowers $(\overline{g},\overline{Y},\overline{P},\overline{Q}),$ where $(g_{i},P,Q)$ are tight triples supported in $Y_{i}.$  Associated to our jet J with support $\overline{Y}$ we have an \emph{initial point} or \emph{basepoint} of our jet  $\iota(J)=\iota_{\omega}(\overline{g}, \overline{P},\overline{Q}) \in \mathcal{P}_{\omega}(S)$ with a representative ultrapower $ \iota(g_{i}, P,Q).$  Similarly, we a terminal point of our jet $\tau(J)=\tau_{\omega}(\overline{g}, \overline{P},\overline{Q}) \in \mathcal{P}_{\omega}(S)$ with a representative ultrapower $ \tau(g_{i}, P,Q).$  A jet is called \emph{macroscopic} if $\iota(J) \ne \tau(J)$ and \emph{microscopic} otherwise.  To simplify notation, we set $\norm{(g_{i}, P,Q)}_{Y_{i}} =\norm{g_{i}}_{J}.$  We will only consider microscopic jets.   

Let J be a microscopic jet with support $\overline{Y}$ and tight geodesics $g_{i}.$  Then we can consider the ultralimit $\mathcal{Q}_{\omega}(\overline{\iota} \cup \overline{ \partial Y})$ which can be though of as $\iota(J) \times \mathcal{P}_{\omega}(\overline{Y^{c}}) \subset \mathcal{P}_{\omega}(S).$  Then we can define an equivalence relation on $\mathcal{P}_{\omega}(S) \setminus \left(\iota(J) \times \mathcal{P}_{\omega}(\overline{Y^{c}})\right)$ given by: $$x_{\omega} \sim_{J} x'_{\omega} \iff  \lim_{\omega} d_{\mathcal{C}(Y_{i})}( \pi_{g_{i}}(x_{i}), \pi_{g_{i}}(x'_{i})) < \infty.$$  
The following theorems regarding the existence and separation properties of microscopic jets will have application in Section \ref{chap:treegraded}.  

\begin{thm}[\cite{bkmm} Lemma 7.5] \label{thm:mjetsexist}
Let $a_{\omega},b_{\omega} \in \mathcal{P}_{\omega}(S)$ with representatives $(a_{i}),(b_{i})$ respectively.  Assume that $\overline{W} \in \eultra{S}$ is such that $\lim_{\omega}d_{\mathcal{C}(W)}(a_{i}, b_{i} ) \ra \infty.$  Then there exists a microscopic jet $J=(\overline{g},\overline{W},\overline{a},\overline{b})$ such that  $a_{\omega} \not \sim_{J} b_{\omega}.$  Moreover, the subsegments $g_{i}$ can be constructed to be contained in tight $\mathcal{C}(W_{i})$ geodesic of a hierarchy between $a_{i}$ and $b_{i}.$   
\end{thm}

\begin{thm}[\cite{bkmm} Theorem 7.2] \label{thm:mjetssep} For $J$ a microscopic jet, each equivalence class under the relation $\sim_{J}$ is open.  In particular, $x_{\omega}, x'_{\omega} \in \mathcal{P}_{\omega}(S) \setminus \left(\iota(J) \times \mathcal{P}_{\omega}(\overline{Y^{c}})\right),$ $x_{\omega} \not \sim_{J} x'_{\omega}$ $\implies$   $x_{\omega}$ and $x'_{\omega}$ are separated by $\iota(J) \times \mathcal{P}_{\omega}(\overline{Y^{c}}) .$
\end{thm}

\section{Complex of Separating Multicurves}
\label{chap:separatingcomplex}
Along the lines of the curve complex and the pants complex, in this section we introduce and analyze another natural complex associated to a surface, namely the \emph{complex of separating multicurves}, or simply the \emph{separating complex}.  The separating complex, denoted $\mathbb{S}(S),$ can be thought of as a generalizations of the separating curve complex and the Torelli Complex.  Formally, we have the following definition: 

\begin{defn}[Separating complex] \label{defn:sepcomplex} Given a surface $S$ of finite type, define the \emph{separating complex}, denoted $\mathbb{S}(S),$ to have vertices corresponding to isotopy classes of \emph{separating multicurves} $C \subset \mathcal{C}(S),$ that is multicurves $C$ such that at least two connected components of $S\setminus C$ are essential subsurfaces.  More generally, the separating complex has $k$-cells corresponding to a sets of $(k+1)$ isotopy classes of separating multicurves the complement of whose union in the surface $S$ contains an essential subsurface.  As usual, we will be interested in the one skeleton of $\mathbb{S}(S)$ equipped with the graph metric.  

\begin{figure}[h]
\centering
\includegraphics[height=3 cm]{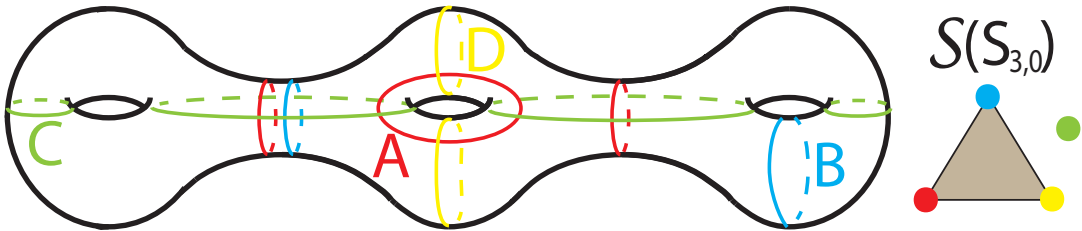}
\caption[Some vertices in $\mathbb{S}(S_{3,0})$]{The separating multicurve $A,B,D$ form a 2-simplex in $\mathbb{S}(S_{3,0}).$  The separating multicurve $C$ is an isolated point in $\mathbb{S}(S_{3,0}).$}\label{fig:sepcomplex}
\end{figure}
\end{defn}

Notice that a vertex in the separating complex representing a separating multicurve $C,$ corresponds to a natural quasi-convex product regions in the pants complex, $\q{C},$ defined in Equation \ref{eq:qcproduct}.  More generally, $k$-cells in the separating complex correspond to a set of $(k+1)$ quasi-convex product regions $\q{C_{0}},$..., $\q{C_{k}}$ such that the coarse intersection between the $k+1$ regions has infinite diameter.  Specifically, consider the multicurve $D=C_{0} \lrcorner C_{1} \lrcorner ... \lrcorner C_{k},$ and note that by Definition \ref{defn:sepcomplex} there is an essential subsurface $Y$ contained in the complement $S \setminus D.$  By Theorem \ref{thm:convex}, the coarse intersection between the product regions $\hat{\bigcap}_{i=0}^{k} \q{C_{i}}= \q{D},$ which in particular has infinite diameter as the complement $S \setminus D$ contains an essential subsurface.  This latter point of view motivates the definition of $\mathbb{S}(S).$   

\begin{rem} \label{rem:disjoint} Note that in Definition \ref{defn:sepcomplex} we did not require disjointness between separating multicurves corresponding to adjacent vertices.  If we let $\mathbb{S}'(S)$ denote a natural relative of our separating complex defined identically to $\mathbb{S}(S)$ in conjunction with an additional assumption of disjointness between representatives of adjacent vertices, then we have the following bi-Lipschitz relation: 
\begin{equation} \label{eq:bilipsep}
\forall C,D \in \mathbb{S}(S), \;\; d_{\mathbb{S}(S)}(C,D) \leq d_{\mathbb{S}'(S)}(C,D) \leq 2d_{\mathbb{S}(S)}(C,D).
\end{equation}
 The point is that while adjacent vertices $C,D \in \mathbb{S}(S)$ need not have disjoint separating multicurve representatives, by definition in the complement $S \setminus \{C,D\}$ there must exist a separating multicurve, $E.$  Then in $\mathbb{S}'(S)$ we have the connected sequence of vertices $C,E,D.$  As we will see, the complex $\mathbb{S}(S)$ is more natural from the point of view of Teichm\"uller space and in particular from the point of view of the asymptotic cones.  Nonetheless, there are situations in this section where for the sake of simplifying the exposition we will prove certain results using $\mathbb{S}'(S),$ and then note that the bi-Lipschitz Equation \ref{eq:bilipsep} implies related results for $\mathbb{S}(S).$ 
\end{rem}




As an immediate consequence of the definition of $\mathbb{S}'(S)$ in conjunction with Equation \ref{eq:bilipsep} we have the following inequality:
\begin{equation} \label{eq:coarsebelow} 
d_{\mathcal{C}(S)}(C,D) \leq d_{\mathbb{S}'(S)}(C,D) \leq 2d_{\mathbb{S}(S)}(C,D). 
\end{equation} 

Recall that in $\mathcal{C}(S),$ two curves are distance three or more if and only if they fill the surface.  Similarly, the following elementary lemma describes the implications of having $\mathbb{S}(S)$ distance at least four.

\begin{lem} \label{lem:smallsepcompdistance}
Let $C,D \in \mathbb{S}(S).$  $d_{\mathbb{S}(S)}(C,D)\geq 4$ implies that any connected essential subsurface of $S \setminus C$ overlaps any connected essential subsurface of $S\setminus D.$ 
\end{lem}
\begin{proof} Assume not, then there are connected essential subsurfaces $Z  \subseteq S \setminus C,$ $Z' \subseteq S \setminus D$ such that $Z$ and $Z'$ are identical, nested, or disjoint.  If $Z\subseteq Z'$ (or equivalently $Z'  \subseteq Z$)  then by definition, $d_{\mathcal{S}(S)}(C,D) \leq 1.$  Finally, if $Z \cap Z'  =\emptyset$ then $d_{\mathcal{S}(S)}(C,D) \leq 3,$ as in $\mathcal{S}(S)$ we have a connected path: $C, \partial Z, \partial Z' , D$ $\Ra\La.$
\end{proof}

In light of our definitions, the following lemma which will have application in Section \ref{chap:thick}.
 \begin{lem} \label{lem:intpoint}
 Let $\overline{W},\overline{V} \in \seultra{S}$ such that \uas $d_{\mathbb{S}(S)}(\partial W_{i},\partial V_{i}) \geq 2.$  Then 
 $$ \Phi_{\overline{W},x_{\omega}} (\qbar{\partial V}) = \{ pt\}, \;\;  \Phi_{\overline{V},y_{\omega}} (\qbar{\partial W} ) = \{ pt \} , $$ where  $\Phi_{\overline{W},x_{\omega}}$ is the projection defined in Equation \ref{eq:projection}.
 \end{lem}
 %

 \begin{proof}
Recall the definition of $ \Phi_{\overline{W},x_{\omega}} = \rho_{\overline{W},x_{\omega}} \circ \pi_{\mathcal{P}(\overline{W})}.$  By assumption, the complement in the surface $S$ of  $\partial W_{i} \cup \partial V_{i}$ \uas does not contain an essential subsurface.  Hence, it follows that $\pi_{\mathcal{P}(\overline{W})}(\qbar{\partial V}) = \{pt\},$ as for any $a_{\omega} \in \qbar{\partial V}$ we can choose a representative $(a_{i})$ of $a_{\omega}$ which  \uas contains $\partial V_{i}.$  Thus, the projection to $\mathcal{P}(W_{i})$ is coarsely entirely determined by the projection of the curve $\partial V_{i}.$
 \end{proof}

\subsection{Separating complex of $S_{2,1}$}
\label{sec:sepcomplexS21}

\subsection{Connected components of $\mathbb{S}(S_{2,1})$ and Point Pushing}
\label{subsec:ccompsS21}
In this subsection, we consider the connected components of $\mathbb{S}(S_{2,1}),$ which will be of interest in Section \ref{chap:thick}.  By Remark \ref{rem:disjoint} the connected components of $\mathbb{S}'(S)$ and $\mathbb{S}(S)$ are equivalent, and hence for the sake of simplifying the exposition, in this section we will in fact consider the connected components of $\mathbb{S}'(S_{2,1}).$  By topological considerations, $\mathbb{S}'(S_{2,1})$ consists of separating curves or disjoint pairs thereof.  Hence, vertices of $\mathbb{S}'(S_{2,1})$ and simplices of $\mathcal{C}_{sep}(S_{2,1})$ are in correspondence.  Moreover, vertices in $\mathbb{S}'(S_{2,1})$ are adjacent if and only if the corresponding simplices are adjacent in $\mathcal{C}_{sep}(S_{2,1}).$  Thus, the connected components of $\mathbb{S}'(S_{2,1}),$ or equivalently $\mathbb{S}(S_{2,1}),$ are precisely the connected components of $\mathcal{C}_{sep}(S_{2,1}).$  

To study the connected components of $\mathcal{C}_{sep}(S_{2,1}),$ we begin by considering the projection $ \pi_{\mathcal{C}(S_{2,0})}= \pi_{\mathcal{C}(S_{2,0})}\co \mathcal{C}(S_{2,1}) \ra \mathcal{C}(S_{2,0})$ given by forgetting about the boundary component.  Up to homeomorphism there is only one separating curve on the surfaces $S_{2,1}$ and $S_{2,0}.$  In fact under the projection $\pi_{\mathcal{C}(S_{2,0})}$ the image of a separating curve is a separating curve, and similarly the preimage of a separating curve is a union of separating curves. 

\begin{lem} \label{lemsurj} The map $\pi_{\mathcal{C}(S_{2,0})}=\pi_{\mathcal{C}}$ has a natural well-defined surjective restriction 
$$\pi_{\mathcal{C}_{sep}(S_{2,0})}=\pi_{\mathcal{C}_{sep}}\co \mathcal{C}_{sep}(S_{2,1}) \ra \mathcal{C}_{sep}(S_{2,0}).$$
\end{lem}

\begin{lem} \label{lem:connected} The fibers of $\pi_{\mathcal{C}_{sep}}$ are connected.
\end{lem}

\begin{proof}
Consider two separating curves $\alpha \ne \beta \in \pi_{\mathcal{C}_{sep}}^{-1}(\gamma).$  If $\alpha$ and $\beta$ are disjoint, we are done.  If not, we will complete the proof by induction on the number of intersections between the curves $\alpha$ and $\beta.$  Look for an innermost bigon $B$ formed by the union of $\alpha$ and $\beta,$ namely a bigon with two vertices given by intersection points of the curves and such that neither of the curves enters the interior of the bigon.  By topological considerations such a bigon must exist.  We can assume that the boundary component of the surface is included in the bigon $B.$  If not, up to a choice of representatives of our curves $\alpha$ and $\beta$ we reduce the intersection number.  

Then we can perform a surgery on $\alpha$ along the bigon $B$ to create the curve $\alpha',$ as in Figure \ref{fig:bigon}.  We can assume that $\alpha'$ is nontrivial, for if not then our original curve $\gamma \in \mathcal{C}_{sep}(S_{2,0})$ would be trivial $\Ra \La.$  Moreover, it is also clear that $\alpha' \in \pi_{\mathcal{C}_{sep}}^{-1}(\gamma).$  Replacing our original curve $\alpha$ with $\alpha'$ reduces the intersection number by two, thereby completing the proof by induction.

\begin{figure}[htpb]
\centering
\includegraphics[height=4 cm]{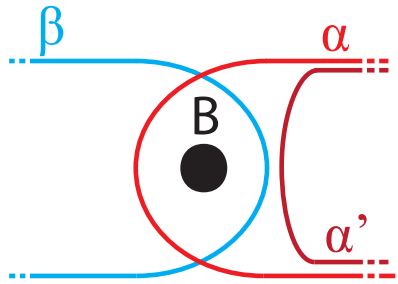}
\caption{Performing surgery to a curve along a bigon to reduce intersection numbers.  }\label{fig:bigon}
\end{figure}

\end{proof}

\begin{lem} \label{lem:coincide}The fibers of $\pi_{\mathcal{C}_{sep}}$ coincide with the connected components of $\mathcal{C}_{sep}(S_{2,1}).$  In particular, since there are infinitely many curves in the range, $\mathcal{C}_{sep}(S_{2,0}),$ it follows that there are infinitely many fibers, and hence infinitely many connected components of $\mathcal{C}_{sep}(S_{2,1}).$
 \end{lem}

\begin{proof}
Since Lemma \ref{lem:connected} ensures that any fiber of $\pi_{\mathcal{C}_{sep}}$ is connected, to prove the lemma it suffices to show that any two curves $\alpha, \beta $ which can be connected in $\mathcal{C}_{sep}(S_{2,1})$ must satisfy $\pi_{\mathcal{C}}(\alpha)=\pi_{\mathcal{C}}(\beta).$  Without loss of generality we can assume that $\alpha \cap \beta = \emptyset.$  Ignoring the boundary component, we have disjoint representatives of $\pi_{\mathcal{C}}(\alpha),$ and $\pi_{\mathcal{C}}(\beta).$  However, there are no distinct isotopy classes of separating curves in $S_{2,0}$ $\implies  \pi_{\mathcal{C}}(\alpha)=\pi_{\mathcal{C}}(\beta).$   
\end{proof}

\label{subsec:ppush}
The \emph{point pushing subgroup} is an important subgroup of the mapping class group of a surface with boundary first considered by Birman, \cite{birmanppush}.  Specifically, for $S_{g,n+1}$ with a fixed boundary component labeled $x,$ such that if we fill in the boundary component $x$ we obtain a topological $S_{g,n}$ with a marked base point $x,$ we have the following short exact sequence:
$$ 1\ra \pi_{1}(S_{g,n},x) \inj \MCG(S_{g,n+1}) \surj \MCG(S_{g,n}) \ra 1.$$
The second map is defined by taking a homeomorphism  of $S_{g,n+1}$ and viewing it as a homeomorphism of the surface $S_{g,n}$ obtained by filling in the boundary component $x.$  On the other hand, the first map is give by ``point pushing.''  Specifically, given a loop $\gamma \in  \pi_{1}(S_{g,n}, x),$ the image of the point pushing map of $\gamma,$ denoted $Push_{\gamma},$ is defined to be $T_{\gamma+\epsilon} \circ T_{\gamma - \epsilon}^{-1} \in \MCG(S_{g,n+1})$ where $\gamma + \epsilon$ and $\gamma - \epsilon$ are the two homotopically distinct push-offs of $\gamma$ in $S_{g,n+1}.$  The point pushing subgroup of the mapping class group is defined to be the group generated by point pushing maps for all loops $\gamma \in  \pi_{1}(S_{g,n}, x).$  See Figure \ref{fig:ppush} for examples.

\begin{figure}
  \centering
  \subfloat{\label{fig:ppusha}\includegraphics[width=0.4\textwidth]{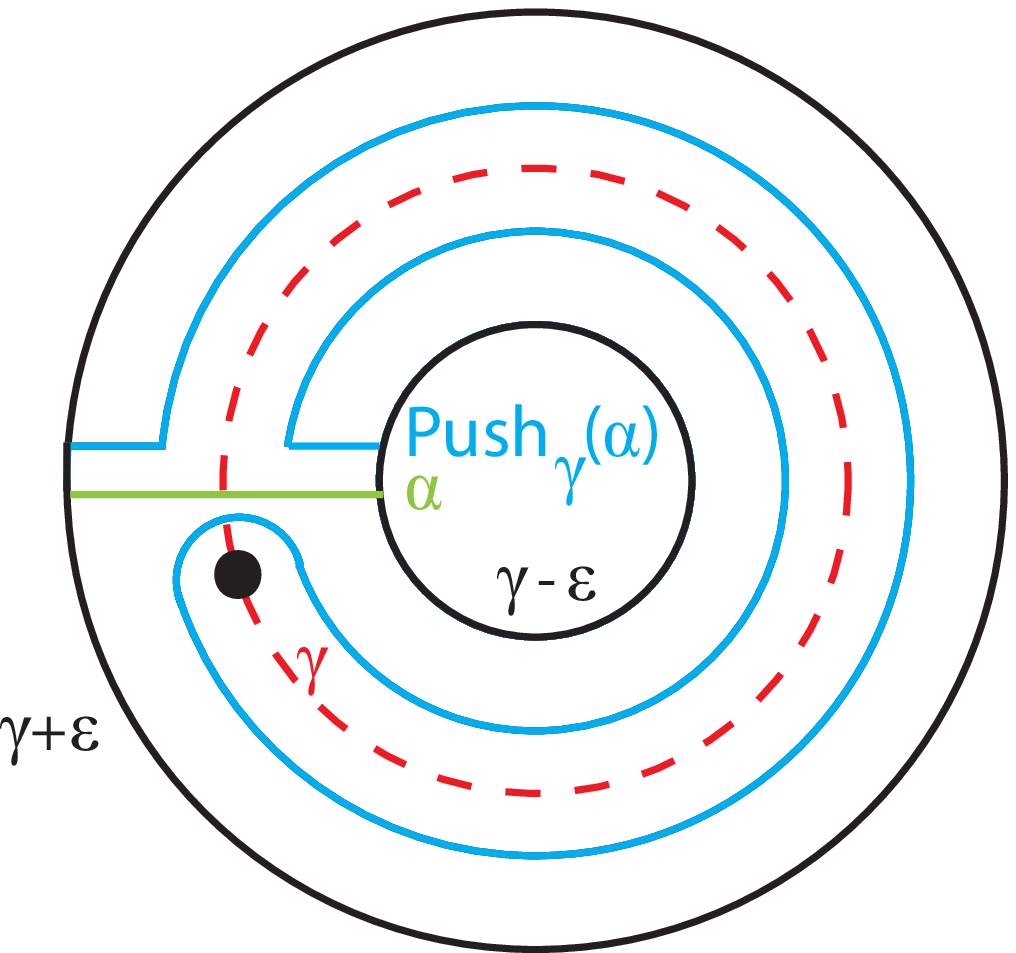}}
  ~ 
  \subfloat{\label{fig:pushdense}\includegraphics[width=0.5\textwidth]{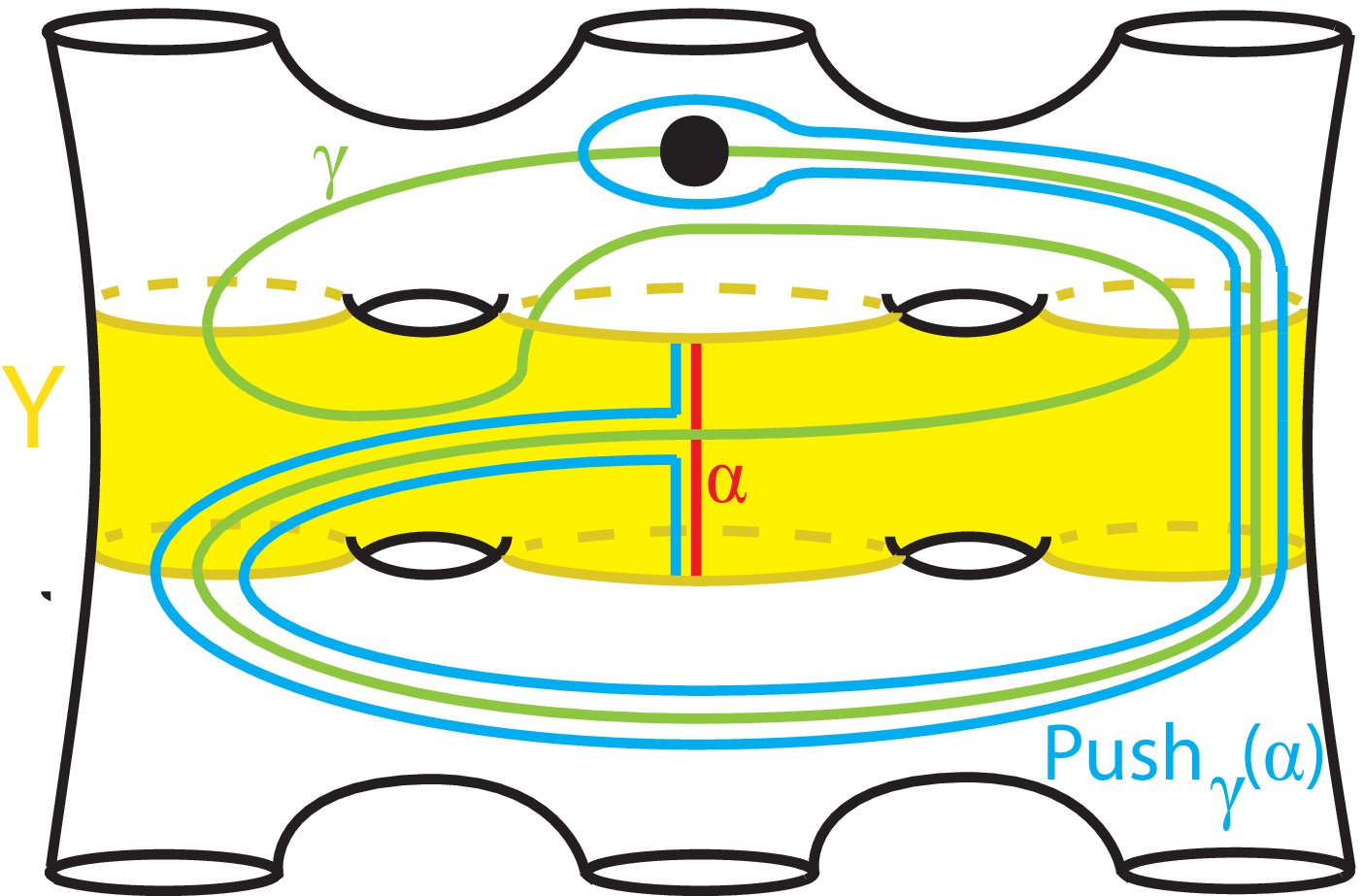}}
  ~ 
  \caption{The point pushing map applied to an arcs $\alpha \subset S.$}
  \label{fig:ppush}
\end{figure}


By construction, the image of this point pushing map is in the kernel of the projection $p:\MCG(S_{g,n+1})  \surj MCG(S_{g,n})$ as the curves $\gamma + \epsilon$ and $\gamma - \epsilon$ viewed in the surface $S_{g,n}$ are the same up to homotopy.  Specifically, since $p$ is a homomorphism we have $p( T_{\gamma+\epsilon} \circ T_{\gamma - \epsilon} ) =p( T_{\gamma+\epsilon}) \circ p( T_{\gamma - \epsilon}^{-1} ) = T_{\gamma}T_{\gamma}^{-1} =Id \in \MCG(S_{g,n}).$  We have just shown the following:

\begin{lem} \label{lem:preserve}
The point pushing subgroup $Push \subset \MCG(S_{2,1})$ preserves the connected components of $\mathcal{C}_{sep}(S_{2,1}).$  Similarly, $Push \subset \MCG(S_{2,1})$ preserves the fibers of the projection $\pi_{\mathcal{P}}\co \mathcal{P}(S_{2,1}) \ra \mathcal{P}(S_{2,0}).$  \end{lem}

Since there exist pseudo-Anosov point pushing maps, \cite{kra}, and because pseudo-Anosov axes have infinite diameter in $\mathcal{C}(S)$ \cite{mm1}, which in particular ensures that the axes have infinite diameter in $\mathcal{C}_{sep}(S),$ by Lemma \ref{lem:preserve} it follows that the connected components of $\mathcal{C}_{sep}(S_{2,1})$ have infinite diameter.   Putting together Lemmas \ref{lem:coincide} and \ref{lem:preserve}, we have the following corollary which uniquely characterizes the surface $S_{2,1}$ and which is the underlying reason for the unique phenomenon regarding the thickness and divergence of $\mathcal{T}(S_{2,1})$ studied in Section \ref{chap:thick}.

\begin{cor} \label{cor:unique} $\mathcal{C}_{sep}(S_{2,1}),$ and similarly $\mathbb{S}(S_{2,1}),$ has infinitely many connected connected components, each with infinite diameter. 
\end{cor}

\subsection{$\mathbb{S}_{\omega}(S),$ the ultralimit of $\mathbb{S}(S)$}
\label{sec:sepcomplexcone}

Throughout this section we assume a fixed asymptotic cone $\mathcal{P}_{\omega}(S),$ and consider the ultralimit of $\mathbb{S}(S),$ which we denote $\mathbb{S}_{\omega}(S).$  Formally,  

\begin{defn}[$\mathbb{S}_{\omega}(S)$] Given a surface $S$ of finite type, define $\mathbb{S}_{\omega}(S)$ to have vertices corresponding to $\overline{C} \in \mathbb{S}(S)^{\omega}$ such that $\lim_{\omega} \frac{1}{s_{i}} d_{\mathcal{P}(S)}(P^{0}_{i},\q{C_{i}}) < \infty.$  Equivalently, vertices in $\mathbb{S}_{\omega}(S)$ correspond to natural convex nontrivial product regions $\qbar{C} \subset \mathcal{P}_{\omega}(S).$  By abuse of notation, we will sometimes interchange between these two equivalent descriptions of vertices in $\mathbb{S}_{\omega}(S).$  Furthermore, define $\mathbb{S}_{\omega}(S)$ to have an edge between vertices $\qbar{C}$ and $\qbar{D}$ if in the asymptotic cone $\qbar{C \lrcorner D}= \qbar{D \lrcorner C},$ and moreover \uas the complement $S \setminus \{ C_{i},D_{i}\}$ contains an essential subsurface $Y_{i}.$  By Theorem \ref{thm:convex} this is equivalent to the statement that the intersection between the convex product regions, $\qbar{C} \cap \qbar{D},$ has nontrivial (in fact infinite) diameter in the asymptotic cone.  We can define higher dimensional simplices similarly, although they will not be necessary as we will only be interested in the one skeleton of $\mathbb{S}_{\omega}(S)$ equipped with the graph metric.   
\end{defn}

Given our definition of $\mathbb{S}_{\omega}(S),$ we can define a related $[0,\infty]$--valued pseudometric on the asymptotic cone which gives information about the natural product structures connecting points in the asymptotic cone.  Specifically, define
$$ d_{\mathbb{S}_{\omega}(S)}(a_{\omega},b_{\omega}) \equiv \inf_{\overline{A},\overline{B}}d_{\mathbb{S}_{\omega}(S)}(\overline{A}, \overline{B}) $$ where the infimum is taken over all pairs $\overline{A}, \overline{B}$ in the vertex set of $\mathbb{S}_{\omega}(S)$ having the property that $a_{\omega} \in \qbar{A} $ and $b_{\omega} \in \qbar{B}.$ 

This definition is well-defined, as given any pants decompositions $P\in \mathcal{P}(S)$ there is a bound $D(S)$ depending only on the topological type of the surface $S,$ such that there exists a pants decomposition $P' \in \mathcal{P}(S)$ containing a separating curve and $d_{\mathcal{P}(S)}(P,P')\leq D(S).$  In particular, given any element of the asymptotic cone $a_{\omega}$ with any representative $(A_{i})$ there exists an alternative representative, $(A'_{i}),$ with $A'_{i}$ containing a separating curve, thus making it clear that $a_{\omega}$ lies in some natural convex product region of the asymptotic cone.  The following theorem ensures appropriate compatibility of $\mathbb{S}(S)$ and $\mathbb{S}_{\omega}(S).$

\begin{thm} \label{thm:sepcomplexrelasymptotic}
Let $\overline{C}, \overline{D}$ be vertices in $\mathbb{S}_{\omega}(S).$ Then we have the following inequality:  
$$d_{\mathbb{S}_{\omega}(S)}(\overline{C},\overline{D})   \leq 2\lim_{\omega}d_{\mathbb{S}(S)}(C_{i},D_{i})    \leq 2d_{\mathbb{S}_{\omega}(S)}(\overline{C},\overline{D}). $$ Moreover, when $d_{\mathbb{S}_{\omega}(S)}(\overline{C},\overline{D})$ is finite yet nontrivial, for each of the finite number of natural convex product regions $\qbar{A} \subset \mathcal{P}_{\omega}(S)$ traveled through in the path between $\qbar{C}$ and $\qbar{D},$ the separating curve $A_{i}$ is \uas in the same connected components as the finite $\mathbb{S}(S)$ geodesic from $C_{i}$ to $D_{i}.$  
\end{thm} 

\begin{rem}
The multiplicative term of $2$ in the bi-Lipschitz inequality of Theorem \ref{thm:sepcomplexrelasymptotic} is not believed to be necessary, although is used for technical aspects in the proof.
\end{rem}

\begin{proof}[Proof of Theorem \ref{thm:sepcomplexrelasymptotic}]
First we will prove $\lim_{\omega}d_{\mathbb{S}(S)}(C_{i},D_{i})  \leq d_{\mathbb{S}_{\omega}(S)}(\overline{C},\overline{D}).$  It suffices to assume that $d_{\mathbb{S}_{\omega}(S)}(\overline{C},\overline{D})=1$ and show that  $\lim_{\omega}d_{\mathbb{S}(S)}(C_{i},D_{i})    \leq 1.$  Since $$d_{\mathbb{S}_{\omega}(S)}(\overline{C},\overline{D})=1$$ it follows that in the asymptotic cone, the natural convex product regions $\qbar{C},\qbar{D}$ whose intersection is $\qbar{C \lrcorner D}=\qbar{D \lrcorner C} $ is an infinite diameter set.  In particular, $S \setminus (C_{i} \cup D_{i})$ \uas contains an essential subsurface, $Y_{i}.$  Accordingly, in $\mathbb{S}(S)$ \uas we have a connected chain $C_{i},D_{i}$ thus proving $\lim_{\omega}d_{\mathbb{S}(S)}(C_{i},D_{i}) \leq 1$ as desired.  

In order to complete the proof we will show $d_{\mathbb{S}_{\omega}(S)}(\overline{C},\overline{D}) \leq 2 \lim_{\omega}d_{\mathbb{S}(S)}(C_{i},D_{i}).$  Considering the first part of the proof we can assume $\lim_{\omega}d_{\mathbb{S}(S)}(C_{i},D_{i})$ is finite, which by Lemma \ref{lem:finite}  implies that  \uas $d_{\mathbb{S}(S)}(C_{i},D_{i})=n$ for some non-negative constant $n.$  By Remark \ref{rem:disjoint}, it follows that \uas $d_{\mathbb{S}'(S)}(C_{i},D_{i}) = n' \leq 2n.$  Hence, \uas we have a finite $\mathbb{S}'(S)$ geodesic $\overline{C}=\overline{C^{0}}, ..., \overline{C^{n'}}=\overline{D}.$ Since \uas $C^{j}_{i} \cap C^{j+1}_{i}$ are disjoint, it follows that \uas $C^{j}_{i} \lrcorner C^{j+1}_{i}=C^{j+1}_{i} \lrcorner C^{j}_{i}.$   

Putting things together, in order to prove that $d_{\mathbb{S}_{\omega}(S)}(\overline{C},\overline{D}) \leq 2n,$ and hence complete the proof of the lemma, it suffices to show that there are natural convex product regions $\qbar{C^{j}} \subset \mathcal{P}_{\omega}(S)$ in the asymptotic cone for $ j \in \{1,...,n'-1\}$  corresponding to the terms in the sequence of $\mathbb{S}(S)$ geodesics $C^{0}_{i},...,C^{n'}_{i}.$  Equivalently, it suffices to show that $\lim_{\omega} \frac{1}{s_{i}} d_{\mathcal{P}(S)}(P^{0}_{i},\q{C^{j}_{i}}) < \infty$ for all $ j \in \{1,...,n'-1\}$ (by the assumptions of our lemma we already have this for $j=0,n'$).  Once we show this, we will have the following chain of natural convex product regions in the asymptotic cone with each product region intersecting its neighbor in an infinite diameter set:
$$\qbar{C}=\qbar{C^{0}} ,...,\qbar{C^{n'}}=\qbar{D}.$$     

\begin{figure}[htpb]
\centering
\includegraphics[height=3 cm]{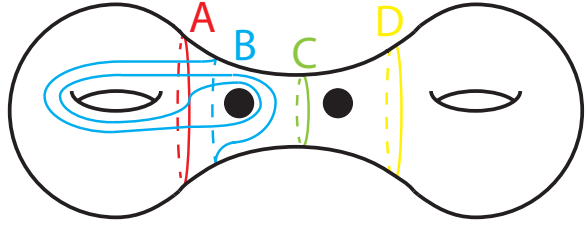}
\caption[Tightening a $\mathbb{S}'(S)$ geodesics.] {$d_{\mathbb{S}'(S_{2,2})}(A,B)=2$ and in fact the sequence $A,D,B$ is a $\mathbb{S}'(S)$ geodesic.  Note that $D \not \subset \mathcal{N}(A \cup B).$  However, replacing $D$ by $C,$ we have a new $\mathbb{S}'(S)$ geodesic $A,C,B$ with $C \subset \mathcal{N}(A \cup B).$  This replacement process is akin to tightening in Subsection \ref{subsec:tight}.}\label{fig:chainsepcomplex}
\end{figure}  

Fix some $\overline{C^{j}},$ for  $ j \in \{1,...,n'-1\}.$  By replacement if necessary, we can assume $C^{j}_{i}$ is contained in a regular neighborhood of  $C^{j-1}_{i}$ and $C^{j+1}_{i}.$  We denote this latter condition by $C^{j}_{i} \subset \mathcal{N}(C^{j-1}_{i} \cup C^{j+1}_{i}).$  See Figure \ref{fig:chainsepcomplex} for an example of such a replacement.  We will show that  $\lim_{\omega} \frac{1}{s_{i}} d_{\mathcal{P}(S)}(P^{0}_{i},\q{C^{j}_{i}}) < \infty.$  Then, iteratively repeating the same argument for each of two the resulting shorter sequences $\overline{C^{0}}, ..., \overline{C^{j}} \mbox{ and } \overline{C^{j}}, ..., \overline{C^{n'}},$ we eventually obtain an entire chain of length $n'$ with the desired property, namely $\lim_{\omega} \frac{1}{s_{i}} d_{\mathcal{P}(S)}(P^{0}_{i},\q{C^{j}_{i}}) < \infty$ for all $ j \in \{1,...,n'-1\}.$  

In order to show that $\lim_{\omega} \frac{1}{s_{i}} d_{\mathcal{P}(S)}(P^{0}_{i},\q{C^{j}_{i}}) < \infty$ we will show:
\begin{eqnarray}
\label{eq:suffice} d_{\mathcal{P}(S)}(P^{0}_{i},\q{C^{j}_{i}}) \lesssim d_{\mathcal{P}(S)}(P^{0}_{i},\q{C^{0}_{i}}) + d_{\mathcal{P}(S)}(P^{0}_{i},\q{C^{n'}_{i}}).
\end{eqnarray}
Then our assumption of $\lim_{\omega} \frac{1}{s_{i}}  \left( d_{\mathcal{P}(S)}(P^{0}_{i},C^{0}_{i}) \\ + d_{\mathcal{P}(S)}(P^{0}_{i},C^{n'}_{i}) \right) < \infty,$ in conjunction with Equation \ref{eq:suffice} completes the proof of the theorem.      

In order to prove equation \ref{eq:suffice}, by Lemma 2.2 of \cite{behrstockminsky} it suffices to show that for any connected essential subsurface $Y\in \e{S}$ such that $Y \cap C^{j}_{i},$  $$d_{\mathcal{C}(Y)}(P^{0}_{i},C^{j}_{i})  \leq d_{\mathcal{C}(Y)}(P^{0}_{i},\{C^{0}_{i},C^{n'}_{i}\}) +n'r'$$ where $r'$ is some constant.  First assume that $Y$ intersects $C^{m}_{i}$ \uas for all $m \in \{0,..,j-1\}.$  In this case we are done as by Lemma \ref{lem:curveprojub} it follows that $d_{\mathcal{C}(Y)}(C^{0}_{i},C^{j}_{i}) \leq  n'r'.$  Similarly, we are done if $Y$ intersects $C^{m}_{i}$ \uas for all $m \in \{j+1,...,n'\}.$  Since $\{C^{j}_{i}\} _{j_{0}}^{n'}$ is a geodesic in $\mathbb{S}'(S),$ it follows that if $Y$ is \uas disjoint from $C^{k}_{i}$ then $Y$ intersects all $C^{l}_{i}$ for all $l$ such that $|l-k|\geq 3.$  Since by assumption $Y \cap C^{j} \ne \emptyset$ and because $C^{j}_{i} \subset \mathcal{N}(C^{j-1}_{i} \cup C^{j+1}_{i}),$ it follows that either $Y \cap C^{j-1} \ne \emptyset$ or $Y \cap C^{j+1} \ne \emptyset.$  In other words, any connected essential subsurface $Y$ which intersects $C^{j}_{i}$ actually intersects two consecutive separating multicurves: either $C^{j-1}_{i},C^{j}_{i}$ or $C^{j}_{i},C^{j+1}_{i}.$  In either case, it follows that $Y$ must \uas  intersect $C^{m}_{i}$ either for all $m \in \{0,...,j-1\}$ or for all $m \in \{j+1,...,n'\},$ thereby completing the proof.
\end{proof}

The bi-Lipschitz relation in Theorem \ref{thm:sepcomplexrelasymptotic} guarantees that one of the terms is infinite if and only if the other term is infinite.  It should be stressed that the term $\lim_{\omega}d_{\mathbb{S}(S)}(C_{i},D_{i})$ can be infinite due to two different reasons.  It is possible that  \uas $C_{i}$ and $D_{i}$ are connected in $\mathbb{S}(S)$ however their distances are unbounded.  On the other hand, for small enough complexity surfaces, it is possible that  \uas $C_{i}$ and $D_{i}$ are in different connected components of $\mathbb{S}(S).$  This distinction will be crucial in Section \ref{chap:treegraded}.  

\section{Asymptotic cone of Teichm\"uller space}
\label{chap:treegraded}
In this section we explore the asymptotic cone of Teichm\"uller space.  The first subsection introduces a notion called structurally integral corners which provide a desired separation property in the asymptotic cone.  The second subsection characterizes when two points in the asymptotic cone of Teichm\"uller space are separated by a cut-point.  Finally, in subsection \ref{sec:contracting} we characterize strongly contracting quasi-geodesics in Teichm\"uller space.

\subsection{Structurally integral corners}
\label{sec:corners}

\subsubsection{Structurally integral corners are well-defined}
Informally, a structurally integral corner entails the joining of two particular natural convex product regions in the asymptotic cone of the pants complex at a ``corner'' such that the removal of the corner joining the regions separates the two product regions from each other.  More formally, fixing some ultrafilter $\omega,$ we have the following definition:

\begin{defn}[structurally integral corner] \label{defn:corner}
Let $\overline{\alpha} \ne \overline{\beta} \in \mathbb{S}^{\omega}$ be such that the following conditions hold:
\begin{enumerate}
\item \uas $\alpha_{i}$ and $\beta_{i}$ are in different connected components of $\mathbb{S}(S).$  In particular, it follows that $\lim_{\omega}d_{\mathbb{S}(S)}(\alpha_{i} , \beta_{i} ) \ra \infty$ and $\alpha_{i} \lrcorner \beta_{i}, \beta_{i} \lrcorner \alpha_{i} \in \mathcal{P}(S).$  And, 
\item $\lim_{\omega}d_{\mathcal{P}(S)}(\alpha_{i} \lrcorner \beta_{i},\beta_{i} \lrcorner \alpha_{i} )$ is bounded.  In particular, for any $\overline{Y} \in \eultra{S},$ the limit $$\lim_{\omega} d_{\mathcal{C}(Y_{i})}(\alpha_{i} \lrcorner \beta_{i}, \beta_{i} \lrcorner \alpha_{i}) \mbox{ is bounded} .$$
\end{enumerate}
In this setting we call the point $(\alpha \lrcorner \beta)^{\omega}$ (or equivalently the point  $(\beta \lrcorner \alpha)^{\omega}$) a \emph{structurally integral corner}, and denote it by $_{\overline{\alpha}}C_{\overline{\beta}}.$  
\end{defn}

\begin{rem} \label{rem:nooccur1}
It should be stressed that due to condition (1) in Definition \ref{defn:corner}, structurally integral corners can only exist for surfaces $S$ with disconnected separating complexes, or equivalently for surfaces with $|\chi(S)| \leq 4,$ see \cite{sultanthesis} Theorem 3.1.1.  
\end{rem}


After descending from elements of ultrapowers to elements of the asymptotic cone, the structurally integral corners $(\alpha \lrcorner \beta)_{\omega}$ and $(\beta \lrcorner \alpha)_{\omega}$ will be identified and moreover, this point will serve as a cut-point between the quasi-convex product regions $\qbar{\alpha}$ and $ \qbar{\beta}.$  We must assume that our cone $\mathcal{P}_{\omega}(S)$ contains the corner $(\alpha \lrcorner \beta)_{\omega},$ or equivalently we must assume $\lim_{\omega} \frac{1}{s_{i}} d_{\mathcal{P}(S)}(P^{0}_{i},\alpha_{i} \lrcorner \beta_{i}) < \infty.$ 


\begin{ex}[A structurally integral corner in $\mathcal{P}_{\omega}(S_{2,1})$] Let $\alpha_{i}, \beta_{i} \in \mathcal{C}_{sep}(S_{2,1})$ be such that $\lim_{\omega} \frac{1}{s_{i}} d_{\mathcal{P}(S)}(P^{0}_{i},\q{\alpha_{i}} < \infty,$ $\lim_{\omega} \frac{1}{s_{i}} d_{\mathcal{P}(S)}(P^{0}_{i},\q{\beta_{i}} < \infty.$  Moreover, assume that 
\uas (i) the intersection number $i(\alpha_{i},\beta_{i})$ is bounded, and (ii) $\alpha_{i},\beta_{i}$ are in different connected components of $\mathcal{C}_{sep}(S).$  In this case $_{\overline{\alpha}}C_{\overline{\beta}}$ is a structurally integral corner in $\mathcal{P}_{\omega}(S_{2,1}).$  The only nontrivial point to note is that the bound on the intersection number between $\alpha_{i}$ and $\beta_{i}$ guarantees condition (2) of Definition \ref{defn:corner}.  
\end{ex}
  
Given the notion of a structurally integral corner, we will now introduce a relation $\sim_{{\overline{\alpha}},{\overline{\beta}}}$ on $\mathcal{P}^{\omega}(S)$ which descends to an equivalence relation on $\mathcal{P}{\omega}(S) \setminus _{\overline{\alpha}}C_{\overline{\beta}}.$  Moreover, each equivalence class is open.  In particular, it will follow that in the asymptotic cone, $\mathcal{P}_{\omega}(S),$ the corner $_{\overline{\alpha}}C_{\overline{\beta}}$ is a cut-point between points of $\mathcal{P}_{\omega}(S) \setminus _{\overline{\alpha}}C_{\overline{\beta}}$ which are in different equivalence classes under the relation $\sim_{{\overline{\alpha}},{\overline{\beta}}}.$  We begin with the following definition of a relation $\sim_{{\overline{\alpha}},{\overline{\beta}}}$ on $\mathcal{P}^{\omega}(S).$
\begin{defn} \label{defn:equivrel} 
Let $_{\overline{\alpha}}C_{\overline{\beta}}$ be a structurally integral corner.  Then we have relation $\sim_{{\overline{\alpha}},{\overline{\beta}}}$ on $\mathcal{P}^{\omega}(S)$ given by saying $\overline{P} \sim_{{\overline{\alpha}},{\overline{\beta}}} \overline{Q}$ if and only if $\overline{P}$ and $\overline{Q}$ fall into the same case under the following trichotomy.  Namely, given $\overline{P},$  
\begin{enumerate}
\item $\overline{P}$ is in case one if $\exists \overline{W_{\alpha}} \in \seultra{S}$ such that the following two conditions hold:
\subitem (i) $\lim_{\omega} d_{\mathbb{S}(S)}(\alpha_{i},\partial W_{\alpha, i})$ is bounded, and
\subitem (ii) $\lim_{\omega} d_{\mathcal{C}(W_{\alpha, i})}(P_{i},\beta_{i}) \ra \infty.$ 
\item $\overline{P}$ is in case two if $\exists \overline{W_{\beta}}  \in \seultra{S}$ such that the following two conditions hold:
\subitem (i) $\lim_{\omega} d_{\mathbb{S}(S)}(\beta_{i},\partial W_{\beta, i})$ is bounded, and
\subitem (ii) $\lim_{\omega} d_{\mathcal{C}(W_{\beta, i})}(P_{i},\alpha_{i}) \ra \infty.$ 
\item $\overline{P}$ is in case three if neither the conditions of case one nor case two apply to $\overline{P}$
\end{enumerate}
\end{defn}

\begin{figure}[h]
\centering
\includegraphics[height=7 cm]{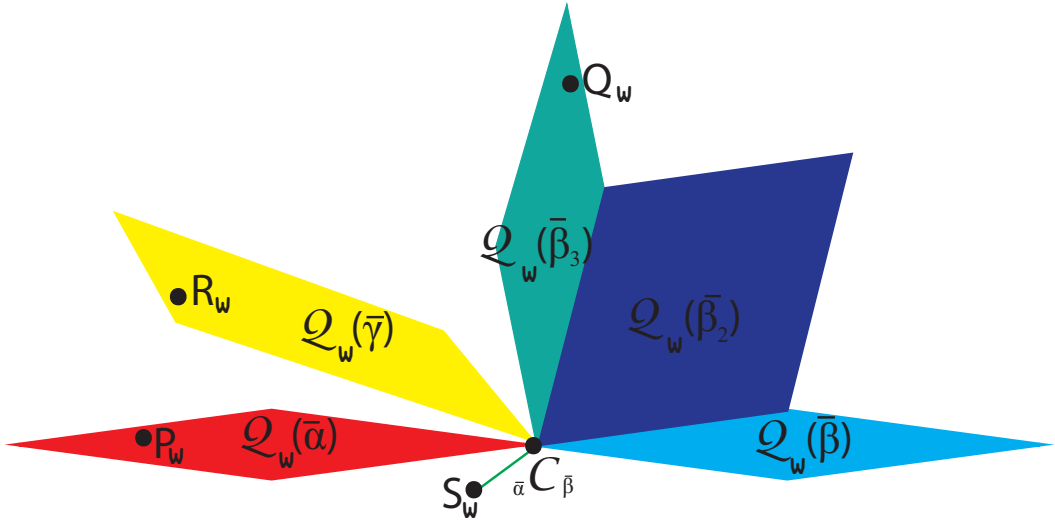}
\caption[A structurally integral corner]{A structurally integral corner $_{\overline{\alpha}}C_{\overline{\beta}} \in \mathcal{P}_{\omega}(S).$  $P_{\omega}$ is in case one of the equivalence relation $\sim_{{\overline{\alpha}},{\overline{\beta}}},$ $Q_{\omega}$ is in case two, and the points $R_{\omega},S_{\omega}$ are in case three.  In the picture we are assuming $d_{\mathbb{S}_{\omega}(S)}(\overline{\alpha},\{\overline{\beta}, \overline{\beta_{2}}, \overline{\beta_{3}} \} )=\infty,$ $d_{\mathbb{S}_{\omega}(S)}(\{\overline{\alpha},\overline{\beta}, \overline{\beta_{1}}, \overline{\beta_{2}} \} , \overline{\gamma})=\infty.$}\label{fig:corner}
\end{figure}

As a first order of business, the following lemma guarantees the mutual exclusivity of the three cases in the definition of  $\sim_{{\overline{\alpha}},{\overline{\beta}}},$ thus ensuring that the equivalence relation of Definition \ref{defn:equivrel} is well-defined.

\begin{lem} \label{lem:welldefined}
Let $\overline{P} \in \mathcal{P}^{\omega}(S).$  Then $\overline{P}$ falls into one and only one of the three cases in the trichotomy of Definition \ref{defn:equivrel}. 
\end{lem}

\begin{proof} It suffices to show that $\overline{P}$ cannot simultaneously be in cases one and two.  Assume not, that is, assume $\exists$ elements $\overline{W_{\alpha}},\overline{W_{\beta}} \in \seultra{S}$ such that $$\lim_{\omega} d_{\mathbb{S}(S)}(\alpha_{i},\partial W_{\alpha, i}) \mbox{ and } \lim_{\omega} d_{\mathbb{S}(S)}(\beta_{i},\partial W_{\beta, i})$$ are bounded (and by Remark \ref{rem:disjoint} similarly for $\mathbb{S}'(S)),$ while  $$\lim_{\omega} d_{\mathcal{C}(W_{\alpha, i})}(P_{i},\beta_{i}) \mbox{ and } \lim_{\omega} d_{\mathcal{C}(W_{\beta, i})}(P_{i},\alpha_{i})$$ are unbounded.

Since $_{\overline{\alpha}}C_{\overline{\beta}}$ is a structurally integral corner, in particular, we have that $\lim_{\omega} d_{\mathbb{S}(S)}(\alpha_{i}, \beta_{i})$ is unbounded, and consequently by our assumptions,  $\lim_{\omega} d_{\mathbb{S}(S)}(\partial W_{\alpha, i}, \partial W_{\beta, i})$ is unbounded as well.  Lemma \ref{lem:smallsepcompdistance} then guarantees  that $\overline{W_{\alpha}} \pitchfork \overline{W_{\beta}}.$  

By Lemma \ref{lem:curveprojub} if $Y_{i} \in \e{S}$ \uas intersects every separating multicurve in the bounded path of disjoint separating multicurves in $\mathbb{S}'(S)$ connecting $\beta_{i}$ and $\partial W_{\beta, i},$ then $$\lim_{\omega} d_{\mathcal{C}(Y_{i})}(\beta_{i},\partial W_{\beta, i})$$ is bounded as well.  In particular, since the distance in $\mathbb{S}'(S)$ between $\partial W_{\alpha, i}$ and the bounded path connecting $\beta_{i}$ and $\partial W_{\beta, i},$ is unbounded, Lemma \ref{lem:smallsepcompdistance} implies that \uas $\partial W_{\alpha, i}$ intersects every separating multicurve in the bounded path of separating multicurves in $\mathbb{S}'(S)$ connecting $\beta_{i}$ and $\partial W_{\beta, i}.$  Hence, $\lim_{\omega} d_{\mathcal{C}(W_{\alpha, i})}(\beta_{i},\partial W_{\beta, i})$ is bounded.  Similarly, $\lim_{\omega} d_{\mathcal{C}(W_{\beta, i})}(\alpha_{i},\partial W_{\alpha, i})$ is bounded.  In conjunction with our assumptions, it follows that $\lim_{\omega} d_{\mathcal{C}(W_{\alpha, i})}(P_{i},\partial W_{\beta, i})$ and $\lim_{\omega} d_{\mathcal{C}(W_{\beta, i})}(P_{i},\partial W_{\alpha, i})$ are unbounded.  Since $\overline{W_{\alpha}} \pitchfork \overline{W_{\beta}},$ this contradicts Lemma \ref{lem:projectionestimates}. 
\end{proof}

\subsubsection{Equivalence relation induced by structurally integral corners}

Having proven that the relation $\sim_{{\overline{\alpha}},{\overline{\beta}}}$ is well-defined, in this subsection we will prove that the relation in fact descends to an equivalence relation on $\mathcal{P}_{\omega}(S) \setminus _{\overline{\alpha}}C_{\overline{\beta}}.$

\begin{thm} \label{thm:equivrel} The relation $\sim_{{\overline{\alpha}},{\overline{\beta}}}$ descends to an equivalence relation on $\mathcal{P}_{\omega}(S) \setminus \; _{\overline{\alpha}}C_{\overline{\beta}}.$  Moreover, each equivalence class is open.
\end{thm}

The proof of Theorem \ref{thm:equivrel} will follow from the following technical lemma.
\begin{lem} \label{lem:equivrel} There exists a constant $C \geq 0$ such that for $_{\overline{\alpha}}C_{\overline{\beta}}$ a structurally integral corner if $\overline{P}, \overline{Q}$ are sequences representing points $P_{\omega},Q_{\omega} \in \mathcal{P}_{\omega}(S),$ and if $\overline{P} \not \sim_{{\overline{\alpha}},{\overline{\beta}}} \overline{Q}.$  Then, 
$$d_{\mathcal{P}_{\omega}(S)}(P_{\omega},Q_{\omega}) \geq Cd_{\mathcal{P}_{\omega}(S)}(P_{\omega}, _{\overline{\alpha}}C_{\overline{\beta}}).$$
\end{lem}

\begin{proof}[Proof of Theorem \ref{thm:equivrel}]
Assume that $\overline{P}$ and $\overline{Q}$ are representatives of the same point of the asymptotic cone.  Then by Lemma \ref{lem:equivrel} either $\overline{P} \sim_{{\overline{\alpha}},{\overline{\beta}}} \overline{Q}$ or in the asymptotic cone, $P_{\omega}= \; _{\overline{\alpha}}C_{\overline{\beta}}.$  Hence, the relation $\sim_{{\overline{\alpha}},{\overline{\beta}}}$ descends to a relation on $\mathcal{P}_{\omega}(S) \setminus _{\overline{\alpha}}C_{\overline{\beta}}$ which is reflexive.  Furthermore, since by definition it is immediate that $\sim_{{\overline{\alpha}},{\overline{\beta}}}$ is symmetric and transitive, it follows that $\sim_{{\overline{\alpha}},{\overline{\beta}}}$ descends to an equivalence relation on $\mathcal{P}_{\omega}(S) \setminus _{\overline{\alpha}}C_{\overline{\beta}}.$  Lemma \ref{lem:equivrel} implies that any point $P_{\omega} \in \mathcal{P}_{\omega}(S) \setminus _{\overline{\alpha}}C_{\overline{\beta}}$ has an open neighborhood consisting entirely of points which are in the same equivalence class.  Hence, the equivalence classes are open. 
\end{proof}

\begin{proof} [Proof of Lemma \ref{lem:equivrel}]  
$P_{i},Q_{i},\alpha_{i} \lrcorner \beta_{i}$ are pants decompositions of a surface and hence have nontrivial subsurface projection to any essential subsurface.  For any $Y \in \e{S},$ let $\sigma^{Y}_{i}$ be a $\mathcal{C}(Y)$ geodesic from $P_{i}$ to $Q_{i}.$  Moreover, let $\pi_{\sigma^{Y}_{i}}(\alpha_{i} \lrcorner \beta_{i})$ be the nearest point projection of $\pi_{\mathcal{C}(Y)}(\alpha_{i} \lrcorner \beta_{i})$ onto the geodesic $\sigma^{Y}_{i}.$  By definition, $\forall Y \in \e{S}$ we have 
\begin{equation} \label{eq:lemequivrel}
d_{\mathcal{C}(Y)}(P_{i},Q_{i}) \geq d_{\mathcal{C}(Y)}(P_{i},\pi_{\sigma^{Y}_{i}}(\alpha_{i} \lrcorner \beta_{i})).
\end{equation}

In order to complete the proof we will show that there is a uniform constant $k$ such that $\forall Y \in \e{S},$\begin{equation} \label{eq:lemequivrel2}
d_{\mathcal{C}(Y)}(\alpha_{i} \lrcorner \beta_{i},\pi_{\sigma^{Y}_{i}}(\alpha_{i} \lrcorner \beta_{i})) < k.
\end{equation}

Combining Equations \ref{eq:lemequivrel} and \ref{eq:lemequivrel2}, for all $Y \in \e{S},$ we have:
\begin{eqnarray} \label{eq:sufficestep2}
d_{\mathcal{C}(Y)}(P_{i},Q_{i}) \geq d_{\mathcal{C}(Y)}(P_{i},\alpha_{i} \lrcorner \beta_{i}) -k.
\end{eqnarray}
In particular, by Theorem \ref{thm:quasidistance}, in the asymptotic cone we have the following inequality thus completing the proof:
\begin{eqnarray} \label{eq:step2}
d_{\mathcal{P}_{\omega}(S)}(P_{\omega},Q_{\omega}) \geq Cd_{\mathcal{P}_{\omega}(S)}(P_{\omega},_{\omega}).   
\end{eqnarray}

By condition (2) in the definition of a structurally integral corner  $_{\overline{\alpha}}C_{\overline{\beta}}$ it follows that $\lim_{\omega}\mbox{diam}_{\mathcal{C}(Y)}(\{ \alpha_{i},\beta_{i}, \alpha_{i} \lrcorner \beta_{i}, \beta_{i} \lrcorner \alpha_{i} \} )$ is bounded, and hence, in place of Equation \ref{eq:lemequivrel2} it suffices to show that $\lim_{\omega}d_{\mathcal{C}(Y_{i})}(\sigma^{Y_{i}}_{i}, \{ \alpha_{i}, \beta_{i} \})$ is bounded.

By assumption $\overline{P}$ and $\overline{Q}$ are in different equivalence classes, and hence by definition $\overline{P}$ and $\overline{Q}$ fall into different cases in Definition \ref{defn:equivrel}.  By symmetry of the cases, without loss of generality we can assume that $\overline{P}$ is in case one of Definition \ref{defn:equivrel}, while $\overline{Q}$ is not.  Namely, $\exists \overline{W_{\alpha}} \in \seultra{S}$ such that $\lim_{\omega} d_{\mathbb{S}(S)}(\alpha_{i},\partial W_{\alpha, i})$ is bounded, while $\lim_{\omega} d_{\mathcal{C}(W_{\alpha, i})}(P_{i},\beta_{i}) \ra \infty.$  Furthermore, for any element $\overline{U} \in \seultra{S}$ such that $\lim_{\omega} d_{\mathbb{S}(S)}(\alpha_{i},\partial U_{i})$ is bounded, perforce $\lim_{\omega} d_{\mathcal{C}(U_{i})}(Q_{i},\beta_{i})$ is also bounded.  By Remark \ref{rem:disjoint} the same statements hold for $\mathbb{S}'(S).$

We proceed by considering cases for the relationship between $\overline{Y}$ and $\overline{W_{\alpha}}$ where $\overline{Y}$ is an arbitrary element of the ultrapower of connected essential subsurfaces.  By Lemma \ref{lem:finite} since there are only a finite number of possibilities for the relationship between two essential subsurfaces - identical, nested, overlapping, and disjoint -  the same finitely many possibilities for the relationship between $\overline{Y}$ and $\overline{W_{\alpha}}.$  In each case we will show $\lim_{\omega}d_{\mathcal{C}(Y_{i})}(\sigma^{Y_{i}}_{i}, \{ \alpha_{i}, \beta_{i} \})$ is bounded, thus completing the proof of the lemma.

\textbf{Case 1: Either $\overline{Y} \subset \overline{W_{\alpha}}$ or $\overline{Y} \cap \overline{W_{\alpha}} = \emptyset.$}  In either case, \uas $d_{\mathbb{S}(S)}(\partial W_{\alpha, i}, \partial Y_{i})\leq 1$ and hence by our assumptions $\lim_{\omega} d_{\mathbb{S}(S)}(\alpha_{i},\partial Y_{i})$ is bounded.  Since $\overline{Q}$ is not in case one of the equivalence relation $\sim_{{\overline{\alpha}},{\overline{\beta}}},$ it follows that $\lim_{\omega} d_{\mathcal{C}(Y_{i})}(Q_{i},\beta_{i})$ is bounded.  In particular, this implies that $\lim_{\omega}d_{\mathcal{C}(Y_{i})}(\sigma^{Y_{i}}_{i}, \{ \alpha_{i}, \beta_{i} \})$ is bounded, completing this case.   

\textbf{Case 2: $\overline{W_{\alpha}} \subset \overline{Y}$ and $\lim_{\omega} d_{\mathcal{C}(Y_{i})}(\partial W_{\alpha, i}, \{\alpha_{i},\beta_{i}\})$ is bounded.}  By our assumptions, $$\lim_{\omega} d_{\mathcal{C}(W_{\alpha, i})}(P_{i},\beta_{i}) \ra \infty,$$ while  $\lim_{\omega} d_{\mathcal{C}(W_{\alpha, i})}(Q_{i},\beta_{i})$ is bounded.  In particular, $\lim_{\omega} d_{\mathcal{C}(W_{\alpha, i})}(P_{i},Q_{i}) \ra \infty.$  Then \uas $d_{\mathcal{C}(Y_{i})}(\partial W_{\alpha, i},\sigma^{Y_{i}}_{i}) \leq 1.$  If not, then Theorem \ref{thm:projdefined} would imply that \uas $d_{\mathcal{C}(W_{\alpha, i})}(P_{i},Q_{i})$ is uniformly bounded which is a contradiction.  However, the assumption of the case that $\lim_{\omega} d_{\mathcal{C}(Y_{i})}(\partial W_{\alpha, i}, \{\alpha_{i},\beta_{i}\})$ is bounded then implies that  $\lim_{\omega} d_{\mathcal{C}(Y_{i})}(\{\alpha_{i},\beta_{i}\},\sigma^{Y_{i}}_{i})$ is bounded, thus completing this case.      
    
   
\textbf{Case 3: $\overline{Y} \pitchfork \overline{W_{\alpha}}$ and $\lim_{\omega} d_{\mathcal{C}(Y_{i})}(\partial W_{\alpha, i}, \{\alpha_{i},\beta_{i}\})$ is bounded.}  As in Case 2, by our assumptions $\lim_{\omega} d_{\mathcal{C}(W_{\alpha, i})}(P_{i},\beta_{i}) \ra \infty,$ while  $\lim_{\omega} d_{\mathcal{C}(W_{\alpha, i})}(Q_{i},\beta_{i})$ is bounded.  In particular, $$\lim_{\omega} d_{\mathcal{C}(W_{\alpha, i})}(P_{i},Q_{i}) \ra \infty.$$  Since \uas $W_{\alpha, i} \pitchfork Y_{i},$ it follows that $\lim_{\omega} d_{\mathcal{C}(Y_{i})}(\partial W_{\alpha, i},\{P_{i},Q_{i}\})$ is uniformly bounded.  If not,  then Lemma \ref{lem:projectionestimates} implies that $d_{\mathcal{C}(W_{\alpha, i})}(P_{i},Q_{i})$ is uniformly bounded which is a contradiction.  However, the assumption of the case that $$\lim_{\omega} d_{\mathcal{C}(Y_{i})}(\partial W_{\alpha, i}, \{\alpha_{i},\beta_{i}\})$$ is bounded then implies that  $\lim_{\omega} d_{\mathcal{C}(Y_{i})}(\{\alpha_{i},\beta_{i}\},\{P_{i},Q_{i}\})$ is bounded.  Since $\sigma^{Y_{i}}_{i}$ is $\mathcal{C}(Y_{i})$ geodesic between $P_{i}$ and $Q_{i},$ it follows that $\lim_{\omega} d_{\mathcal{C}(Y_{i})}(\{\alpha_{i},\beta_{i}\},\sigma^{Y_{i}}_{i})$ is bounded, thus completing this case. 

\textbf{Case 4: Either $\overline{W_{\alpha}} \subset \overline{Y} $ or $\overline{Y} \pitchfork \overline{W_{\alpha}},$ and in both cases, $\lim_{\omega} d_{\mathcal{C}(Y_{i})}(\partial W_{\alpha, i}, \{\alpha_{i},\beta_{i}\})$ is unbounded.}  Since $\lim_{\omega} d_{\mathbb{S}'(S)}(\alpha_{i},\partial W_{\alpha, i})$ is bounded, it follows that there is a bounded path of connected multicurves in the curve complex $\mathcal{C}(S)$ from $\alpha_{i}$ to $\partial W_{\alpha, i}$ such that each multicurve is a separating multicurve.  Call this path $\rho_{i}.$  On the other hand, the assumption of the case is that $\lim_{\omega} d_{\mathcal{C}(Y_{i})}(\partial W_{\alpha, i}, \{\alpha_{i},\beta_{i}\}) \ra \infty.$  Putting things together, by Lemma \ref{lem:curveprojub} it follows \uas $Y_{i}$ is disjoint from some vertex in $\rho_{i}.$  By construction, it follows that $\partial Y_{i} \in \mathbb{S}(S),$ and in fact $\lim_{\omega} d_{\mathbb{S}(S)}(\alpha_{i},\partial Y_{i})$ is bounded.  Since $\overline{Q}$ is not in case one of the equivalence relation $\sim_{{\overline{\alpha}},{\overline{\beta}}},$ it follows that $\lim_{\omega} d_{\mathcal{C}(Y_{i})}(Q_{i},\beta_{i})$ is bounded.  It follows that $\lim_{\omega} d_{\mathcal{C}(Y_{i})}(\{\alpha_{i},\beta_{i}\},\sigma^{Y_{i}}_{i})$ is bounded.  This completes the proof of the final case thereby completing the proof of the lemma.
\end{proof}

\subsubsection{Separation property of structurally integral corners}

As an immediate corollary of Theorem \ref{thm:equivrel} we have the following useful separation property of structurally integral corners in the asymptotic cone.  This separation property should be compared with the separation property of microscopic jets recorded in Theorem \ref{thm:mjetssep}.

\begin{cor} \label{cor:cornersep}
Let $_{\overline{\alpha}}C_{\overline{\beta}}$ be a structurally integral corner, and let $x_{\omega}, x'_{\omega} \in \mathcal{P}_{\omega}(S) \setminus _{\overline{\alpha}}C_{\overline{\beta}}$ be points in the asymptotic cone such that $x_{\omega} \not \sim_{{\overline{\alpha}},{\overline{\beta}}}x'_{\omega}.$  Then $x_{\omega}$ and $x'_{\omega}$ are separated by the corner $ _{\overline{\alpha}}C_{\overline{\beta}}.$  
\end{cor}

%

\subsection{Finest pieces}
\label{sec:finestpieces}
 Behrstock showed that every point in the asymptotic cone of both the mapping class group and Teichm\"uller space is a global cut-point,  \cite{behrstock}.  On the other hand, it is well established that for surfaces $S$ with $\xi(S) \geq 2,$ the mapping class group admits quasi-isometric embeddings of $\Z'$ flats, while for surfaces with $\xi(S) \geq 3$ Teichm\"uller space admits quasi-isometric embeddings of $\Z'$ flats, \cite{behrstockminsky, brockfarb, mm2}.  Hence, for high enough complexity surfaces the mapping class group and Teichm\"uller space are not $\delta$-hyperbolic and in particular, their asymptotic cones are not $\R$-trees.  Putting things together, for high enough complexity surfaces, the asymptotic cones of the mapping class group and Teichm\"uller space are nontrivial \emph{tree-graded} spaces with the property that every point is a cut-point globally, but not locally for some nontrivial local regions.  In such settings, we have canonically defined \emph{finest pieces of the tree-graded structure} which are maximal subsets of the asymptotic cone subject to the condition that no two points in a finest piece can be separated by the removal of a point.  In this subsection, we will characterize of the canonically defined finest pieces in the tree-graded structure of  $\mathcal{T}_{\omega}(S).$  Our theorem is motivated by and should be compared with the following theorem of \cite{bkmm}:
\begin{thm}[\cite{bkmm} Theorem 7.9]\label{thm:mcgpieces} 
Let $S=S_{g,n}$ and let $\MCG_{\omega}(S)$ be any asymptotic cone of $\mathcal{MCG}(S).$  Then for all $a_{\omega},b_{\omega}  \in \mathcal{MCG}_{\omega}(S),$ the following are equivalent:
\begin{enumerate}
\item No point separates $a_{\omega}$ and $b_{\omega},$ and 
\item In any neighborhood of $a_{\omega}, b_{\omega}$ there exists $a'_{\omega},b'_{\omega},$ with representatives $(a'_{i}),(b'_{i})$ respectively, such that: $$\lim_{\omega}d_{\mathcal{C}(S)}(a'_{i},b'_{i}) < \infty.$$  
\end{enumerate}
\end{thm}    

\begin{ex}[$\MCG(S) \mbox { vs } \mathcal{P}(S)$; a partial pseudo-Anosov axis] The following example demonstrates that Theorem \ref{thm:mcgpieces} cannot be applied without modification to $\mathcal{P}(S).$  Consider a representative $(P^{0}_{i})$ of the basepoint of our asymptotic cone $\mathcal{P}_{\omega}(S),$ and let $\gamma_{i} \in  P^{0}_{i} $ be a non-separating curve.  Let $g_{i} \in \MCG(S\setminus \gamma_{i})$ be a pseudo-Anosov map.  Then consider the following two points in the asymptotic cone:
$$a_{\omega}= (P^{0}_{i}), \;\;\;  b_{\omega}=(g_{i}^{s_{i}}P^{0}_{i}).$$  
By construction, $a_{\omega} \ne b_{\omega}$ lie on a partial psuedo-Anosov axis in the asymptotic cone.  Furthermore, by construction, using notation from Subsection \ref{subsec:sublineargrowth} we have: $$a_{\omega}, b_{\omega} \in P_{\overline{S \setminus \gamma},a_{\omega}}= F_{\overline{S \setminus \gamma},a_{\omega}} \times \{ \overline{\gamma} \} = \R\mbox{-tree} \times \{pt\}  \subset \mathcal{P}_{\omega}(S).$$  
Hence, $a_{\omega},$ and $b_{\omega}$  can be separated by a cut-point.  Nonetheless, $a_{\omega}$ and $b_{\omega}$ have representatives $(P^{0}_{i}),$ $(g_{i}^{s_{i}}P^{0}_{i}),$ respectively, each containing $\gamma_{i}.$  In particular, $\forall i\in \N,$ $d_{\mathcal{C}(S)}(P^{0}_{i}, g_{i}^{s_{i}}P^{0}_{i}) = 0.$  Hence in $\mathcal{P}(S),$ statement (1) of  Theorem \ref{thm:mcgpieces} can fail even though statement (2) holds.   
\end{ex}

Despite the fact that Theorem \ref{thm:mcgpieces} does not apply verbatim to $\mathcal{P}(S),$ the following slightly modified theorem with condition (2) strengthened does apply to $\mathcal{P}(S).$ 
\begin{thm}\label{thm:pieces} Let $S=S_{g,n}$ and let $\mathcal{P}_{\omega}(S)$ be any asymptotic cone of $\mathcal{P}(S).$  Then for all $a_{\omega},b_{\omega}  \in \mathcal{P}_{\omega}(S),$ the following are equivalent:
\begin{enumerate}
\item No point separates $a_{\omega}$ and $b_{\omega},$ or equivalently $a_{\omega}$ and $b_{\omega}$ are in the same canonical finest piece, and
\item In any neighborhood of $a_{\omega}, b_{\omega},$ respectively, there exists $a'_{\omega},b'_{\omega},$ with representative sequences $(a'_{i})$,$(b'_{i})$, such that $\lim_{\omega}d_{\mathbb{S}(S)}(a'_{i},b'_{i}) < \infty.$  
\end{enumerate}
\end{thm}

\begin{rem} \label{rem:alternativecondition2}
Note that condition (2) of Theorem \ref{thm:pieces} implies condition (2) of Theorem \ref{thm:mcgpieces} as distance in $\mathbb{C}(S)$ is coarsely bounded above by distance in $\mathbb{S}(S),$ see Equation \ref{eq:coarsebelow}.  Moreover, note that by Theorem \ref{thm:sepcomplexrelasymptotic}, condition (2) of Theorem \ref{thm:pieces} can be replaced by the following statement:  In any neighborhood of $a_{\omega}, b_{\omega},$ respectively, there exist points $a'_{\omega},b'_{\omega},$ such that $d_{\mathbb{S}_{\omega}(S)}(a'_{\omega},b'_{\omega}) < \infty.$
\end{rem}
\begin{proof} [Proof of Theorem \ref{thm:pieces}]

\textbf{$(2) \implies (1)$}:  As noted in Remark \ref{rem:alternativecondition2}, Property (2) implies that $a_{\omega}, b_{\omega}$ are limit points of sequences in the asymptotic cone which have finite $\mathbb{S}_{\omega}(S)$ distance.  Since the canonically defined finest pieces are closed sets \cite{drutusapir}, it suffices to show that points in the asymptotic cone with finite $\mathbb{S}_{\omega}(S)$ distance cannot be separated by a point.  Specifically, assume we have a chain of natural convex nontrivial product regions $\qbar{\gamma_{0}}, ...,\qbar{\gamma_{K}}$ in the asymptotic cone $\mathcal{P}_{\omega}(S)$ such that $a'_{\omega} \in  \qbar{\gamma_{0}},$ $b'_{\omega} \in \qbar{\gamma_{K}},$ and for all $j\in \{0,...,K-1 \}$ $\qbar{\gamma_{j}} \bigcap \qbar{\gamma_{j+1}}$ has infinite diameter intersection.  Clearly, each product region cannot be separated by a point.  Furthermore, by assumption each product region cannot be separated from its neighbor by a point.  It follows that $a'_{\omega}$ and $b'_{\omega}$ cannot be separated by a point, thus completing the proof of $(2) \implies (1).$

\textbf{$(1) \implies (2)$}:  We will prove the contrapositive, namely $\sim (2) \implies \sim (1).$  The negation of property (2) implies that there exists an $r_{1}>0$ such that all points in $r_{1}$ open neighborhoods of $a_{\omega}$ and $b_{\omega}$ respectively have infinite or undefined $\mathbb{S}_{\omega}(S)$ distance.  By Theorem \ref{thm:conebasics}, $\mathcal{P}_{\omega}(S)$ is locally path connected.  Let $r_{2}>0$ be a constant such that the $r_{2}$ open neighborhoods of $a_{\omega}$ and $b_{\omega}$ are path connected.  Set $3r=\min(r_{1},r_{2}).$  By choosing $r_{1}$ to be sufficiently small, we can assume that $d_{\mathcal{P}_{\omega}(S)}(a_{\omega},b_{\omega})>6r.$  

Let the sequences $(a'_{i}),(b'_{i})$ represent any points $a'_{\omega},b'_{\omega}$ in $r$ neighborhoods of $a_{\omega},b_{\omega}$ respectively, let $\gamma_{i}$ be a hierarchy path between $a'_{i}$ and $b'_{i},$ and let $\gamma_{\omega}$ represent its ultralimit.  By construction $\gamma_{\omega}$ is a $(K,0)$-quasi-geodesic.  Let $a''_{\omega}$ denote a point on $\gamma_{\omega}$ of distance $r$ from $a'_{\omega},$ and let $a'''_{\omega}$ denote a point on $\gamma_{\omega}$ of distance $2r$ from $a'_{\omega}.$  Similarly, let $b''_{\omega}$ denote a point on $\gamma_{\omega}$ of distance $r$ from $b'_{\omega},$ and let $b'''_{\omega}$ denote a point on $\gamma_{\omega}$ of distance $2r$ from $b'_{\omega}.$  See Figure \ref{fig:neighborhoods}.  We will show that the quasi-geodesic $\gamma_{\omega}$ contains a cut-point between the points $a''_{\omega}$ and  $b''_{\omega}.$  Then, local path connectedness implies that the cut-point also separates $a_{\omega}$ and $b_{\omega},$ thus completing the proof of the negation of (1) and hence the proof of the Theorem.  Specifically, since by assumption $a_{\omega}$ and $a'_{\omega}$ (and similarly $b_{\omega}$ and $b'_{\omega}$) are within distance $r$ of each other, and because the cut-point between $a''_{\omega}$ and $b''_{\omega}$ is at least distance $r$ from $a'_{\omega},$  (and similarly from $b'_{\omega}$) it follows that a geodesic path between $a_{\omega}$ and $a'_{\omega}$ (and similarly between $b_{\omega}$ and $b'_{\omega}$) does not contain the cut-point.   

We will proceed by considering two cases.  In the first case we will obtain a cut-point using the machinery of microscopic jets and in the second case we will obtain a cut-point using the machinery of structurally integral corners.

\begin{figure}[htpb]
\centering
\includegraphics[height=5 cm]{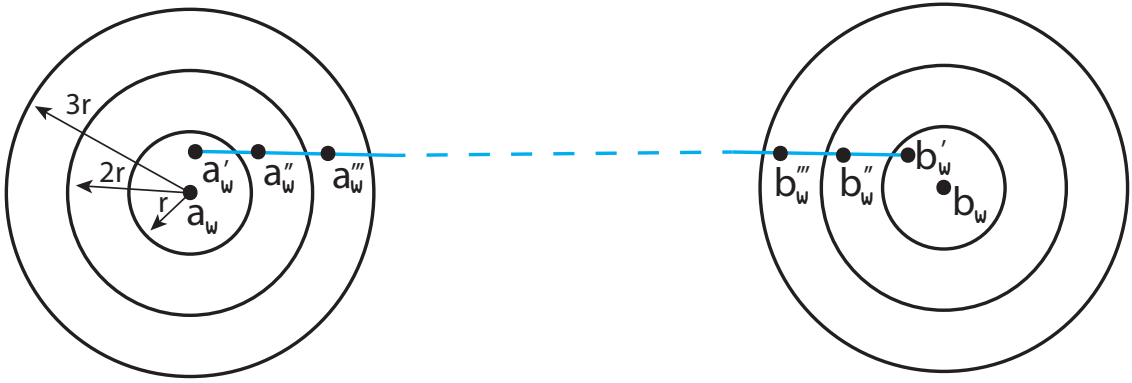}
\caption[Local neighborhoods of points in $\mathcal{P}_{\omega}(S)$]{The dotted line is a quasi-geodesic $\gamma_{\omega}$ from $a'_{\omega}$ to $b'_{\omega}.$}\label{fig:neighborhoods}
\end{figure}

\noindent \textbf{Case One: $\exists r'$ such that for all $a^{0}_{\omega},$ $b^{0}_{\omega}$ in $3r'$ neighborhoods of $a_{\omega},b_{\omega},$ with $(a^{0}_{i}),$$(b^{0}_{i})$ any representatives thereof, respectively, $\exists \overline{Y} \in \nseultra{S}$ with $\lim_{\omega}d_{\mathcal{C}(Y_{i})}(a^{0}_{i},b^{0}_{i}) \ra \infty.$}  

By abuse of notation assume that we have replaced $r$ described above by $r=\min\{r,r'\}.$  In particular, since $a'''_{\omega}, b'''_{\omega}$ are contained in $3r'$ neighborhoods of $a_{\omega},b_{\omega},$ respectively, the assumption of the case ensures that for some $\overline{Y} \in \nseultra{S},$ we have $\lim_{\omega}d_{\mathcal{C}(Y_{i})}(a'''_{i},b'''_{i}) \ra \infty.$  Then, by Theorem \ref{thm:mjetsexist} there exists a microscopic jet $J=(\overline{g},\overline{Y},\overline{a'''},\overline{b'''})$ with $\overline{g} \subset \gamma_{\omega}|_{[a'''_{\omega},b'''_{\omega}]}$ and such that  $a'''_{\omega} \not \sim_{J} b'''_{\omega}.$  By definition, $\lim_{\omega}d_{\mathcal{C}(Y_{i})}(\pi_{g_{i}}(a'''_{i}),\pi_{g_{i}}(b'''_{i})) \ra \infty.$  By the properties of hierarchies in Theorem \ref{thm:hierarchy} it follows that $\lim_{\omega}d_{\mathcal{C}(Y_{i})}(\pi_{g_{i}}(a''_{i}),\pi_{g_{i}}(b''_{i})) \ra \infty,$ and hence $a''_{\omega} \not \sim_{J} b''_{\omega}.$ 

Since the complement $\overline{Y^{c}}$ is the emptyset,  $\iota(J) \times \mathcal{P}_{\omega}(\overline{Y^{c}})$ is a single point in the asymptotic cone.  Moreover, by construction it is not equal to either $a''_{\omega}$ or $b''_{\omega}.$  Theorem \ref{thm:mjetssep} implies that the initial point of the jet is a cut-point between $a''_{\omega}$ and  $b''_{\omega}.$  This completes the proof of case one.  It should be noted that the proof of case one follows closely the proof of Theorem \ref{thm:mcgpieces} in \cite{bkmm}.  In fact, for the special case of $\overline{Y}=S$ the proofs are identical.

\noindent \textbf{Case Two: The negation of case one.  Namely, in any neighborhoods of $a_{\omega}, b_{\omega}$ there exists $a^{0}_{\omega},b^{0}_{\omega}$ with representatives $(a^{0}_{i}),$$(b^{0}_{i}),$ such that $\forall \overline{Y} \in \nseultra{S}, \lim_{\omega}d_{\mathcal{C}(Y_{i})}(a^{0}_{i},b^{0}_{i}) < \infty.$}

For $r$ neighborhoods of $a_{\omega}, b_{\omega}$ set the points $a^{0}_{\omega},b^{0}_{\omega}$ with representatives $(a^{0}_{i}),$$(b^{0}_{i}),$ guaranteed to exist by the hypothesis of the case to be equal to $a'_{\omega},b'_{\omega},$ with representatives $(a'_{i}),$$(b'_{i}),$ respectively.  Then as above, let $\gamma_{i}$ be a hierarchy path between $a'_{i}$ and $b'_{i},$ and similarly define the points $a''_{i},a'''_{i},b''_{i},b'''_{i}.$  By the assumptions of the case the hierarchies $\gamma_{i}$ have the property that for all $Y \in \nse{S},$ the projection of $\gamma_{i}$ to $\mathcal{C}(Y)$ is uniformly bounded.  In particular, the hierarchies $\gamma_{i}$ have uniformly bounded main geodesic length and travels for uniformly bounded distances in all connected nonseparating essential subsurfaces $Y.$  By Lemma \ref{lem:finite} there is a $k$ such that \uas the main geodesic in $\gamma_{i}$ has length exactly $k.$  Specifically, \uas there is a tight main geodesic in $\mathcal{C}(S),$ with simplices $g_{0i},..., g_{ki}$ such that $g_{0i} \subset  a'_{i} ,$ $g_{ki} \subset  b'_{i} .$  By construction, the hierarchy $\gamma_{i}$ travels through the finite set of quasi-convex regions, $\q{g_{0i}},$ ..., $\q{g_{ki}}.$  See Figure \ref{fig:path}.

\begin{figure}[htpb]
\centering
\includegraphics[height=4 cm]{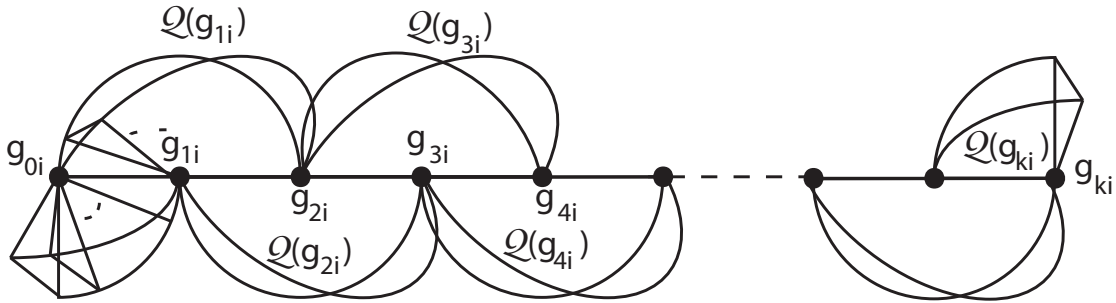}
\caption[The ultralimit of hierarchy paths with uniformly bounded main geodesics.]{The ultralimit of hierarchy paths with a uniformly bounded main geodesics.  Notice that each of the vertices along the finite length main geodesic are separating multicurves.}\label{fig:path}
\end{figure}

Without loss of generality we can assume that for all $j,$ either $\gamma_{ji} \in \mathcal{P}(S),$ i.e $\gamma_{ji}$ is an entire pants decomposition of a surface, or for any $(W_{i})$ a sequence of connected essential subsurfaces in the complement $S \setminus g_{ji},$ we have $\lim_{\omega}d_{\mathcal{C}(W_{i})}(a'_{i}, b'_{i} ) \ra \infty.$  If not, by iterating the argument we used above for a finite length $\mathcal{C}(S)$ main geodesic we can \uas replace the multicurve $g_{ji}$ by a finite list of connected simplices in $\mathcal{C}(S)$ each containing $g_{ji}$ as a proper multicurve.  This iteration process of replacing a multicurve $g_{ji}$ from our our finite list $\{g_{0i},...,g_{ki}\}$ with finite sequences of multicurves each containing the original multicurve as a proper multicurve must terminate due to the finite complexity of the surface $S.$  Accordingly, we have a finite list of nontrivial quasi-convex regions and singletons through which our hierarchy path $\gamma_{i}$ from $a'_{i}$ to $b'_{i}$ \uas travels.  Since the list of nontrivial quasi-convex regions and singletons is bounded \uas, coarsely we can ignore the singletons.  That is, coarsely our hierarchy path $\gamma_{i}$ from $a'_{i}$ to $b'_{i}$ \uas travels through only a finite list of nontrivial quasi-convex regions, $\q{g_{0i}},$ ..., $\q{g_{k'i}}$ such that for any $(W_{i})$ a sequence of connected component of $S \setminus g_{ji},$ we have $\lim_{\omega}d_{\mathcal{C}(W_{i})}(a'_{i}, b'_{i} ) \ra \infty.$  By the assumptions of our case, for each $j,$ $\omega$-a.s $g_{i,j}$ is a separating multicurve, or equivalently for each $j$ the region $\q{g_{ji}}$ is a nontrivial quasi-convex product region.  Moreover, by construction for all $j,$ $\lim_{\omega}d_{\mathcal{P}(S)}(g_{ij} \lrcorner g_{(i+1)j}, g_{(+1)j} \lrcorner g_{ij}, )$ is bounded.  Notice that all of the above analysis holds after restricting to the subquasi-geodesic $\gamma_{i}|_{a'''_{i},b'''_{i}}.$  Assume we have done so.   

However, by the negation of condition (2) of the theorem, it follows that there exist consecutive separating multicurves, $g_{ji}, g_{(j+1)i}$ in our list such that: $$\lim_{\omega}d_{\mathbb{S}(S)}(g_{ji},g_{(j+1),i}) \ra \infty.$$  In particular, in conjunction with the analysis of the previous paragraph, we have a structurally integral corner $_{\overline{g_{j}}} C_{\overline{g_{j+1}}}.$  Moreover, by construction $a''_{\omega},b''_{\omega} \ne  \; _{\overline{g'_{j}}}C_{\overline{g'_{j+1}}}$ as the corner is on the quasi-geodesic $\gamma_{\omega}|_{[a'''_{\omega},b'''_{\omega}]}.$  Furthermore, $a''_{\omega} \not \sim_{\overline{g'_{j}},\overline{g'_{j+1}}} b''_{\omega},$ as by our assumptions $a''_{\omega}$ is in case one of the equivalence relation $\sim_{\overline{g'_{j}},\overline{g'_{j+1}}}$ while $b''_{\omega}$ is in case two of the equivalence relation $\sim_{\overline{g'_{j}},\overline{g'_{j+1}}}.$  Corollary \ref{cor:cornersep} implies that the structurally integral corner $_{\overline{g'_{j}}}C_{\overline{g'_{j+1}}}$ is a cut-point between the points $a''_{\omega},b''_{\omega}.$  This completes the proof of the theorem.  
\end{proof}

\begin{rem} \label{rem:nooccur2}
As in Remark \ref{rem:nooccur1}, Case Two in the proof of Theorem \ref{thm:pieces} occur only for surfaces with $|\chi(S)| \leq 4.$
\end{rem}

\subsubsection{Applications of the classification of finest pieces}

Special cases of Theorem \ref{thm:pieces} include the following celebrated theorems of others.

\begin{cor}[\cite{behrstock,brockfarb} Theorem 5.1, Theorem 1.1] \label{cor:hyperbolic}
Let $S=S_{1,2}$ or $S_{0,5}.$ Then $\mathcal{P}(S)$ is $\delta$-hyperbolic. 
\end{cor}

\begin{proof}
It suffices to show that for all choices of asymptotic cones, $\mathcal{P}_{\omega}(S)$ is an $\R$-tree, see \cite{drutu, gromov}.  Equivalently, it suffices to show that the finest pieces in any asymptotic cone are trivial, or equivalently, any two points $a_{\omega} \ne b_{\omega} \in \mathcal{P}_{\omega}(S)$ can be separated by a point.  However, by Theorem \ref{thm:pieces}  this is immediate as $\mathbb{S}(S)=\emptyset.$ 
\end{proof}

\begin{cor}[\cite{brockmasur} Theorem 1]\label{cor:relhyperbolic} Let $\xi(S)=3,$ then $\mathcal{P}(S)$ is relatively hyperbolic with respect to natural quasi-convex product regions consisting of all pairs of pants with a fixed separating curve.
\end{cor}     
\begin{proof}
It suffices to show that $\mathcal{P}(S)$ is asymptotically tree-graded with respect to peripheral subsets consisting of all natural quasi-convex product regions $\q{\gamma}$ for any $\gamma \in \mathcal{C}_{sep}(S).$    

By topological considerations any two separating curves $\gamma \ne \delta \in \mathcal{C}_{sep}(S),$ $S \setminus (\gamma \cup \delta)$ does not contain an essential subsurface.  Consequently, $d_{\mathbb{S}(S)} \in \{0,\infty\},$ and similarly for all $\overline{C},\overline{D} \in \mathbb{S}^{\omega}(S),$ the expression $\lim_{\omega}d_{\mathbb{S}(S)}(C_{i},D_{i})$ takes values in $\{0,\infty\}.$ Accordingly, Theorem \ref{thm:pieces} implies any two points $a_{\omega}, b_{\omega}$ are either in a common natural convex product region (such regions are closed) or are separated by a cut-point.  In particular, any simple nontrivial geodesic triangle in $\mathcal{P}_{\omega}(S)$ must be contained entirely inside a single piece $\qbar{\gamma}.$ 
\end{proof}

While stated for $\mathcal{P}(S),$ Corollaries \ref{cor:hyperbolic} and \ref{cor:relhyperbolic} immediately apply to $\mathcal{T}(S)$ as hyperbolicity and strong relative hyperbolicity are quasi-isometry invariant properties.

\subsection{Hyperbolic type quasi-geodesics}
\label{sec:contracting}
%
In this section, after some definitions of the various types of hyperbolic type geodesics, we will characterize hyperbolic type quasi-geodesics in Teichm\"uller space.  See \cite{lenzhenrafitao} for a similar analysis of strongly contracting quasi-geodesics in Teichm\"uller space equipped with the Lipschitz metric.  
 
 \begin{defn}[Morse] A (quasi-)geodesic $\gamma$ is called a \emph{Morse (quasi-)geodesic} if every $(K,L)$-quasi-geodesic with endpoints on $\gamma$ is within a bounded distance from $\gamma,$ with the bound depending only on the constants $K,L.$  Similarly, the definition of Morse can be associated to a sequence of (quasi-)geodesic segments with uniform quasi-isometry constants.
\end{defn} 

\begin{defn}[contracting quasi-geodesic]  \label{def:bccont}A quasi-geodesic $\gamma$ is said to be \emph{(b,c)--contracting} if $\exists$ constants $0<b\leq 1$ and $0<c$ such that $\forall x,y \in X:$ $$d_{X}(x,y)<bd_{X}(x,\pi_{\gamma}(x))  \implies d_{X}(\pi_{\gamma}(x),\pi_{\gamma}(y))<c.$$
For the special case of a (b,c)--contracting quasi-geodesic where $b$ can be chosen to be $1,$ the quasi-geodesic $\gamma$ is called \emph{strongly contracting.}
\end{defn}    

In \cite{sultanmorse}, hyperbolic type quasi-geodesics in CAT(0) spaces are analyzed.  In particular, the following result is proven:
\begin{thm}[\cite{sultanmorse} Theorem 3.4] \label{thm:hypquasi} Let $X$ be a CAT(0) space and $\gamma \subset X$ a quasi-geodesic.  Then, the following are equivalent: (1) $\gamma$ is (b,c)--contracting, (2) $\gamma$ is strongly contracting, (iii) $\gamma$ is Morse, and (iv) In every asymptotic cone $X_{\omega},$ any two distinct points in the ultralimit $\gamma_{\omega}$ are separated by a cut-point.
\end{thm}

Recall that $\overline{\mathcal{T}}(S)$ is CAT(0).  Combining Theorems \ref{thm:pieces} and \ref{thm:hypquasi}, the following corollary characterizes all strongly contracting quasi-geodesics in $\mathcal{\overline{T}}(S).$  Equivalently, in light of Theorem \ref{thm:hypquasi} the theorem also characterizes Morse quasi-geodesics in $\mathcal{\overline{T}}(S).$  The characterization represents a generalization of quasi-geodesics with \emph{bounded combinatorics} studied in \cite{behrstock, bmm2}.  Specifically, in  \cite{behrstock, bmm2} it is shown that quasi-geodesics in $\mathcal{P}(S)$ which have uniformly bounded subsurface projections to all connected proper essential subsurfaces.  More generally, we show:


\begin{thm} \label{thm:contracting} Let $\gamma$ be a quasi-geodesic in $\mathcal{\overline{T}}(S),$ and using Theorem \ref{thm:brock} let $\gamma'$ be a corresponding quasi-geodesic in $\mathcal{P}(S).$  Then $\gamma$ is strongly contracting if and only if there exists a constant $C$ such that for all $Y \in \se{S}$ the subsurface projection $\pi_{Y}(\gamma')$ has diameter bounded above by $C.$ 
\end{thm}

\begin{proof}
Assume there is no uniform bound $C$ on the subsurface projection $\pi_{Y}(\gamma'),$ where $Y$ ranges over $\se{S}.$  Then we can construct $\overline{Y} \in \seultra{S}$ such that $\lim_{i} diam(\pi_{Y_{i}}(\gamma')) \ra \infty.$  By the properties of hierarchies in Theorem \ref{thm:hierarchy}, it follows that there is a sequence of hierarchy quasi-geodesic segments $\{\gamma'_{r}\}_{r}$ with endpoints on $\gamma'$  traveling through product regions $\q{\partial Y_{r}}$ for unbounded connected subsegments.  In particular, the sequence of quasi-geodesics $\{\gamma'_{r}\}_{r}$ are not Morse, and furthermore since the hierarchy segments $\gamma'_{r}$ are all quasi-geodesics with uniform constants which have endpoints on $\gamma'$, the quasi-geodesic $\gamma'$ is also not Morse.  Moreover, considering the quasi-isometry taking $\gamma'$ to $\gamma,$ it similarly follows that $\gamma$ is not Morse.  By Theorem \ref{thm:hypquasi}, $\gamma$ is not strongly contracting.

On the other hand, assume $\forall Y \in \se{S}$ that the subsurface projection $\pi_{Y}(\gamma')$ is uniformly bounded.  Let $\mathcal{P}_{\omega}(S)$ be any asymptotic cone with $a_{\omega},b_{\omega}$ any two distinct points on $\gamma'_{\omega}$ with representatives sequences $(a_{i}),(b_{i}) \in \gamma',$ respectively.  Proceeding as in Case One of the proof of Theorem \ref{thm:pieces}, consider a sequence of hierarchy quasi-geodesic segments $\rho(a_{i},b_{i}),$ between the points $a_{i}$ and $b_{i}$ on $\gamma',$ and define distinct points $a''_{\omega},a'''_{\omega},b''_{\omega},b'''_{\omega}$  with representatives $(a''_{i}),(a'''_{i}),(b''_{i}),(b'''_{i})$ along the sequence of hierarchy quasi-geodesic segments $\rho(a_{i},b_{i}).$  By assumption, $\forall \overline{Y} \in \seultra{S}, \lim_{\omega}d_{\mathcal{C}(Y_{i})}(a'''_{i},b'''_{i})$ is bounded.  On the other hand, since $a'''_{\omega} \ne b'''_{\omega}$ by Corollary \ref{cor:sepbounded} there is some $\overline{W} \in \eultra{S}$ such that  $\lim_{\omega}d_{\mathcal{C}(W_{i})}(a_{i},b_{i})$ is unbounded.  Perforce, $\overline{W} \in \nseultra{S}.$  Then, as in Case One of the proof of Theorem \ref{thm:pieces}, there exists a microscopic jet which gives rise to a cut-point between $a_{\omega}$ and $b_{\omega}.$  Since $a_{\omega}$ and $b_{\omega}$ are arbitrary and because cut-points in asymptotic cones are preserved by quasi-isometries, by Theorem \ref{thm:hypquasi} $\gamma$ is strongly contracting.
\end{proof}

\section{Thickness and Divergence of Teichm\"uller Spaces}
\label{chap:thick}
In this section we focus our analysis on the surface $S_{2,1}$ which has previously proven to be difficult to understand, as is apparent from the surrounding literature.  In particular, we complete the thickness classification of Teichm\"uller space of all surfaces of finite type presented in Table \ref{table:thick}.  Specifically, we prove that the Teichm\"uller space of the surface $S_{2,1}$ is thick of order two and has superquadratic divergence, thereby answering questions of \cite{behrstockdrutu, bdm,brockmasur}.  The proof in this section is broken up into three subsections.  In Subsection \ref{sec:atmosttwo} we recall the construction in \cite{brockmasur} where it is shown that $\mathcal{T}(S_{2,1})$ is thick of order at least one and at most two.  Then, in Subsection \ref{sec:thickorder2} we prove that $\mathcal{T}(S_{2,1})$ cannot be thick of order one.  In Subsection \ref{sec:divergence} using our understanding from the previous sections we prove that $\mathcal{T}(S_{2,1})$ can be uniquely characterized among all Teichm\"uller spaces as it has a divergence function which is superquadratic yet subexponential.  Throughout this section we will use the pants complex as a quasi-isometric model for Teichm\"uller space, often making statements and theorems about Teichm\"uller space with proofs obtained from considering the pants complex.

\subsection{$\mathcal{T}(S_{2,1})$ is thick of order one or two}
\label{sec:atmosttwo} In this section we recall results of Behrstock in \cite{behrstock} and Brock-Masur in \cite{brockmasur}.  Specifically, we first recall a result of Behrstock that shows that for all surfaces $\mathcal{T}(S)$ is never wide.  By definition, it follows that $\mathcal{T}(S)$ is never thick of order zero.  Then, we record a slightly adapted version of a proof in \cite{brockmasur} that $\mathcal{T}(S_{2,1})$ is thick of order at most two.  Putting things together, this section implies that $\mathcal{T}(S_{2,1})$ is thick of order one or two.  The reason for the necessary slight adaptation in this section of the proof in \cite{brockmasur} is due to the various versions of thickness in the literature.  See Remark \ref{rem:wide}.

We begin by recalling the following theorem of Behrstock:
\begin{thm}[\cite{behrstock} Theorem 7.1] \label{thm:notwide} Let $\gamma$ be any pseudo-Anosov axis in $\mathcal{P}(S),$ and let $\gamma_{\omega}$ be its ultralimit in any asymptotic cone $\mathcal{P}_{\omega}(S).$  Then any distinct points on $\gamma_{\omega}$ are separated by a cut-point.  
\end{thm}
Since all mapping class groups of surfaces with positive complexity contain pseudo-Anosov elements, and given any pseudo-Anosov axis, one can choose an asymptotic cone in which its ultralimit is nontrivial, by Theorem \ref{thm:notwide} it follows that $\mathcal{T}(S)$ is never wide, and hence never thick of order zero.

Next, we consider the proof in \cite{brockmasur} proving that $\mathcal{T}(S_{2,1})$ is thick of order at most two.  Given $\alpha \in \mathcal{C}_{sep}(S_{2,0}),$ let $\tilde{\alpha} \in \mathcal{C}_{sep}(S_{2,1})$ denote any lift of $\alpha$ with respect to the projection $\pi=\pi_{\mathcal{C}(S_{2,0})}\co \mathcal{C}_{sep}(S_{2,1}) \ra \mathcal{C}_{sep}(S_{2,0})$ which forgets about the boundary component.  By topological considerations $S\setminus \tilde{\alpha} = Y_{1} \sqcup Y_{2} = S_{1,1} \sqcup S_{1,2}.$  Since $diam(\mathcal{P}(Y_{i})) = \infty,$ we can choose bi-infinite geodesics $\rho_{i} \in \mathcal{P}(Y_{i}),$ and in fact, by Theorem \ref{thm:quasidistance}, the span of any two such bi-infinite geodesics in the different connected components $Y_{1},Y_{2}$ comprise a quasi-flat.  In particular, it follows that the sets $\q{\tilde{\alpha}}$ are nontrivial product regions, and in particular are wide.  Again, using Theorem \ref{thm:quasidistance}, it is also immediate that subsets  $\q{\tilde{\alpha}}$ are quasi-convex.  Moreover, using the property of hierarchies in Theorem \ref{thm:hierarchy}, it follows that these subsets $\q{\tilde{\alpha}}$ satisfy the non triviality property of every point having a bi-infinite quasi-geodesic through it.  Hence, the subsets $\q{\tilde{\alpha}}$ are thick of order zero.

%
%
%

With the notation as above, set 
\begin{equation} \label{eq:subpieces} \x{\alpha} =\{Q \in \mathcal{P}(S_{2,1}) \; | \; \alpha \in \pi(Q) \} = \bigcup_{\tilde{\alpha}}\q{\tilde{\alpha}} 
\end{equation}

Presently we will prove the following theorem:
\begin{thm} [\cite{brockmasur} Theorem 18] \label{thm:brockmasur} $\mathcal{T}(S_{2,1})$ is thick of order at most two.
\end{thm}

To prove Theorem \ref{thm:brockmasur}, Brock-Masur show that the subsets $\x{\alpha}$ are thick of order at most one, any two subsets $\x{\alpha}, \x{\alpha '}$ can be thickly chained together, and the union of uniform neighborhoods of all subsets $\x{\alpha}$ is all of $\mathcal{P}(S_{2,1}).$  Each of these steps will be worked out.

\textbf{1. $\x{\alpha}$ is thick of order one:}  For a given separating curve $\alpha \in \mathcal{C}_{sep}(S_{2,0}),$ consider the set of all thick of order zero subsets $\q{\tilde{\alpha}},$ with $\pi(\tilde{\alpha})=\alpha.$  By definition, the union of all the thick of order zero subsets $\q{\tilde{\alpha}}$ is precisely all of $\x{\alpha}.$  Furthermore, since by Lemma \ref{lem:connected} the fiber of $\alpha$ under the projection map $\pi$ is connected in $\mathcal{C}_{sep}(S_{2,1}),$ in order to prove thick connectivity of elements in the set of all thick of order zero subsets $\q{\tilde{\alpha}},$ it suffices to notice that for $\tilde{\alpha}$ and $\tilde{\alpha'}$ disjoint separating curves, the quasi-convex product regions $\q{\tilde{\alpha}}$ and $\q{\tilde{\alpha'}}$ thickly intersect.  However, this is immediate as $\q{\tilde{\alpha}} \cap \q{\tilde{\alpha'}} = \q{\tilde{\alpha} \cup \tilde{\alpha'}}$ is itself a natural quasi-convex nontrivial product regions and in particular has infinite diameter.

\textbf{2.  Subsets $\x{\alpha}$ and $\x{\alpha'}$ can be thickly chained together:}  Given any separating curves $\alpha,\alpha' \in \mathcal{C}_{sep}(S_{2,0})$ there is a sequence of separating curves between them such that each separating curve intersects its neighboring curves in the sequence minimally.  Specifically, there is a sequence of separating curves $$\alpha = a_{0},a_{1},...,a_{n}=\alpha'$$ with $|a_{i}\cap a_{i+1}| =4,$ see for instance \cite{schleimer}.  Hence, we can assume that $\alpha, \alpha'$ intersect four times.  Up to homeomorphism there are only a finite number of such similar situations, one of which is presented in Figure \ref{fig:pants5}. 

\begin{figure}[htpb]
\centering
\includegraphics[height=3.3 cm]{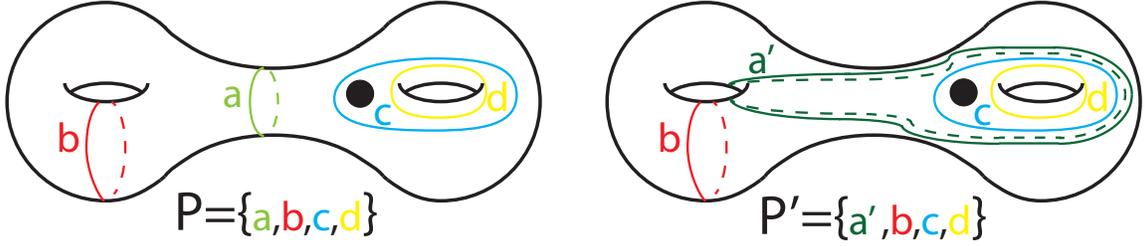}
\caption[Pants decompositions with minimally intersecting separating curves.]{Pants decompositions with minimally intersecting separating curves that are distance two in $\mathcal{P}(S_{2,0}).$}\label{fig:pants5}
\end{figure}
As in Figure \ref{fig:pants5}, we then have pants decompositions, $P_{1} \in \x{\alpha}, P'_{1} \in \x{\alpha'}$ such that  $d_{\mathcal{P}(S_{2,1})}(P_{1},P'_{1})=D,$ for some uniform constant $D.$  Then for any (partial) pseudo-Anosov element $ g \in Push \subset \MCG(S_{2,1}),$ set $P_{n}= g^{n}P_{1}, \;\; P'_{n}=g^{n}P'_{1}.$  By Lemma \ref{lem:preserve}, $\forall \; n \in \Z$ $P_{n} \in \x{\alpha},   P'_{n} \in \x{\alpha'},$ and moreover, $$d_{\mathcal{P}(S_{2,1})}(P_{n},P'_{n})=d_{\mathcal{P}(S_{2,1})}(g^{n}P_{1}, g^{n}P'_{1}) = d_{\mathcal{P}(S_{2,1})}(P_{1},P'_{1})=D.$$  It follows that $diam(N_{D}(\x{\alpha}) \cap N_{D}(\x{\alpha'})) = \infty$ as it contains the axes of (partial) pseudo-Anosov elements. 

\textbf{3.  $N_{1}(\bigcup_{\alpha} \x{\alpha})= \mathcal{P}(S_{2,1}):$}  This follows immediately from the observation that any pair of pants in  $\mathcal{P}(S_{2,1})$ is distance at most one from a pair of pants decomposition containing a separating curve.  

Unfortunately, the above argument for proving that $\mathcal{P}(S_{2,1})$ is thick of order at most two is using a version of thickness which is weaker than the version of thickness in Definition \ref{defn:thick}, and hence we must adapt their proof slightly.  Specifically, recall that in our definition of thickness to show that a space is thick of order at most two it is required that the space have a collection of subsets that are quasi-convex, thick of order one, coarsely make up the entire space, and thickly intersect.  In the argument above from \cite{brockmasur}  we satisfied all the requirements with the exception of quasi-convexity which appears unlikely for to hold for subsets $\x{\alpha},$ see \cite{sultanthesis} Example 5.1.3.  Nonetheless, we will see that we can modify the above argument such that the conclusion that $\mathcal{P}(S_{2,1})$ is thick of order at most two remains true even with the stronger definition of thickness as in Definition \ref{defn:thick}. The idea will be to consider particular quasi-convex subsets of the sets $\x{\alpha}.$

Let $\tilde{\alpha} \in \mathcal{C}_{sep}(S_{2,1})$ with $\pi(\tilde{\alpha})=\alpha \in \mathcal{C}_{sep}(S_{2,0}),$ and let $f$ be any point pushing pseudo-Anosov mapping class of $S_{2,1},$ such that $d_{\mathcal{C}_{sep}(S_{2,1})}(\tilde{\alpha}, f(\tilde{\alpha}))$ is less than some uniform bound.  Let $\rho = \rho(f,\tilde{\alpha},Q)$ be any quasi-geodesic axis of $f$ in the pants complex which goes through some point $Q$ in $\q{\tilde{\alpha}}.$  Then consider the set $$\mathcal{X}(f,\tilde{\alpha},Q )=: \rho \bigcup_{n} \q{f^{n}(\tilde{\alpha}}).$$  Intuitively, this set $\mathcal{X}(f,\tilde{\alpha},Q )$ should be thought of as a point pushing pseudo-Anosov axis thickened up by product regions which it crosses through.  Note that by construction the sets $\mathcal{X}(f,\tilde{\alpha},Q )$ are coarsely contained in $\x{\alpha}$ and moreover, the earlier proof from \cite{brockmasur} that $\x{\alpha}$ is thick of order one, carries through to show that the subsets $\mathcal{X}(f,\tilde{\alpha},Q )$ are similarly thick of order one in the induced metric from the pants complex.   Moreover, the following lemma shows that the subsets $\mathcal{X}(f,\tilde{\alpha},Q )$ are quasi-convex.  


 \begin{lem} 
\label{lem:quasiconvex}
The sets $\mathcal{X}(f,\tilde{\alpha},Q )$ are quasi-convex.
\end{lem}

\begin{proof}
Pick any elements $A,B \in \mathcal{X}(f,\tilde{\alpha},Q ).$  We will see that they can be connected by a hierarchy quasi-geodesic $\sigma(A,B)$ that remains in a uniform neighborhood of $\mathcal{X}(f,\tilde{\alpha},Q ),$ thus completing the proof.  Without loss of generality we can assume that $A$ and $B$ are contained in natural product regions $\q{f^{j}(\tilde{\alpha})}, \q{f^{k}(\tilde{\alpha})},$ respectively.  Note that remaining in the natural product regions $\q{f^{j}(\tilde{\alpha})}, \q{f^{k}(\tilde{\alpha})},$ the points $A,B$ can be connected to points $f^{j}(Q),f^{k}(Q),$ respectively, both of which lie on the pseudo-Anosov axis $\rho.$  

Since pseudo-Anosov axes have uniformly bounded subsurface projections to all connected proper essential subsurfaces \cite{behrstock, mm2}, it follows that there is a hierarchy quasi-geodesic path connecting $f^{j}(Q)$ and $f^{k}(Q)$ in which the only component domain, for some sufficiently large threshold, is the entire surface $S.$  Accordingly, in the hierarchy quasi-geodesic $\sigma(A,B)$ the only component domains, for some sufficiently large threshold, are the entire surface $S$ and possibly connected essential subsurfaces $Y$ with $Y \subset S \setminus f^{j}(\tilde{\alpha})$ or with $Y \subset S \setminus f^{k}(\tilde{\alpha}).$  By definition, the portion of the $\sigma$ traveling through the component domains of connected essential subsurfaces $Y$ with $Y \subset S \setminus f^{j}(\tilde{\alpha})$ or with $Y \subset S \setminus f^{k}(\tilde{\alpha})$ is coarsely contained in the set $\mathcal{X}(f,\tilde{\alpha},Q).$  

\end{proof}

\begin{proof}[Proof of Theorem \ref{thm:brockmasur}]
Let $\{P\}_{\Gamma}$ be the set consisting of all thick of order zero subsets $\q{\gamma}$ for $\gamma$ any separating curve in $\mathcal{C}_{sep}(S_{2,1})$ as well as all quasi-convex thick of order one subsets of the form $\mathcal{X}(f,\tilde{\gamma},Q ).$  It is immediate that the union of the sets is coarsely the entire space.  In fact, this is true for just the union of the thick of order zero subsets in $\{P\}_{\Gamma}.$  Finally, to complete our argument we will show that any two subsets $P_{a},P_{b} \in \{P\}_{\Gamma}$ can be thickly chained together.  Without loss of generality we can assume that $P_{a}$ and $P_{b}$ are thick of order zero subsets $\q{\alpha},\q{\beta}$ for $\alpha,\beta$ in different connected components of $\mathcal{C}_{sep}(S_{2,1}).$  But then we can construct a sequence of separating curves $\alpha = \gamma_{1},...,\gamma_{n}=\beta$ such that each of the consecutive curves are either disjoint or intersect minimally (four times), \cite{schleimer}.  Hence, we can reduce the situation to showing that we can thickly connect $\q{\alpha}$ and $\q{\beta}$ where $\alpha,\beta$ are separating curves in different connected components of $\mathcal{C}_{sep}(S_{2,1})$ which intersect four times.  Fix any thick of order one sets $P_{c}= \mathcal{X}(f,\alpha,\alpha \lrcorner \beta), P_{d}= \mathcal{X}(f,\beta,\beta \lrcorner \alpha).$  By construction, we have the following chain of thickly intersecting subsets: $P_{a},P_{c},P_{d},P_{b}.$  Note that the fact that $P_{c}$ and $P_{d}$ have infinite diameter coarse intersection was precisely what was in fact shown in part (2) of the Brock-Masur proof recorded earlier in this section.
\end{proof}

\subsection{$\mathcal{T}(S_{2,1})$ is thick of order two}
\label{sec:thickorder2}
Recall the definition of the sets $\x{\alpha}$ in Equation \ref{eq:subpieces}.  Generalizing to the asymptotic cone, we define the following ultralimits:
\begin{equation} 
\label{eq:xbar}
\xbar{\alpha}=:\{ x_{\omega} \in \mathcal{P}_{\omega}(S) | x_{\omega} \mbox{has a representative $(x'_{i})$ with } x'_{i} \in \x{\alpha_{i}} \; \omega \mbox{-a.s}\}.
\end{equation}


\begin{lem} \label{lem:closed}
For $\overline{\alpha} \in \mathcal{C}^{\omega}_{sep}(S_{2,0}),$  $\xbar{\alpha} \subset \mathcal{P}_{\omega}(S_{2,1})$ is a closed set.
\end{lem}

\begin{proof}
Consider the 1-Lipschitz (hence continuous) projection $\pi_{\mathcal{P}_{\omega}(S_{2,0}) }\co \mathcal{P}_{\omega}(S_{2,1}) \ra \mathcal{P}_{\omega}(S_{2,0})$ which takes a representative sequence $(a_{i})$ for $a_{\omega}$ and maps it to a representative sequence of $(\pi_{\mathcal{P}(S_{2,0})}(a_{i}))$  where the map $\pi_{\mathcal{P}(S_{2,0}) }\co \mathcal{P}(S_{2,1}) \ra \mathcal{P}(S_{2,0})$ is the natural projection which forgets about the boundary component.  By definition $\left(\pi_{\mathcal{P}_{\omega}(S_{2,0})}\right)^{-1}(\qbar{\alpha} )=\xbar{\alpha}.$  By continuity, the result of the lemma follows from the fact that $\qbar{\alpha} \subset \mathcal{P}_{\omega}(S_{2,0})$ is closed.
\end{proof}

Recall Lemma \ref{lem:intpoint}.  In light of the notation developed in this section, as a special case we have the following corollary:
\begin{cor}\label{cor:intpoint2}
Assume $\overline{\alpha} \ne \overline{\beta} \in \mathcal{C}^{\omega}_{sep}(S_{2,1}),$ and let $\xbar{\alpha} =  \bigcup \qbar{\tilde{\alpha}}$ and $\xbar{\beta} = \cup  \qbar{\tilde{\beta}} .$  Then $ | \qbar{\tilde{\alpha}} \cap \qbar{\tilde{\beta}}  | \leq 1 $ and moreover, for $\overline{W},\overline{V} \in \eultra{S}$ with $\overline{\partial W} = \overline{\tilde{\alpha}},$ $\overline{\partial V} = \overline{\tilde{\beta}}$ we have: $$ \Phi_{\overline{W},x_{\omega}} (\qbar{\tilde{\beta}}) = \{ pt\}, \;\;  \Phi_{\overline{V},y_{\omega}} (\qbar{\tilde{\alpha}}) = \{ pt \},$$ where  $\Phi_{\overline{W},x_{\omega}}$ is the projection defined in Equation \ref{eq:projection}.
\end{cor}

The next theorem will be used to prove that the ultralimit of any thick of order zero subset $Z$ in $\mathcal{P}(S_{2,1})$ must be contained entirely inside a particular single closed set of the form $\xbar{\alpha}.$  Recall that by definition, a quasi-convex subspace $Z$ is thick of order zero if (i) it is wide, namely in every asymptotic cone $P_{\omega}(S_{2,1}),$ the subset corresponding to the ultralimit $$Z_{\omega} =: \{x_{\omega} \in \mathcal{P}_{\omega}(S_{2,1}) | x_{\omega} \mbox{ has a representative sequence $(x'_{i})$ with }x'_{i}\in  Z\; \omega\mbox{-a.s} \}$$ has the property that any two distinct points in $Z_{\omega}$ are not separated by a cut-point, and moreover (ii) $Z$ satisfies the nontriviality condition of every point being distance at most $c$ from a bi-infinite quasi-geodesic in $Z.$  

\begin{thm} \label{thm:zeropieces}
Let $(Z_{i}) \subset \mathcal{P}(S_{2,1})$ be any sequence of subsets, and let $\mathcal{P}_{\omega}(S_{2,1})$ be any asymptotic cone such that the ultralimit $Z_{\omega}$ does not have cut-points.  Then $Z_{\omega} \subset \xbar{\alpha},$ for some $\overline{\alpha} \in \mathcal{C}^{\omega}_{sep}(S_{2,0}).$  Moreover, if in any asymptotic cone $\mathcal{P}_{\omega}(S_{2,1}),$ the ultralimit $Z_{\omega}$ contains at least two points, then there exists a unique such $\overline{\alpha}$ satisfying the following condition: in any neighborhoods of $a_{\omega} \ne b_{\omega} \in Z_{\omega}$ there are points $a'_{\omega},b'_{\omega}$ with $d_{\mathbb{S}_{\omega}(S_{2,1})}(a'_{\omega},b'_{\omega})$ bounded, and such that each of the natural quasi-convex product regions $\qbar{C} \in \mathcal{P}_{\omega}(S)$ in a finite $\mathbb{S}_{\omega}(S_{2,1})$ chain from $a'_{\omega}$ to $b'_{\omega}$ are entirely contained in $\xbar{\alpha}.$
\end{thm}

Before proving Theorem \ref{thm:zeropieces} we first prove the following lemma.  
\begin{lem} \label{lem:zeropieces}
Let $(Z_{i}) \subset \mathcal{P}(S_{2,1})$ be any sequence of subsets, and let $\mathcal{P}_{\omega}(S_{2,1})$ be any asymptotic cone such that the ultralimit $Z_{\omega}$ is nontrivial and does not have cut-points.  Then $\forall a_{\omega} \ne b_{\omega} \in Z_{\omega},$ it follows that $a_{\omega},b_{\omega} \subset \xbar{\alpha},$ for some $\overline{\alpha} \in \mathcal{C}^{\omega}_{sep}(S_{2,0}).$ In fact, $\overline{\alpha}$ can be uniquely identified by the following condition: in any neighborhoods of $a_{\omega} \ne b_{\omega} \in Z_{\omega}$ there are points $a'_{\omega},b'_{\omega}$ with $d_{\mathbb{S}_{\omega}(S_{2,1})}(a'_{\omega},b'_{\omega})$ bounded, and such that each of the natural quasi-convex product regions $\qbar{C} \in \mathcal{P}_{\omega}(S)$ in a finite $\mathbb{S}_{\omega}(S_{2,1})$ chain from $a'_{\omega}$ to $b'_{\omega}$ are entirely contained in $\xbar{\alpha}.$
\end{lem}

\begin{rem}
Alternatively, as in the proof of Theorem \ref{thm:sepcomplexrelasymptotic} the unique characterization of the element $\overline{\alpha} \in \mathcal{C}^{\omega}_{sep}(S_{2,0})$ in Theorem \ref{thm:zeropieces} and Lemma \ref{lem:zeropieces} can be described as follows: in any neighborhoods of $a_{\omega} \ne b_{\omega} \in Z_{\omega}$ there are points $a'_{\omega},b'_{\omega}$ with representatives $(a'_{i}),(b'_{i})$ with $\lim_{\omega} d_{\mathcal{C}_{sep}(S_{2,1})}(a'_{i},b'_{i})$ bounded, and such that \uas a finite $\mathcal{C}_{sep}(S_{2,1})$ geodesic between $(a'_{i})$ and $(b'_{i})$ is contained in the connected components of $\mathcal{C}_{sep}(S_{2,1})$ corresponding to $\overline{\alpha}.$

\end{rem}
\begin{proof}
Since $Z_{\omega}$ does not have any cut points, by Theorem \ref{thm:pieces} and Remark \ref{rem:alternativecondition2}, in any neighborhoods of $a_{\omega}, b_{\omega}$ there exist points $a'_{\omega},b'_{\omega}$ with $d_{\mathbb{S}_{\omega}(S_{2,1})}(a'_{\omega},b'_{\omega})$ bounded.  That is, there is a finite chain of convex nontrivial product regions $\qbar{\tilde{\alpha}_{1}},$...,$\qbar{\tilde{\alpha}_{K}}$ such that $a'_{\omega} \in  \qbar{\tilde{\alpha}_{1}},$ $b'_{\omega} \in \qbar{\tilde{\alpha}_{K}},$ and $|\qbar{\tilde{\alpha}_{j}}\cap \qbar{\tilde{\alpha}_{j+1}}| \geq 2.$  As suggested by the notation, for all $j \in \{1,...,K\},$ $\pi_{\mathcal{C}^{\omega}(S_{2,0})} (\overline{\tilde{\alpha}}_{j}) = \overline{\alpha}$ for some fixed $\overline{\alpha} \in \mathcal{C}^{\omega}(S_{2,0})$ where the projection $$\pi_{\mathcal{C}^{\omega}(S_{2,0})} \co\mathcal{C}^{\omega}(S_{2,1}) \ra \mathcal{C}^{\omega}(S_{2,0})$$ is the extension to the ultrapower of the natural projection map which forgets about the boundary component.  In particular, all the natural convex product regions $\qbar{\tilde{\alpha}_{j}}$ in the chain connecting $a'_{\omega}, b'_{\omega} $ are contained in the set $\xbar{\alpha}.$

Since by Lemma \ref{lem:closed} the sets $\xbar{\alpha}$ are closed, in order to complete the proof of the lemma it suffices to show that for all $a'_{\omega},$ $b'_{\omega}$ in small enough neighborhoods of $a_{\omega}, \; b_{\omega},$ respectively, such that $d_{\mathbb{S}_{\omega}(S_{2,1})}(a'_{\omega},b'_{\omega})$ is bounded, we have that $a'_{\omega}$ and $b'_{\omega}$ are all always contained in the same set $\xbar{\alpha}$ as above.  Assume not, that is, assume that in any neighborhoods of $a_{\omega},b_{\omega}$ there exist points $a^{1}_{\omega},b^{1}_{\omega}$ and $a^{2}_{\omega},b^{2}_{\omega}$ such that $d_{\mathbb{S}_{\omega}(S_{2,1})}(a^{1}_{\omega},b^{1}_{\omega}) < \infty$ and $d_{\mathbb{S}_{\omega}(S_{2,1})}(a^{2}_{\omega},b^{2}_{\omega}) < \infty ,$ yet $a^{1}_{\omega},b^{1}_{\omega} \in \xbar{\alpha}$ while $a^{2}_{\omega},b^{2}_{\omega} \in \xbar{\beta} $ where $\overline{\alpha}\ne \overline{\beta}.$  In particular, we can assume that 
 $a^{1}_{\omega},b^{1}_{\omega}$ lie in an $r$-neighborhood of $a_{\omega}$ and $a^{2}_{\omega},b^{2}_{\omega}$ lie in an $r$-neighborhood of $b_{\omega}$ where $r\geq 0$ is a constant such that open $r$-neighborhoods of $a_{\omega},b_{\omega}$ are path connected.  In addition, we can assume that $2r<d_{\mathcal{P}_{\omega}(S)}(a_{\omega},b_{\omega}).$  See Figure \ref{fig:chains} for an illustration of this.

\begin{figure}[htpb]
\centering
\includegraphics[height=7 cm]{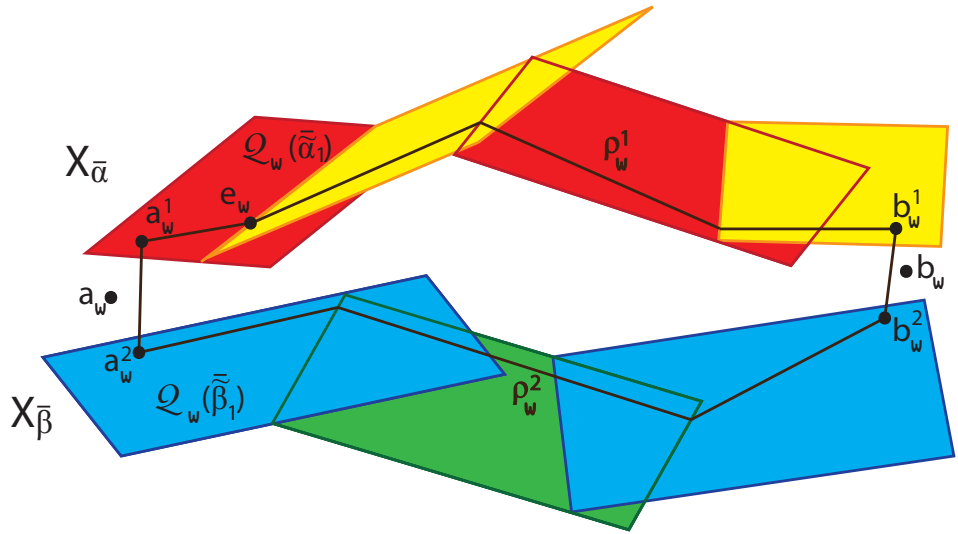}
\caption[$\mathcal{T}(S_{2,1})$ is not thick of order one.]{In neighborhoods of $a_{\omega},b_{\omega}$ there exist points $a^{1}_{\omega},b^{1}_{\omega}$ and $a^{2}_{\omega},b^{2}_{\omega},$ respectively, such that $d_{\mathbb{S}_{\omega}(S)}(a^{1}_{\omega},b^{1}_{\omega}) < \infty ,$ $d_{\mathbb{S}_{\omega}(S)}(a^{2}_{\omega},b^{2}_{\omega}) < \infty ,$ yet $a^{1}_{\omega},b^{1}_{\omega} \in \xbar{\alpha}$ while $a^{2}_{\omega},b^{2}_{\omega} \in \xbar{\beta} $ where $\overline{\alpha}\ne \overline{\beta}.$  This situation cannot occur in $\mathcal{P}_{\omega}(S_{2,1}).$  }\label{fig:chains}
\end{figure}

Let $\qbar{\tilde{\alpha}_{1}},$...,$\qbar{\tilde{\alpha}_{m}}$ be a finite chain of convex nontrivial product regions in $\xbar{\alpha}$ connecting $a^{1}_{\omega}$ and $b^{1}_{\omega}.$  Moreover, as in Theorem \ref{thm:pieces} there is a quasi-geodesic path $\rho^{1}_{\omega},$ the ultralimit of hierarchy paths, through the product regions connecting $a^{1}_{\omega}$ and $b^{1}_{\omega}.$  Similarly, let $\qbar{\tilde{\beta}_{1}},$...,$\qbar{\tilde{\beta}_{n}}$ be a finite chain of convex nontrivial product regions in $\xbar{\beta}$ connecting $a^{2}_{\omega}$ and $b^{2}_{\omega},$ and let $\rho^{2}_{\omega}$ be a quasi-geodesic path through the product regions connecting $a^{2}_{\omega}$ and $b^{2}_{\omega}.$  By omitting product regions as necessary and using properties of hierarchies in Theorem \ref{thm:hierarchy} we can assume that initial product region  $\qbar{\tilde{\alpha}_{1}}$ of the path $\rho^{1}_{\omega}$ has the property that $\rho^{1}_{\omega}$ exits the product region $\qbar{\tilde{\alpha}_{1}}$ once at a point $e_{\omega} \ne a^{1}_{\omega}.$  By Lemma \ref{lem:sepfiber}, there is some $\overline{W} \in \seultra{S}$ which is \uas a connected component of $\overline{S \setminus \tilde{\alpha}_{1}},$ such that $\pi_{\mathcal{P}_{\omega}(\overline{W})}(a^{1}_{\omega}) \ne \pi_{\mathcal{P}_{\omega}(\overline{W})}(e_{\omega}) \in F_{\overline{W},a^{1}_{\omega}}.$  

By our assumptions, $a^{1}_{\omega}$ and $a^{2}_{\omega}$ are connected by a path that remains entirely inside an $r$-neighborhood of $a_{\omega}.$  Let $[a^{1}_{\omega},a^{2}_{\omega}]$ denote such a path.  Similarly, let $[b^{1}_{\omega},b^{2}_{\omega}]$ denote a path between the points $b^{1}_{\omega}$ and $b^{2}_{\omega}.$  We can assume that $(a^{1}_{\omega},a^{2}_{\omega}]$ and $(b^{1}_{\omega},b^{2}_{\omega}]$ are contained in $\mathcal{P}_{\omega}(S_{2,1}) \setminus \qbar{\tilde{\alpha}_{1}}.$  If not, we can replace $a^{1}_{\omega}$ and/or $b^{1}_{\omega}$ with points closer to $a^{2}_{\omega}$ and/or $b^{2}_{\omega}$ respectively such that this is the case.  

Consider the closed pentagon $P$ with vertices $\{a^{1}_{\omega},e_{\omega}, b^{1}_{\omega},b^{2}_{\omega},a^{2}_{\omega}\}$ and edges $$\rho^{1}_{\omega}|_{[a^{1}_{\omega},e_{\omega}]}, \rho^{1}_{\omega}|_{[e_{\omega},b^{1}_{\omega}]}, [b^{1}_{\omega},b^{2}_{\omega}], \rho^{2}_{\omega},[a^{1}_{\omega},a^{2}_{\omega}]$$
It should be noted that some sides of the pentagon may be trivial, although this does not affect the argument.  Applying the continuous projection $\Phi_{\overline{W},x_{\omega}}$ of Theorem \ref{thm:locconstant} to the pentagon $P,$ we have $\Phi_{\overline{W},x_{\omega}}(e_{\omega})=\Phi_{\overline{W},x_{\omega}}(b^{1}_{\omega})=\Phi_{\overline{W},x_{\omega}}(b^{2}_{\omega}).$  Similarly, $\Phi_{\overline{W},x_{\omega}}(a^{1}_{\omega})=\Phi_{\overline{W},x_{\omega}}(a^{2}_{\omega})$ as by construction the edges $\rho^{1}_{\omega}|_{[e_{\omega},b^{1}_{\omega}]}, [b^{1}_{\omega},b^{2}_{\omega}]$ and $[a^{1}_{\omega},a^{2}_{\omega}]$ are contained in $\mathcal{P}_{\omega}(S_{2,1}) \setminus P_{\overline{W},x_{\omega}}.$  Furthermore, by Corollary \ref{cor:intpoint2} and continuity of the projection, $\Phi_{\overline{W},x_{\omega}}(\rho^{2}_{\omega})$ is a single point and is in fact equal to $\Phi_{\overline{W},x_{\omega}}(a^{2}_{\omega})=\Phi_{\overline{W},x_{\omega}}(b^{2}_{\omega}).$  Putting things together we have 
$$\Phi_{\overline{W},x_{\omega}}(e_{\omega})=\Phi_{\overline{W},x_{\omega}}(b^{1}_{\omega})=\Phi_{\overline{W},x_{\omega}}(b^{2}_{\omega})=\Phi_{\overline{W},x_{\omega}}(a^{2}_{\omega})=\Phi_{\overline{W},x_{\omega}}(a^{1}_{\omega})$$
However, this is a contradiction to our assumption that  $\Phi_{\overline{W},x_{\omega}}(a^{1}_{\omega}) \ne \Phi_{\overline{W},x_{\omega}}(e_{\omega}),$ thus completing the proof.
\end{proof}

Using the proof of Lemma \ref{lem:zeropieces}, presently we prove Theorem \ref{thm:zeropieces}.
\begin{proof} [Proof of Theorem \ref{thm:zeropieces}]
By Lemma \ref{lem:zeropieces} we know that given any two distinct points $a_{\omega},b_{\omega} \in Z_{\omega},$ the points $a_{\omega},b_{\omega}$ are contained in a common subset $\xbar{\alpha}$ where $\overline{\alpha} \in \mathcal{C}^{\omega}_{sep}(S_{2,0})$ is such that in any neighborhoods of $a_{\omega} \ne b_{\omega} \in Z_{\omega}$ there are points $a'_{\omega},b'_{\omega}$ with $d_{\mathbb{S}_{\omega}(S_{2,1})}(a'_{\omega},b'_{\omega})$ bounded, and such that each of the natural quasi-convex product regions $\qbar{C}$ 
in a finite $\mathbb{S}_{\omega}(S_{2,1})$ chain from $a'_{\omega}$ to $b'_{\omega}$ are entirely contained in $\xbar{\alpha}.$ 

Let $c_{\omega} \in Z_{\omega}$ be any third point in $Z_{\omega},$ (possibly the same as $a_{\omega}$ or $b_{\omega}$).  Similarly, it follows that the points $a_{\omega},c_{\omega}$ ($b_{\omega},c_{\omega}$) are contained in a common subset $\xbar{\beta}$ ($\xbar{\gamma}$) where $\overline{\beta}$ ($\overline{\gamma}$) is an element of $\mathcal{C}^{\omega}_{sep}(S_{2,0})$ such that in any neighborhoods of $a_{\omega}$ and $c_{\omega}$ ($b_{\omega}$ and  $c_{\omega}$) there are points $a'_{\omega},c'_{\omega}$ ($b'_{\omega},c'_{\omega}$) with $d_{\mathbb{S}_{\omega}(S_{2,1})}(a'_{\omega},c'_{\omega})$ bounded ($d_{\mathbb{S}_{\omega}(S_{2,1})}(b'_{\omega},c'_{\omega})$ bounded), and such that each of the natural quasi-convex product regions $\qbar{C} \in \mathcal{P}_{\omega}(S)$ in a finite $\mathbb{S}_{\omega}(S_{2,1})$ chain from $a'_{\omega}$ to $c'_{\omega}$ ($b'_{\omega}$ to $c'_{\omega}$) are entirely contained in $\overline{\beta}$ ($\overline{\gamma}).$  But then, considering the triangle between the points $a'_{\omega},b'_{\omega},c'_{\omega}$ and using the same projection arguments in Lemma \ref{lem:zeropieces} to generalize the contradiction argument with the pentagon, it follows that $\overline{\alpha}=\overline{\beta}=\overline{\gamma}.$  Notice if $c_{\omega}$ is the same as $a_{\omega}$ or $b_{\omega},$ the proof is identical to the proof in Lemma \ref{lem:zeropieces}.  

Since $c_{\omega}$ is arbitrary, it follows that $Z_{\omega} \subset \xbar{\alpha}$ where $\overline{\alpha}$ is uniquely determined by the property described in the statement of the theorem.
%
%
\end{proof} 


As a corollary of the proof of Lemma \ref{lem:zeropieces}, we have the following corollary:
\begin{cor} \label{cor:intpieces}
Let $(Z_{i}),(Z_{i}') \subset \mathcal{P}(S_{2,1})$ be any sequences subsets, and let $\mathcal{P}_{\omega}(S_{2,1})$ be an asymptotic cone such that $Z_{\omega}, Z'_{\omega}  \subset \mathcal{P}_{\omega}(S_{2,1})$ each one contains at least two points, and each one has no cut-points.  As in Theorem \ref{thm:zeropieces} assume that $Z_{\omega} \subset \xbar{\alpha}$ and $Z'_{\omega} \subset \xbar{\beta}$  for some $\overline{\alpha}, \overline{\beta} \in \mathcal{C}^{\omega}_{sep}(S_{2,0}),$ such that \uas $\alpha_{i} \ne \beta_{i},$ then: 
$$|Z_{\omega} \cap Z'_{\omega}| \leq 1.$$
In particular, if the asymptotic cone $\mathcal{P}_{\omega}(S_{2,1})$ has a constant base point, and the sequences of subsets $(Z_{i})=\overline{Z}$ and $(Z'_{i})=\overline{Z'}$ are constant and quasi-convex, then the subsets $Z$ and $Z'$ have bounded coarse intersection.  
\end{cor}

\begin{proof}
We will show $|Z_{\omega} \cap Z'_{\omega}| \leq 1$ by contradiction.  That is, assume $a_{\omega} \ne b_{\omega} \in \left(Z_{\omega} \cap Z'_{\omega} \right).$  By Theorem \ref{thm:pieces}, in any neighborhoods of $a_{\omega}, b_{\omega}$ there exist points $a^{1}_{\omega},b^{1}_{\omega}$ and $a^{2}_{\omega},b^{2}_{\omega},$ such that $d_{\mathbb{S}_{\omega}(S_{2,1})}(a^{1}_{\omega},b^{1}_{\omega}) < \infty$ and $d_{\mathbb{S}_{\omega}(S_{2,1})}(a^{2}_{\omega},b^{2}_{\omega}) < \infty ,$ yet $a^{1}_{\omega},b^{1}_{\omega} \in \xbar{\alpha}$ while $a^{2}_{\omega},b^{2}_{\omega} \in \xbar{\beta} $ where $\overline{\alpha}\ne \overline{\beta}.$  Precisely this situation was shown to be impossible in the proof of Lemma \ref{lem:zeropieces}.  

Next, consider the special case of the first part of the Corollary in which the asymptotic cone $\mathcal{P}_{\omega}(S_{2,1})$ has a constant base point, and the sequences of subsets $(Z_{i})=\overline{Z}$ and $(Z'_{i})=\overline{Z'}$ are constant and quasi-convex.  Then the coarse intersection $\overline{Z} \hat{\cap} \overline{Z'}$ is the constant quasi-convex, and hence connected, sequence of subsets $\overline{Z \hat{\cap} Z'}.$  Since our asymptotic cone has a constant base point, assuming $Z \hat{\cap} Z'$ is nontrivial (if not then we are done), its ultralimit $\overline{Z \hat{\cap} Z'}=(Z \hat{\cap} Z')_{\omega}$ in the asymptotic cone is similarly nontrivial.  That is, in the asymptotic cone  $(Z\hat{\cap} Z')_{\omega}$ contains at least - and hence by the first part exactly- one point, namely the point in the cone with constant representative sequence.  It follows that the diameter of the connected coarse intersection $Z \hat{\cap} Z'$ is sublinear in $s_{i}.$  On the other hand, since the diameter of the coarse intersection $Z \hat{\cap} Z',$ is not only sublinear but also constant, it follows that $Z$ and $Z'$ have bounded coarse intersection. 
\end{proof}

We are now prepared to prove the following theorem.    

\begin{thm} \label{thm:thick2}
$\mathcal{T}(S_{2,1})$ is thick of order two.
\end{thm}

\begin{proof}
Since thickness is a quasi-isometry invariant property, \cite{bdm}, it suffices to prove that $\mathcal{P}(S_{2,1})$ is thick of order two.  In Section \ref{sec:atmosttwo} we showed that $\mathcal{P}(S_{2,1})$ is thick of order at most two and at least one.  Hence, it suffices to show that $\mathcal{P}(S_{2,1})$ is not thick of order one.  In fact, we will show that any thick of order one subset is entirely contained inside a \emph{nontrivially proper subset} of the entire pants complex (that is, a subset which has infinite Hausdorff from the entire pants complex).

Fix an asymptotic cone $\mathcal{P}_{\omega}(S_{2,1})$ with a constant base point and scaling sequence $s_{i}.$  Note that since $\mathcal{P}(S_{2,1})$ is connected, for any $q\in \mathcal{P}(S_{2,1}),$ the constant sequence $\overline{q}$ all represent the same base point of the asymptotic cone $\mathcal{P}_{\omega}(S_{2,1}).$

Let $Z$ be any thick of order zero subset in $\mathcal{P}(S_{2,1}).$  By hypothesis, $Z$ coarsely contains a bi-infinite quasi-geodesic through any point.  Fix some point $z\in Z,$ and some quasi-geodesic ray $\gamma$ beginning near $z$ and remaining in $Z.$  Then for every $s_{i},$ set $y_{i}=\gamma(s_{i}) \in Z.$  By construction, in the asymptotic cone the sequences $\overline{z}$ and $(y_{i})$ represent distinct points contained in $Z_{\omega} \subset \mathcal{P}_{\omega}(S_{2,1}).$  In particular, we have just shown that every thick of order zero subset $Z \subset \mathcal{P}(S_{2,1})$ has ultralimit $Z_{\omega}$ containing at least two distinct points in the asymptotic cone $\mathcal{P}(S_{2,1}).$  By Theorem \ref{thm:zeropieces} it follows that every thick of order zero subset $Z$ in $\mathcal{P}(S_{2,1})$ can be assigned a unique element $\overline{\alpha} \in \mathcal{C}^{\omega}_{sep}(S_{2,0}).$  Moreover, Corollary \ref{cor:intpieces} implies that a necessary condition for any two thick of order zero subsets $Z,Z'$ to be thickly chained together, as in condition (ii) of \ref{defn:thick}, is that the two thick of order zero subsets $Z,Z'$ are assigned the same element $\overline{\alpha} \in \mathcal{C}^{\omega}_{sep}(S_{2,1}).$

It follows that any thick of order one subset $Y$ of the space $\mathcal{P}(S_{2,1})$ can consist of at most the union of thick of order zero subsets with the same labels $\overline{\alpha} \in \mathcal{C}^{\omega}_{sep}(S_{2,0}).$  Hence, the ultralimit $Y_{\omega}$ in the asymptotic cone $\mathcal{P}_{\omega}(S_{2,1})$ is entirely contained inside the subset $\xbar{\alpha}$ which we will see is a proper subset of $\mathcal{P}_{\omega}(S_{2,1}).$  The proof of the Theorem then follows from the observation that if a subset $Y \subset X$ has finite Hausdorff distance from $X,$ then in any asymptotic cone the ultralimit $Y_{\omega}=X_{\omega}.$

To see that $\xbar{\alpha}$ is a proper subset of $\mathcal{P}_{\omega}(S_{2,1}),$ notice that under under the surjective projection $\pi\co \mathcal{P}_{\omega}(S_{2,1}) \surj \mathcal{P}_{\omega}(S_{2,0}),$ the subset $\xbar{\alpha}$ is mapped into the natural quasi-convex product region $\qbar{\alpha},$ a proper subset of $\mathcal{P}_{\omega}(S_{2,0}).$    %
\end{proof}

\begin{rem}
Theorem \ref{thm:thick2} completes the thickness classification of the pants complexes of all surfaces of finite type which is presented in Table \ref{table:thick}.  Moreover, among all surfaces of finite type of equal or higher complexity, $S_{2,1}$ is the only surface such that its pants complex is not thick of order one.\end{rem}

\begin{table}
\begin{center}
\begin{tabular}{|c||c|c|c|c|c|c|c}
$\vdots$ & \cellcolor[rgb]{0,1,1} $\vdots$& \cellcolor[rgb]{0,1,1} $\vdots$ & \cellcolor[rgb]{0,1,1} $\vdots$ &  \cellcolor[rgb]{0,1,1} $\vdots$ & \cellcolor[rgb]{0,1,1} $\vdots$ & \cellcolor[rgb]{0,1,1}$\vdots$ & \cellcolor[rgb]{0,1,1} $\iddots$  \\
\hline
$7$ & \cellcolor[rgb]{0,1,1} T1& \cellcolor[rgb]{0,1,1} T1 & \cellcolor[rgb]{0,1,1} T1  & \cellcolor[rgb]{0,1,1} T1 & \cellcolor[rgb]{0,1,1} T1& \cellcolor[rgb]{0,1,1} T1 & \cellcolor[rgb]{0,1,1}$\hdots$  \\
\hline
$6$ &\cellcolor[rgb]{1,1,0} RH&    \cellcolor[rgb]{0,1,1} T1& \cellcolor[rgb]{0,1,1} T1 & \cellcolor[rgb]{0,1,1} T1  & \cellcolor[rgb]{0,1,1} T1& \cellcolor[rgb]{0,1,1} T1 & \cellcolor[rgb]{0,1,1}$\hdots$ \\ 
\hline
$5$ &  \cellcolor[rgb]{1,.5,0} H&     \cellcolor[rgb]{0,1,1} T1 & \cellcolor[rgb]{0,1,1} T1 & \cellcolor[rgb]{0,1,1} T1 & \cellcolor[rgb]{0,1,1} T1& \cellcolor[rgb]{0,1,1} T1 & \cellcolor[rgb]{0,1,1} $\hdots$   \\ 
\hline
$4$ &  \cellcolor[rgb]{1,.5,0} H& \cellcolor[rgb]{0,1,1} T1     & \cellcolor[rgb]{0,1,1} T1  & \cellcolor[rgb]{0,1,1} T1 & \cellcolor[rgb]{0,1,1} T1 & \cellcolor[rgb]{0,1,1} T1  & \cellcolor[rgb]{0,1,1} $\hdots$   \\
\hline
$3$ &    &\cellcolor[rgb]{1,1,0} RH& \cellcolor[rgb]{0,1,1} T1  & \cellcolor[rgb]{0,1,1} T1 & \cellcolor[rgb]{0,1,1} T1 & \cellcolor[rgb]{0,1,1} T1 & \cellcolor[rgb]{0,1,1} $\hdots$  \\
\hline
$2$ &    &  \cellcolor[rgb]{1,.5,0} H & \cellcolor[rgb]{0,1,1} T1    &\cellcolor[rgb]{0,1,1} T1 & \cellcolor[rgb]{0,1,1} T1 & \cellcolor[rgb]{0,1,1} T1 & \cellcolor[rgb]{0,1,1} $\hdots$    \\
\hline
$1$ &    &  \cellcolor[rgb]{1,.5,0} H & \cellcolor[rgb]{.5,1,0} \textbf{T2}   & \cellcolor[rgb]{0,1,1} T1  & \cellcolor[rgb]{0,1,1} T1 & \cellcolor[rgb]{0,1,1} T1 & \cellcolor[rgb]{0,1,1} $\hdots$  \\ 
\hline
$0$ &    &     & \cellcolor[rgb]{1,1,0}RH    & \cellcolor[rgb]{0,1,1} T1 & \cellcolor[rgb]{0,1,1} T1 & \cellcolor[rgb]{0,1,1} T1& \cellcolor[rgb]{0,1,1} $\hdots$  \\ 
\hline
\hline
$n \uparrow \; g \rightarrow  $  & $0$  & $1$ &$2$  & $3$ & $4$ & $5$  & $\hdots$ \\  
\hline   
\end{tabular}
\end{center}
\caption[Hyperbolicity/Thickness classification of Teichm\"uller spaces]{Hyperbolicity/Thickness classification of Teichm\"uller spaces for all surfaces.  H=hyperbolic, RH=relatively hyperbolic, T1=thick of order one, and T2=thick of order two.}
 \label{table:thick}
\end{table}

%
%
%
%

\subsection{$\mathcal{T}(S_{2,1})$ has superquadratic divergence}
\label{sec:divergence}
Informally the \emph{divergence} of a metric space, a notion introduced by Gromov, is a measure of inefficiency of detours paths.  More specifically, divergence quantifies the cost of going from a point $x$ to a point $y$ in a (typically one-ended geodesic) metric space X while avoiding a metric ball based at a point $z.$  Throughout the literature there are a couple of closely related definitions of divergence that emerge based on stipulations regarding the points $x,y,z.$  See \cite{drutumozessapir} for a comparison of various definitions and criterion for when the different definitions agree.  We will consider the following definition of divergence which is a lower bound on all other definitions of divergence in the literature.  In particular, it follows that the novel result in this section regarding the superquadratic divergence of $\mathcal{T}(S_{2,1})$ remains true for any definition of divergence.

\begin{defn}[Divergence] \label{defn:divergence}  Let $\gamma$ be a coarsely arc length parameterized bi-infinite quasi-geodesic in a one-ended geodesic metric space.  Then the \emph{divergence along $\gamma$}, denoted $div(\gamma,\epsilon)$ is defined to be the growth rate of the function $$d_{X \setminus B_{\epsilon r}(\gamma(0))}(\gamma(-r),\gamma(r))$$ with respect to $r$ where the scalar $\epsilon>0$ is chosen so that $\gamma(\pm r) \not \in B_{\epsilon r}(\gamma(0)).$  As divergence is independent of the choice of a small $\epsilon,$ we will often omit $\epsilon$ from the notation.  Divergence can be similarly associated to a sequence of quasi-geodesic segments $\gamma_{i}.$  The \emph{divergence of $X$} denoted $div(X)$ is defined to be $\max_{\gamma,\epsilon} div(\gamma,\epsilon).$   
\end{defn}   


The proof of following lemma is straightforward.

\begin{lem} [\cite{drutumozessapir} Lemma 3.15] \label{lem:superlinear}
Let X be a geodesic metric space, $X_{\omega}$ any asymptotic cone, and assume $a_{\omega} \ne b_{\omega} \in X_{\omega}$ have representative sequences $(a_{i}),(b_{i}),$ respectively.  Then $X_{\omega}$ has a global cut-point separating $a_{\omega}$ and $b_{\omega}$ if and only if \uas the sequence of geodesics $[a_{i},b_{i}]$ has superlinear divergence.
\end{lem}  

The plan for the rest of the section is to show that $\mathcal{T}(S_{2,1})$ has at least superquadratic and at most cubic divergence.  First we prove the lower bound, and then see that the upper bound follows from Theorem \ref{thm:thick2} in conjunction with results in \cite{behrstockdrutu}.

\subsection{$\mathcal{T}(S_{2,1})$ has at least superquadratic divergence}

Recall Theorem \ref{thm:hypquasi} which characterizes contracting quasi-geodesics in CAT(0) spaces.  Presently, we will provide a standard argument for the following small ingredient of the theorem as it serves as motivation for ideas in this section.

\begin{lem} \label{lem:scatleastquad}
A (b,c)--contracting quasi-geodesic $\gamma$ in a geodesic metric space $X$ has at least quadratic divergence.
\end{lem}

\begin{proof}
To streamline the exposition we will assume $\gamma$ is a strongly contracting geodesic, although the same argument carries through for $\gamma$ a (b,c)--contracting quasi-geodesic.  Recall that by Definition \ref{def:bccont} since $\gamma$ is strongly contracting geodesic there exists a constant $c$ such that $\forall x,y \in X$ if $d(x,y)<d(x,\gamma)$ then $d(\pi_{\gamma}(x),\pi_{\gamma}(y))<c,$ where the map $\pi_{\gamma}\co X \ra 2^{\gamma}$ is a nearest point projection.  To prove the lemma we will consider an arbitrary detour path $\alpha_{r}$ connecting $\gamma(-r)$ and $\gamma(r)$ while avoiding the metric ball $B_{r}(\gamma(0)),$ and show that the length of $\alpha_{r}$ is at least a quadratic function in $r.$  

Presently we will discretize the detour path in terms of nearest point projections on to the subgeodesic $[\gamma(-r/2),\gamma(r/2)].$  Specifically, for each $$j \in \{-\lfloor \frac{r}{2c} \rfloor, ...,-1, 0, 1, ..., \lfloor \frac{r}{2c} \rfloor \},$$ fix $z_{r}^{jc} \in \alpha_{r}$ such that $z^{jc}_{r} \in \pi^{-1}_{\gamma}(\gamma(jc)).$  Notice that by construction $d(z_{r}^{jc},\gamma) \geq \frac{r}{2}.$  Furthermore, since $d(\pi_{\gamma}(z_{r}^{jc}),\pi_{\gamma}(z_{r}^{(j+1)c}))=c,$ by the strongly contracting property it follows that $$d(z_{r}^{jc},z_{r}^{(j+1)c})\geq d(z_{r}^{jc},\gamma) \geq \frac{r}{2}.$$  Putting things together, the following inequality gives the desired lower bound on the length of the detour path $\alpha_{r}:$
$$|\alpha_{r}|  \geq \sum_{j=1}^{\lfloor \frac{r}{c} \rfloor} d(z_{r}^{(j-1)c},z_{r}^{jc}) \geq \sum_{j=1}^{\lfloor \frac{r}{c} \rfloor} \frac{r}{2}  \geq \frac{r^{2}}{2c} -1.$$
Since $c$ is a uniform constant, the statement of the lemma follows.  
\end{proof}

The following lemma is closely related to ideas in \cite{bestvinafujiwara} regarding the thinness of polygons with edges along a contracting contracting geodesic.
\begin{lem}\label{lem:quasigeo}
Using the notation from Lemma \ref{lem:scatleastquad}, let $\sigma=\sigma_{r}^{jc}$ be the concatenated path $$[z_{r}^{jc},\gamma(jc)] \cup [\gamma(jc),\gamma((j+1)c)] \cup [\gamma((j+1)c),z_{r}^{(j+1)c}],$$ then $\sigma$ is a (2,c)-quasi-geodesic.  
\end{lem}

\begin{proof}  Let $x,y$ be any points along $\sigma.$  If $x,y \in [z_{r}^{jc},\gamma(jc)],$ then it is immediate that $$d(x,y)= d_{\sigma}(x,y),$$ where $d_{\sigma}(x,y)$ represents the distance along $\sigma$ from $x$ to $y.$  In particular, for any points $x,y \in [z_{r}^{jc},\gamma(jc)],$ the $(2,c)$ quasi-isometric inequality is trivially satisfied.  Similarly, the same conclusion holds for  $x,y \in [z_{r}^{(j+1)c},\gamma((j+1)c)]$ or $x,y \in [\gamma(jc),\gamma((j+1)c)].$  Moreover, since $|[\gamma(jc),\gamma((j+1)c)]|=c,$ for the cases $x \in [z_{r}^{jc},\gamma(jc)] \cup  [\gamma((j+1)c),z_{r}^{(j+1)c}]$ and $y \in [\gamma(jc),\gamma((j+1)c)]$ (or vice versa) the $(2,c)$ quasi-isometric inequality is similarly satisfied. 

Hence, we can assume $x \in [z_{r}^{jc},\gamma(jc)]$ and $y \in [\gamma((j+1)c),z_{r}^{(j+1)c}].$  Since $x$ and $y$ have nearest point projections onto $\gamma$ which are distance $c$ apart, by (1,c)-contraction of $\gamma$ we have: $$\max\{ d(x, \gamma(jc)), d(y, \gamma((j+1)c))\}= D \leq d(x,y).$$  Specifically, since 
$d(\gamma(jc), \gamma((j+1)c))= d(\gamma((j+1)c), \gamma(jc) ) =c$ the definition of (1,c)-contraction (Definition \ref{def:bccont}) implies that: $$d(x,y)\geq d(x, \gamma(jc)) \mbox{ and similarly }d(y,x)\geq d(y, \gamma((j+1)c)).$$  But then, we have the following inequality completing the proof:
$$ d(x,y) \leq d_{\sigma}(x,y) = d(x, \gamma(jc)) +  c+  d(y, \gamma((j+1)c)) \leq  2D+c \leq 2d(x,y) + c.$$
\end{proof}

\begin{rem}  Note that in the special case of $\gamma$ a strongly contracting quasi-geodesic in Lemma \ref{lem:quasigeo} we showed that the piecewise geodesic paths $\sigma^{jc}_{r}$ are (2,c)-quasi-geodesics.  More generally, for $\gamma$ a (b,c)-contracting quasi-geodesic the same argument shows that the piecewise geodesic paths $\sigma^{jc}_{r}$ are similarly all the quasi-geodesics with uniformly bounded quasi-isometry constants.  
\end{rem}
%

We will now aim toward proving the following main theorem of this subsection.  
\begin{thm} \label{thm:atleastsuperquad}
$\mathcal{T}(S_{2,1})$ has at least superquadratic divergence.
\end{thm}
 
Recall in the proof of Lemma \ref{lem:scatleastquad} we showed a contracting quasi-geodesic has at least quadratic divergence by showing that in order for a detour path to have more than a uniformly bounded ``shadow'' (i.e. nearest point projection set) onto $\gamma$ the detour path must travel at least a linear distance.  In other words, the at least quadratic divergence was a consequence of the fact that the detour path had to travel a linear amount of at least linear distances.  To prove Theorem \ref{thm:atleastsuperquad} we will construct a quasi-geodesic such that a detour path must travel a linear amount of at least superlinear distances.  

More specifically, recall the sequence of quasi-geodesic segments $\{\sigma^{jc}_{r}\}_{r}$ which coincide with $\gamma$ along the segment $[\gamma(jc),\gamma((j+1)c)]$ in the proof of Lemma \ref{lem:scatleastquad}.  By definition, the portion of the detour path $\alpha_{r}$ which connects the endpoints of $\sigma^{jc}_{r}$ cannot fellow travel with $\sigma^{jc}_{r}.$  In fact, by construction $\alpha_{r}$ lies outside of the ball $N_{r/2}( [\gamma(jc),\gamma((j+1)c)]).$  In particular, in order to prove that $\gamma$ has at least superquadratic divergence, we will show that the sequence of quasi-geodesic segments $\{\sigma^{jc}_{r}\}_{r}$ for almost all $j$ have superlinear divergence.  

To be sure, showing that a detour path must travel a linear amount of superlinear distances without controlling the degree of superlinearity does not ensure superquadratic divergence.  Specifically, consider the following example.

\begin{ex} \label{ex:degenerate}
Let $\alpha_{r}$ be a sequence of paths with each $\alpha_{r}$ partitioned into $r$ subsegments $\{\tau_{r,j}\}_{j=1}^{r}$ such that for any fixed $j,$ the length of the sequence of segments $\{\tau_{r,j}\}_{r}$ is superlinear in $r.$  Then, 
\begin{eqnarray*}
|\alpha_{r}| &\geq& \sum_{j=1}^{r} |\tau_{r,j}| = \sum_{j=1}^{r} r \epsilon_{j}(r) \geq r^{2} \min^{r}_{j=1}(\epsilon_{j}(r)).
\end{eqnarray*}  
where for any fixed $j,\; \lim_{r}\epsilon_{j}(r) \ra \infty.$
Taking the limit, it does not necessarily follow that $|\alpha_{r}|$ is superquadratic in $r.$  For example, if we define the functions $\epsilon_{j}(r)=1$ if  $r\leq j$ and $\epsilon_{j}(r)=r$ otherwise.  Notice that $\min^{r}_{j=1}(\epsilon_{j}(r))=1,$ and hence it merely follows that $|\alpha_{r}|$ can be bounded below by $r^{2}.$
\end{ex}

Nonetheless, the potential problem highlighted in Example \ref{ex:degenerate} will be avoided  by using the periodicity of $\gamma$ in conjunction with a contradiction argument.  

As in the proof of Lemma \ref{lem:scatleastquad} let $\gamma$ be a contracting quasi-geodesic, let $\alpha_{r}$ be a sequence of detour paths avoiding balls $B_{r}(\gamma(0)),$ and let $z^{jc}_{r}$ denote fixed points on $\alpha_{r}$ which have nearest point projections to $\gamma(jc).$  Then, for all $ jc \in \Z c$ we obtain sequences of points $\overline{z^{jc}}=\{z^{jc}_{r}\}_{r=2c|j|}^{\infty},$ and similarly sequences of quasi-geodesic paths $\overline{\sigma^{jc}}=\{\sigma^{jc}_{r}\}_{r=2c|j|}^{\infty}.$  Let $\tau^{jc}_{r}$ denote the restriction of the quasi-geodesics $\sigma^{jc}_{r}$ to the intersection $\sigma^{jc}_{r} \cap B_{r/2}(\gamma(jc)),$ and by abuse of notation refer to the endpoints of $\tau^{jc}_{r}$ by $z^{jc}_{r}$ and $z^{(j+1)c}_{r}.$  In fact, by even further abuse of notation, let $\{z^{jc}_{r}\}_{r}$ represent any sequence of points of distance $r/2$ from $\gamma$ such that the nearest point projection of $z^{jc}_{r}$ onto $\gamma$ is $\gamma(jc),$ and similarly, let $\tau^{jc}_{r}$ denote the quasi-geodesic between consecutive points $z^{jc}_{r}$ and $z^{(j+1)c}_{r},$ given by the concatenation: $$\tau^{jc}_{r}=:[z^{jc}_{r}, \gamma(jc)] \cup [\gamma(jc), \gamma((j+1)c)] \cup [\gamma((j+1)c), z^{(j+1)c}_{r}].$$     

\begin{lem} \label{lem:periodic}
With the notation from above, assume in addition that $\gamma$ is a periodic quasi-geodesic such that for all fixed $j,$ the sequence of quasi-geodesic segments $\{\tau^{jc}_{r}\}_{r}$ has divergence which is superlinear in $r,$ the natural numbers.  Then, $\gamma$ has superquadratic divergence.  Similarly, the same conclusion holds if $\gamma$ is a periodic quasi-geodesic such that there is a constant $C$ such that for any fixed $j,$ and any consecutive sequence of sequences of quasi-geodesic segments $$\{\tau^{jc}_{r}\}_{r}, \{\tau^{(j+1)c}_{r}\}_{r},..., \{\tau^{(j+C)c}_{r}\}_{r}$$ with each one beginning from the terminal point of the previous one, at least one of the sequences of quasi-geodesic segments $\{\tau^{(j+m)c}_{r}\}_{r}$ in the list has divergence which is superlinear in $r.$
\end{lem} 
 
\begin{proof}
To simplify the exposition we will prove the first case.  The proof of the similar statement follows almost identically.  Fix a sequence of detour paths $\alpha_{r}$ and corresponding quasi-geodesics $\tau^{jc}_{r}.$  By assumption, for any fixed $j$ the divergence of the sequence $\tau^{jc}_{r}$ is superlinear, say $r \epsilon_{j}(r)$ where $\lim_{r} \epsilon_{j}(r) \ra \infty.$  We will prove the lemma by contradiction.  That is, assume there is a constant $N$ such that $\lim_{r} |\alpha_{r}|< r^{2}N.$  Since $\gamma$ is contracting, as in Lemma \ref{lem:scatleastquad} we have:
\begin{eqnarray*}
|\alpha_{r}|  \geq \sum_{j=1}^{\lfloor \frac{r}{c} \rfloor} d_{X \setminus B_{r}(\gamma(0)}(z_{r}^{(j-1)c},z_{r}^{jc}) \geq \sum_{j=1}^{\lfloor \frac{r}{c} \rfloor}  r \epsilon_{j}(r) \geq \frac{r^{2}}{c} \min^{r}_{j=1}(\epsilon_{j}(r))
\end{eqnarray*}
Putting things together, it follows that 
$$ \lim_{r} \min^{r}_{j=1}(\epsilon_{j}(r)) $$ is uniformly bounded.  (In the situation of Example \ref{ex:degenerate} the uniform bound was one). Set $\min^{r}_{j=1}(\epsilon_{j}(r)) = \epsilon_{j_{min}}(r).$  Then for all values of $r\in \N$ we can use the periodicity of $\gamma$ to translate the points $z^{j_{min}c}_{r}$ to points $z^{0}_{r},$ and correspondingly the quasi-geodesics $\tau^{j_{min}c}_{r}$ to quasi-geodesics $\tau^{0}_{r}.$  After translation, we have a sequence of quasi-geodesic segments $\{\tau^{0}_{r}\}_{r}$ with linear divergence.  This is a contradiction to the hypotheses of the theorem and hence completes the proof by contradiction.
\end{proof} 

Next, consider the following Lemma of \cite{rafischleimer}, which we will use in the construction of a quasi-geodesic with superquadratic divergence in $\mathcal{P}(S_{2,1}):$
\begin{lem} [\cite{rafischleimer} Theorem 2.1] \label{lem:embedding} For any surface $S_{g,n}$ there exists an isometric embedding $i: \mathcal{C}(S_{g,n}) \ra \mathcal{C}(S_{g,n+1})$ such that $\pi \circ i$ is the identity map, where $\pi\co\mathcal{C}(S_{g,n+1}) \ra \mathcal{C}(S_{g,n})$ is given by forgetting about the puncture.
\end{lem}

%

Fix $\bar{\alpha}_{0} \in \mathcal{C}_{sep}(S_{2,0}),$ and let $\bar{f}$ be a pseudo-Anosov axis in $\mathcal{C}(S_{2,0})$ containing the curve $\bar{\alpha}_{0}.$  Furthermore, assume that $|\bar{f}(\bar{\alpha}_{0}) \cap \bar{\alpha}_{0}| =4.$  See Figure \ref{fig:superlinear2} for an example.

\begin{figure}[htpb]
\centering
\includegraphics[height=5 cm]{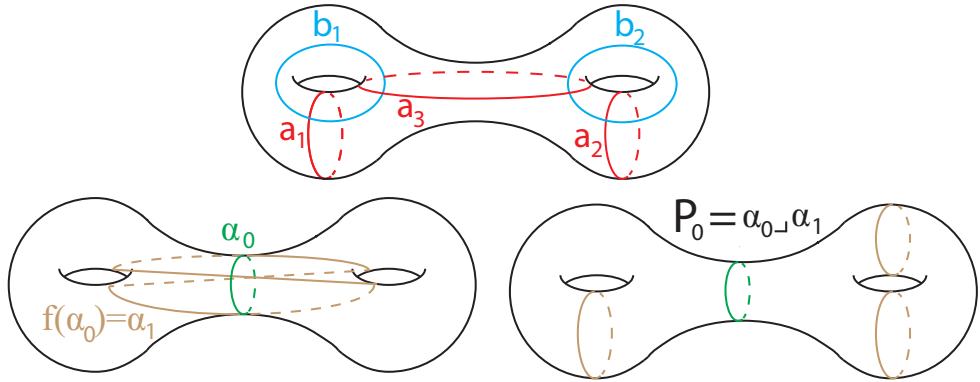}
\caption[A point pushing pseudo-Anosov map.]{$\bar{f}=T_{a_{3}} T^{-1}_{b_{2}}T^{-1}_{b_{1}} T_{a_{2}}T_{a_{1}}$ is a pseudo-Anosov mapping class.  Note that in the lower left figure $|\bar{f}(\bar{\alpha}_{0}) \cap \bar{\alpha}_{0}| =4.$  Moreover, in the lower right figure note that since $\bar{\alpha}_{0}$ and $\bar{f}(\bar{\alpha}_{0})$ are different separating curves, by topological considerations $\bar{\alpha}_{0} \lrcorner \bar{f}(\bar{\alpha}_{0})$ is a pants decomposition.  }\label{fig:superlinear2}
\end{figure}

Denote the separating curve $\bar{f}^{i}(\bar{\alpha}_{0})$ by $\bar{\alpha}_{i}$ for all $ i \in \Z.$  Since $\forall i \ne j,$ $\bar{\alpha}_{i},\bar{\alpha}_{j}$ are in different separating curves of $\mathcal{C}_{sep}(S_{2,0}),$ by topological considerations it follows that $\bar{\alpha}_{i} \lrcorner \bar{\alpha}_{j}$ can be coarsely identified with a pants decomposition of $S_{2,0}.$  In particular, for all $i \in \Z,$ let $\bar{P}_{i}$ denote a fixed pants decomposition of the form $\bar{\alpha}_{i} \lrcorner \bar{\alpha}_{i+1}.$  Let $\bar{\gamma}_{n}$ denote a piecewise geodesic path in the pants complex traveling through the pairs of pants $\bar{P}_{-n}, ...,\bar{P}_{0}, ..., \bar{P}_{n}.$  Moreover, let $\bar{\gamma}$ denote the limit of the paths $\bar{\gamma}_{n}.$  Note that we can assume $\bar{f}^{i}(\bar{P}_{j})=\bar{P}_{i+j},$ and hence $\bar{f}$ acts by translations on the path $\bar{\gamma}.$  It follows that $\bar{\gamma} \in \mathcal{P}(S_{2,0})$ is a contracting (hence Morse) quasi-geodesic as it is the axis of a pseudo-Anosov mapping class \cite{behrstock,bmm2}.  In particular, it follows that in every asymptotic cone $P_{\omega}(S_{2,0})$ any two distinct points on $\bar{\gamma}_{\omega}$ are separated by a cut-point, see for instance  \cite{drutumozessapir} Proposition 3.24.  In particular, it follows that any region of the form $\qbar{\alpha} \subset $ has a unique nearest point on $\bar{\gamma}_{\omega}$ whose removal separates the region $\qbar{\alpha}$ from the two resulting components of $\bar{\gamma}_{\omega}.$  Furthermore, by construction in every asymptotic cone $P_{\omega}(S_{2,0})$ any two distinct points on $\bar{\gamma}_{\omega}$ are not contained in the ultralimit of a natural product region of the form $\qbar{\alpha}$ for any $\bar{\alpha} \in \mathcal{C}^{\omega}_{sep}(S_{2,0}).$ 

Using the isometric embedding $i \co \mathcal{C}(S_{2,0}) \ra \mathcal{C}(S_{2,1})$ of Lemma \ref{lem:embedding}, we can lift all the aforementioned structure from $S_{2,0}$ to $S_{2,0}.$  Specifically, we can lift the separating curves $\bar{\alpha}_{i}$ to separating curves $\alpha_{i} \in \mathcal{C}_{sep}(S_{2,1}),$ the pants decompositions $\bar{P}_{i}$ to pants decompositions $P_{i} \in \mathcal{P}(S_{2,1}),$ and the periodic quasi-geodesic $\bar{\gamma}$ with bounded combinatorics to a periodic geodesic $\gamma \subset \mathcal{P}(S_{2,1})$ which also has bounded combinatorics as it too is the axis of a pseudo-Anosov map $f$ which is a lift of $\bar{f}.$  It follows that in every asymptotic cone $P_{\omega}(S_{2,1})$ any two distinct points on $\gamma_{\omega}$ are similarly separated by a cutpoint and are not contained in the ultralimit of a common subset of the form $\xbar{\alpha}$ for $\bar{\alpha}$ any $\bar{\alpha} \in \mathcal{C}^{\omega}_{sep}(S_{2,0}).$ 
 
Presently we will prove Theorem \ref{thm:atleastsuperquad} by showing that this periodic and contracting quasi-geodesic $\gamma \subset \mathcal{P}(S_{2,1})$ has superquadratic divergence.  
     
%
%
\begin{figure}[htpb]
\centering
\includegraphics[height=5.5 cm]{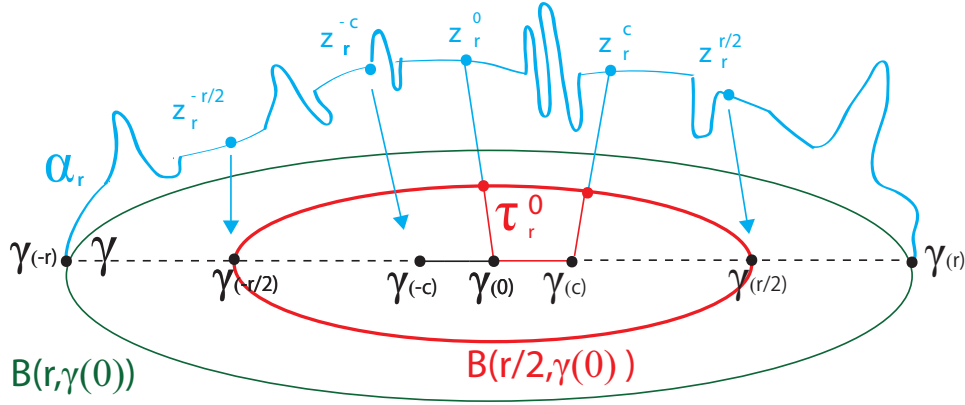}
\caption[$\mathcal{T}(S_{2,1})$ has superquadratic divergence.]{The detour path $\alpha_{r}$ connects $\gamma_{-r}$ while avoiding $B(r,\gamma(0)).$  The points $z_{r}^{jc} \in \alpha_{r}$ project to $\gamma(jc)$ under the nearest point projection onto $\gamma.$  By Lemma \ref{lem:quasigeo},  $\tau^{jc}_{r} = [z_{r}^{jc},\gamma(jc)]\cup [\gamma(jc),\gamma((j+1)c)] \cup [\gamma((j+1)c),z_{r}^{(j+1)c}]$ is a quasi-geodesic with uniform constants.  Moreover, the sequence of quasi-geodesics $\{\tau^{j}_{r}\}_{r}$ almost always has superlinear divergence.}\label{fig:superquadratic2}
\end{figure}

\begin{proof} [Proof of Theorem \ref{thm:atleastsuperquad}]
In light of Lemma \ref{lem:periodic} in order to prove the theorem it suffices to show that the above constructed periodic and contracting quasi-geodesic $\gamma \subset \mathcal{P}(S_{2,1})$ satisfies the hypothesis of Lemma \ref{lem:periodic}.  Assume $\gamma$ does not satisfy the hypothesis of Lemma \ref{lem:periodic}.  Specifically, for any positive integer $k$ there exists some consecutive sequence of sequences of quasi-geodesic segments
$$\{ \tau_{r}^{j_{k}c}\}_{r}, \{\tau_{r}^{(j_{k}+1)c}\}_{r}, .., \{\tau_{r}^{(j_{k}+k)c}\}_{r}$$ each one beginning from the terminal point of the previous one, such that for each fixed $m \in \{0,...,k\},$ the sequence of quasi-geodesic segments $\{\tau_{r}^{(j_{k}+m)c}\}_{r}$ in the list has divergence linear in $r,$ the natural numbers.   

  

Since the sequence of geodesics $[z_{r}^{(j_{k}+m)c},\gamma((j_{k}+m)c)]$ are contained as subsegments of $\tau_{r}^{(j_{k}+m)c}$ with roughly half the total length, and because $\tau_{r}^{(j_{k}+m)c}$ have linear divergence, it follows that $[z_{r}^{(j_{k}+m)c},\gamma((j_{k}+m)c)]$ also have linear divergence.  By Lemma \ref{lem:superlinear}, in the asymptotic cone $Cone_{\omega}(\mathcal{P}(S_{2,1}),\gamma(j_{k}c),(rc))$ the ultralimit of $[z_{r}^{(j_{k}+m)c},\gamma((j_{k}+m)c)]$ is nontrivial and does not have any cut-points for all $m \in \{0,...,k\}.$  By Theorem \ref{thm:zeropieces} it follows that the ultralimits of the form $[z_{r}^{(j_{k}+m)c},\gamma((j_{k}+m)c)]$ are completely contained in subsets of the form $\xbar{\alpha}$ for some unique $\overline{\alpha}$ an element of $\mathcal{C}_{sep}(S_{2,0})^{\omega}.$  

Consider the sequence of geodesic quadrilaterals with vertices given by $$\{z_{r}^{(j_{k}+m)c},z_{r}^{(j_{k}+m+1)c}, \gamma((j_{k}+m+1)c), \gamma((j_{k}+m)c)\}.$$  The sequence of edges $[\gamma((j_{k}+m)c),\gamma((j_{k}+m+1)c)]$ have bounded (constant) length.  On the other hand, the three remaining sequence of edges all have lengths growing linearly in $r$ and have linear divergence.  As in Theorem \ref{thm:zeropieces} it follows that in the same asymptotic cone $$Cone_{\omega}(\mathcal{P}(S_{2,1}),\gamma(j_{k}c),(rc)),$$ the ultralimits of the sequences of quadrilaterals and in particular the edges of them $$[z_{r}^{(j_{k}+m)c},\gamma((j_{k}+m)c)],[z_{r}^{(j_{k}+m+1)c},\gamma((j_{k}+m+1)c)]$$ are completely contained in a common subset of the form $\xbar{\alpha}.$  Repeating this argument and using the fact that adjacent pairs of ultralimits of quadrilaterals have nontrivial intersection in the asymptotic cone, by Corollary \ref{cor:intpieces} it follows that the consecutive string of sequences $$[z_{r}^{j_{k}c},\gamma(j_{k}c)],...,[z_{r}^{(j_{k}+k)c},\gamma((j_{k}+k)c)]$$ have ultralimits in the asymptotic cone $Cone_{\omega}(\mathcal{P}(S_{2,1}),\gamma(j_{k}c),(rc)),$ completely contained in a common subset of the form $\xbar{\alpha}.$ 
 
Now consider the asymptotic cone $Cone_{\omega}(\mathcal{P}(S_{2,1}),\gamma(j_{3r}),(rc)).$  In particular, consider the distinct points in the asymptotic cone with representative sequences $$\{z^{j_{3r}c}_{rc}\}_{r} \mbox{and} \{z^{(j_{3r}+3r)c}_{rc}\}_{r}.$$  By the argument above, in conjunction with appropriate translations along $\gamma,$ we have seen that these two points in the cone have representative sequences that identify them as being contained in a common subset of the form $\xbar{\alpha}.$  Furthermore, the points $$\{z^{(j_{3r}c)}_{rc}\}_{r} \mbox{, and } \{z^{((j_{3r}+3r)c)}_{rc}\}_{r}$$ are of distance at most (in fact exactly) one from the distinct points with representatives $$\{\gamma(j_{3r}c)\}_{r}\mbox{, and }\{\gamma((j_{3r}+3r)c)\}_{r} $$ on the ultralimit $\gamma_{\omega},$ respectively.  Projecting this situation from $S_{2,1}$ to $S_{2,0},$ we obtain points $\{\bar{z}^{(j_{3r}c)}_{rc}\}_{r}, \{\bar{z}^{((j_{3r}+3r)c)}_{rc}\}_{r}$ which are of distance at most one (the projection is 1-Lipschitz) from the distinct points with representatives $\{\bar{\gamma}(j_{3r}c)\}_{r},\{\bar{\gamma}((j_{3r}+3r)c)\}_{r} $ on the ultralimit $\bar{\gamma}_{\omega},$ respectively.  On the other hand, the points $\{\bar{z}^{(j_{3r}c)}_{rc}\}_{r}, \{\bar{z}^{((j_{3r}+3r)c)}_{rc}\}_{r}$ are in a common subset of the form $\qbar{\alpha}.$  It follows that there is a path $\rho_{\omega}$ connecting the points  $\{\bar{\gamma}(j_{3r}c)\}_{r},\{\bar{\gamma}((j_{3r}+3r)c)\}_{r} $ which travels for distance at most two (namely $\{[\bar{\gamma}(j_{3r}c), \bar{z}^{(j_{3r}c)}_{rc}]\}_{r}$ and $\{[\bar{\gamma}((j_{3r}+3r)c), \bar{z}^{(j_{3r}+3r)c}_{rc}]\}_{r}$ each of which has length at most one) outside of the region $\qbar{\alpha}.$  However, since in the asymptotic cone $Cone_{\omega}(\mathcal{P}(S_{2,1}),\gamma(j_{3r}c),(rc)),$ the points $\{\bar{\gamma}(j_{3r}c)\}_{r},$ and $\{\bar{\gamma}((j_{3r}+3r)c)\}_{r} $ are distance three apart and because any region of the form $\qbar{\alpha}$ has a unique nearest point on $\bar{\gamma}_{\omega}$ whose removal separates the region $\qbar{\alpha}$ from the two resulting components of $\bar{\gamma}_{\omega},$ this is a contradiction, thus completing the proof.
\end{proof}

\subsection{$\mathcal{T}(S_{2,1})$ has at most cubic divergence}
\label{subsec:atmostcubic} 
There is a strong relationship between the divergence of a metric space and its thickness.  Preliminarily, as a consequence of Lemma \ref{lem:superlinear} it follows that a geodesic metric space is thick of order zero if and only if the divergence of the space is linear.  More generally, considering the inductive nature of the definition of degree of thickness of a space, a natural conjecture is that the polynomial order of divergence of a sufficiently nice metric space - such as the pants complex - is equal to one plus the degree of thickness of the space, \cite{behrstockdrutu}.  Presently we record a theorem providing partial progress toward this conjecture.

\begin{thm} [\cite{behrstockdrutu} Corollary 4.17] \label{thm:atmost}
Let $X$ be a geodesic metric space which is thick of order $n,$ then the $div(X)$ is at most polynomial of order $n+1.$  
\end{thm}  
 In particular, combining Theorems \ref{thm:thick2} and \ref{thm:atmost}, we have:
 \begin{cor} \label{cor:atmost} $\mathcal{T}(S_{2,1})$ has at most cubic divergence.  
\end{cor}

\begin{rem} In light of Theorem \ref{thm:atleastsuperquad}, Theorem \ref{thm:atmost} provides an alternative proof of Theorem \ref{thm:thick2}, namely that $\mathcal{T}(S_{2,1})$ is thick of order two. 
\end{rem}

\subsection{Divergence of Teichm\"uller spaces}
Just as the proof of Theorem \ref{thm:thick2} uniquely characterizes $\mathcal{T}(S_{2,1})$ among all Teichm\"uller spaces and completes the thickness classification of Teichm\"uller spaces, so too Theorem \ref{thm:atleastsuperquad} and Corollary \ref{cor:atmost} also uniquely characterize $\mathcal{T}(S_{2,1})$ among all Teichm\"uller spaces and (almost) complete the divergence classification of all Teichm\"uller spaces.  See Table \ref{table:div}.  

Notice that the Teichm\"uller spaces of low complexity surfaces that are either hyperbolic or relatively hyperbolic, perforce have at least exponential divergence.  It is immediate by observation that for complexity one surfaces the pants complex, or equivalently the Farey graph, has infinitely many ends.  On the other hand, it follows from recent work of \cite{gabai,rafischleimerrigid} that for complexity at least two surfaces, the pants complex is one ended and hence the divergence is in fact exponential.  Specifically, building off of work of Gabai in \cite{gabai}. Rafi-Schleimer in Proposition 4.1 of \cite{rafischleimerrigid} show that the curve complex is one ended for complexity at least two surfaces.  In particular, it follows that the same result holds for the corresponding pants complexes.


\begin{table}[!htb] 
\begin{center}
\begin{tabular}{|c||c|c| m{3cm}|c|c}
 $\vdots$ & \cellcolor[rgb]{0,1,1}$\vdots$& \cellcolor[rgb]{0,1,1}$\vdots$ & \cellcolor[rgb]{0,1,1} \center $\vdots$ &  \cellcolor[rgb]{0,1,1}$\vdots$ &\cellcolor[rgb]{0,1,1} $\iddots$  \\
\hline
$7$ &  \cellcolor[rgb]{0,1,1} quadratic& \cellcolor[rgb]{0,1,1} quadratic& \cellcolor[rgb]{0,1,1} \center quadratic & \cellcolor[rgb]{0,1,1} quadratic& \cellcolor[rgb]{0,1,1}  $\hdots$  \\
\hline
$6$ & \cellcolor[rgb]{1,1,0}  exponential &   \cellcolor[rgb]{0,1,1}   quadratic & \cellcolor[rgb]{0,1,1}  \center quadratic& \cellcolor[rgb]{0,1,1} quadratic &  \cellcolor[rgb]{0,1,1} $\hdots$ \\ 
\hline 
$5$ & \cellcolor[rgb]{1,1,0}  exponential &     \cellcolor[rgb]{0,1,1} quadratic& \cellcolor[rgb]{0,1,1}  \center quadratic& \cellcolor[rgb]{0,1,1} quadratic & \cellcolor[rgb]{0,1,1} $\hdots$   \\ 
\hline
$4$ &  \cellcolor[rgb]{1,.5,0} infinite & \cellcolor[rgb]{0,1,1} quadratic    & \cellcolor[rgb]{0,1,1}  \center quadratic & \cellcolor[rgb]{0,1,1} quadratic  & \cellcolor[rgb]{0,1,1} $\hdots$   \\
\hline
$3$ &    & \cellcolor[rgb]{1,1,0}  exponential & \cellcolor[rgb]{0,1,1} \center  quadratic & \cellcolor[rgb]{0,1,1} quadratic&  \cellcolor[rgb]{0,1,1} $\hdots$  \\
\hline
$2$ &    & \cellcolor[rgb]{1,1,0}  exponential & \cellcolor[rgb]{0,1,1}  \center quadratic   &\cellcolor[rgb]{0,1,1} quadratic&  \cellcolor[rgb]{0,1,1} $\hdots$    \\
\hline
$1$ &    &  \cellcolor[rgb]{1,.5,0} infinite &  \cellcolor[rgb]{.5,1,0} \center superquadratic yet  at most cubic & \cellcolor[rgb]{0,1,1} quadratic& \cellcolor[rgb]{0,1,1} $\hdots$  \\ 
\hline
$0$ &    &     & \cellcolor[rgb]{1,1,0} \center exponential  & \cellcolor[rgb]{0,1,1} quadratic& \cellcolor[rgb]{0,1,1} $\hdots$  \\ 
\hline
\hline
$n \uparrow \; g \rightarrow  $  & $0$  & $1$ &\center $2$  & $3$ & $\hdots$ \\  
\hline   
\end{tabular}
\end{center}
\caption[Divergence of Teichm\"uller spaces]{Divergence of Teichm\"uller spaces for all surfaces of finite type. }
\label{table:div}
\end{table}

\newpage
\bibliographystyle{amsplain}
\bibliography{../bib}

\end{document}